\documentclass[a4paper,10pt,twoside,reqno]{amsart}
\usepackage[utf8]{inputenc}
\usepackage{amsthm,amsfonts,amssymb,amsxtra,appendix,bookmark,euscript,delarray,dsfont,mathrsfs,stackrel,xifthen,mathtools}
\usepackage[normalem]{ulem}
\usepackage{paralist}
\usepackage{enumitem}
\usepackage[dvipsnames]{xcolor}
\usepackage{latexsym,amssymb}
\usepackage{mathrsfs}
\usepackage{longtable}
\usepackage{lmodern}
\usepackage{microtype}
\usepackage{amsmath}
\usepackage{extarrows}
\usepackage{comment}
\usepackage{mhequ}
\usepackage{caption}
\usepackage{subcaption}
\usepackage{bbm}

\usepackage{amsmath,amssymb,amsthm,hyperref,mathtools,esint,enumitem}
\usepackage{empheq}
\hypersetup{linktoc=all,colorlinks=true,linkcolor=blue,citecolor=blue,urlcolor=blue}

\DeclareMathAlphabet{\mathlcal}{U}{dutchcal}{m}{n} 

\DeclarePairedDelimiter{\abs}{\lvert}{\rvert}
\DeclarePairedDelimiter{\norm}{\lVert}{\rVert}
\DeclarePairedDelimiter{\bra}{(}{)}
\DeclarePairedDelimiter{\pra}{[}{]}

\providecommand\given{}
\DeclarePairedDelimiterXPP\condprob[1]{\mathbb{P}}[]{}{
\renewcommand\given{\nonscript\:\delimsize\vert\nonscript\:\mathopen{}}
  #1}
\DeclarePairedDelimiterXPP\condexpect[1]{\mathbb{E}}[]{}{
\renewcommand\given{\nonscript\:\delimsize\vert\nonscript\:\mathopen{}}
#1}
\DeclarePairedDelimiterXPP\condexpectwithinitial[2]{\mathbb{E}_{#1}}[]{}{
\renewcommand\given{\nonscript\:\delimsize\vert\nonscript\:\mathopen{}}
#2}
\providecommand\st{}
\DeclarePairedDelimiterXPP\set[1]{}\{\}{}{
\renewcommand\st{\nonscript\:\delimsize\vert\nonscript\:\mathopen{}}
  #1}

\usepackage{tikz-cd}
\usepackage{tikz}
\usepackage{cases}
\usepackage{amsmath}

\makeatletter
\@namedef{subjclassname@2020}{\textup{2020} Mathematics Subject Classification}
\makeatother
\usepackage{amsthm}
\usepackage{thmtools}
\usepackage[capitalize, nameinlink]{cleveref}

\theoremstyle{plain}
\newtheorem{theorem}{Theorem}[section]

\newtheorem*{theorem*}{Theorem}
\newtheorem{lemma}[theorem]{Lemma}
\newtheorem*{lemma*}{Lemma}
\newtheorem{corollary}[theorem]{Corollary}
\newtheorem{proposition}[theorem]{Proposition}

\newtheorem*{conjecture*}{Conjecture}
\newtheorem{assumption}[theorem]{Assumption}

\theoremstyle{definition}

\newtheorem{defn}[theorem]{Definition}

\theoremstyle{remark}
\newtheorem{remark}[theorem]{Remark}
\theoremstyle{remark}
\theoremstyle{definition}

\newtheorem*{nota*}{Notation}

\newlist{proofparts}{enumerate}{2}
\setlist[proofparts]{
  listparindent=\parindent,
  leftmargin=0pt,
  itemindent=*,
  parsep=0pt plus 1pt,
  }
\setlist[proofparts,1]{label=\textbf{Part \arabic*.}}
\setlist[proofparts,2]{label=\textbf{Part \arabic{proofpartsi}.\arabic*.}}

\usepackage{autonum} 

\numberwithin{equation}{section}

\renewcommand{\P}{\mathbf{P}}

\newcommand{\T}{\mathbb{T}}
\renewcommand{\d}{{\rm d}}

\newcommand{\eps}{\varepsilon}

\newcommand{\N}{\mathbb{N}}

\newcommand{\R}{\mathbb{R}}

\newcommand{\Z}{\mathbb{Z}}

\newcommand{\cF}{\mathcal{F}}

\newcommand{\cP}{\mathcal{P}}

\newcommand{\expect}{\mathbf{E}}

\newcommand{\dummy}{\textcolor{lightgray}{{\bullet}}}

\DeclareMathAlphabet{\mathup}{OT1}{\familydefault}{m}{n}
\DeclareMathOperator{\supp}{supp}

\DeclareMathOperator*{\argmax}{arg\,max}

\DeclareMathAlphabet{\mathup}{OT1}{\familydefault}{m}{n}
\newcommand{\dx}[1]{\mathop{}\!\mathup{d} #1}

\makeatletter
\newcommand{\customlabel}[2]{
   \protected@write \@auxout {}{\string \newlabel {#1}{{#2}{\thepage}{#2}{#1}{}} }
   \hypertarget{#1}{}
}
\makeatother

\numberwithin{figure}{section}
\makeatletter
\AddToHook{cmd/appendix/before}{\def\cref@section@alias{appendix}}
\makeatother

 \def\calE{{\mathcal E}} \def\calF{{\mathcal F}}

 \def\calN{{\mathcal N}} 
\def\calP{{\mathcal P}}  
  
 \def\calW{{\mathcal W}}

\input{widebar.tex}

\newcommand{\tcluster}{t_{\rm clustering}}
\newcommand{\tmetastab}{t_{\rm metastab}}
\newcommand{\tcoalesc}{t_{\rm coalesence}}
\newcommand{\tmicrorev}{t_{\rm reversibility}}
\newcommand{\mcrit}{m_{\rm crit}}

\RequirePackage{etex}
\usepackage{etoolbox} 
\expandafter\patchcmd\csname\string\proof\endcsname
  {\normalparindent}{0pt }{}{}

\crefname{enumi}{}{} 
\crefname{equation}{}{}

\cslet{blx@noerroretextools}\empty 
\usepackage[url=true,backend=biber,style=alphabetic,url=false,doi=false,isbn=false,maxnames=99,giveninits=true,maxalphanames=4,date=year]{biblatex}
\DeclareFieldFormat{titlecase}{\MakeSentenceCase*{#1}}
\renewbibmacro*{journal}{
  \iffieldundef{journaltitle}
    {}
    {\printtext[journaltitle]{
       \printfield[noformat]{journaltitle}
       \setunit{\subtitlepunct}
       \printfield[noformat]{journalsubtitle}}}}
\renewbibmacro{in:}{
	\ifentrytype{article}{}{\printtext{\bibstring{in}\intitlepunct}}}

\newcommand{\rhounif}{\rho_{\rm unif}} 
\newcommand{\rhopert}[1][]{{#1\rho}_{\rm pert}} 

\newcommand{\revfirst}[1]{{\color[RGB]{200,10,10,}#1}}

\title[]{Formation of clusters and coarsening in weakly interacting diffusions}
\author[Gerber]{N. J. Gerber$^{\ast}$}
\thanks{$^{\ast}$Corresponding author, \href{mailto:nicolai.gerber@uni-ulm.de}{nicolai.gerber@uni-ulm.de}}
\address{N. J. Gerber -- Institute for Applied Analysis, Ulm University, Helmholtzstrasse 18, D-89081 Ulm, Germany}
\email{nicolai.gerber@uni-ulm.de}
\author[Gvalani]{R. S. Gvalani}
\address{R. S. Gvalani -- D-MATH, ETH Z\"urich, R\"amistra{\ss}e 101, 8092 Z\"urich}
\email{rgvalani@ethz.ch}
\author[Hairer]{M. Hairer}
\address{M. Hairer -- EPFL, Switzerland and Imperial College London, UK}
\email{martin.hairer@epfl.ch}
\author[Pavliotis]{G. A. Pavliotis}
\address{G. A. Pavliotis -- Department of Mathematics, Imperial College London, London SW7~2AZ, UK}
\email{pavl@ic.ac.uk}
\author[Schlichting]{A. Schlichting}
\address{A. Schlichting -- Institute for Applied Analysis, Ulm University, Helmholtzstrasse~18, D-89081 Ulm, Germany}
\email{andre.schlichting@uni-ulm.de}
\thanks{The research of RSG was partially supported by the Deutsche Forschungsgemeinschaft (DFG, German Research Foundation) under the Special Priority Programme 2410/1 “Hyperbolic Balance Laws in Fluid Mechanics: Complexity, Scales, Randomness”. GAP is partially supported by an ERC-EPSRC Frontier Research Guarantee through Grant No. EP/X038645, ERC Advanced Grant No. 247031 and a Leverhulme Trust Senior Research Fellowship, SRF$\backslash$R1$\backslash$241055. AS is supported by the Deutsche Forschungsgemeinschaft (DFG, German Research Foundation) under Germany's Excellence Strategy EXC 2044 --390685587, Mathematics M\"unster: Dynamics--Geometry--Structure.}
\date{\today}

\subjclass[2020]{} 
\allowdisplaybreaks
\begin{document}

\begin{abstract}
    This paper studies the clustering behaviour of weakly interacting diffusions under the influence of sufficiently localized attractive interaction potentials on the one-dimensional torus. We describe how this clustering behaviour is closely related to the presence of discontinuous phase transitions in the mean-field PDE. For local attractive interactions, we employ a new variant of the strict Riesz rearrangement inequality to prove that all global minimisers of the free energy are either uniform or single-cluster states, in the sense that they are symmetrically decreasing.

    We analyze different timescales for the particle system and the mean-field (McKean--Vlasov) PDE, arguing that while the particle system can exhibit coarsening by both coalescence and diffusive mass exchange between clusters, the clusters in the mean-field PDE are unable to move and coarsening occurs via the mass exchange of clusters. By introducing a new model for this mass exchange, we argue that the PDE exhibits dynamical metastability. 
    We conclude by presenting careful numerical experiments that demonstrate the validity of our model.

\end{abstract}

\maketitle
\setcounter{tocdepth}{2}

\section{Introduction}

Interacting particle systems and their mean-field limits appear in many applications ranging from physics to mathematical models for swarming, flocking, opinion dynamics, and collective behaviour. They are also used in the development of algorithms for sampling and optimization. Often, the emergence of collective behaviour, e.g.~consensus formation~\cite{HegselmannKrause2002} or synchronization~\cite{kuramoto1981rhythms}, can be interpreted as a (first- or second-order) phase transition from a disordered to an ordered state; see~\cite{CGPS20} and the references therein, in particular~\cite{CP10}.

Noise-driven interacting particle systems can exhibit metastable states that persist over long timescales. As examples, we mention the formation of clusters in opinion models~\cite{GarnierEtAl2017} as well as in mathematical models for transformers~\cite{GeshkovskiLetrouitPolyanskiyRigollet2025,GeshkovskiKoubbiPolyanskiyRigollet2024,bruno2024emergence}; see also~\cite[Section 6.1]{ShalovaSchlichting2025}. Noise-induced metastability in interacting particle systems is a manifestation of the presence of discontinuous phase transitions~\cite{PLMKT_1984,GvalaniSchlichting2020}. It is expected that, near a transition point, the mean-field description is no longer valid for arbitrarily long times~\cite{DGPS23} and that the effect of noise, through the Dean--Kawasaki SPDE~\cite{Kawasaki1973,Dean1996}, must be taken into account to capture metastable states accurately~\cite{WSMSPW_2025}. The primary objective of this work is to provide a rigorous mathematical understanding of the emergence of clusters in weakly interacting diffusions exhibiting first-order phase transitions and to derive a model for the dynamics of these clusters. 

In~\cite{PoschNarenhoferThirring1990}, the authors justify that the mean-field description for Hamiltonian systems is inadequate to describe all regimes in the Newtonian dynamics of particles interacting via a pairwise attractive Gaussian potential.
They observed that for sufficiently strong interactions, the initially diffused particles form clusters that eventually collapse to form a single cluster (or, as the authors of~\cite{PoschNarenhoferThirring1990} refer to it, the condensed phase). Although the mean-field Vlasov equation describes the final cluster state, which is a stationary state of the mean-field PDE, it is unable to explain the transient dynamics that leads to the formation of the condensed phase. We quote this main conclusion of theirs below:
\begin{quote}
    \emph{We find that the condensed phase is a stationary solution of the Vlasov equation, but the Vlasov dynamics cannot describe the collapse.}
\end{quote}

The goal of this paper is to study a similar phenomenology in the setting of weakly interacting diffusions with localised attractive interaction potentials.
The thermodynamic limit in our setting is the so-called McKean--Vlasov equation, and our main conclusion reads as follows:
\begin{quote}
    \emph{Although the single-cluster state is a stationary solution of the mean-field McKean-Vlasov PDE, coalescence of clusters is not captured by the mean-field dynamics.
    }
\end{quote}
It turns out that the weakly interacting diffusion model studied in this paper exhibits very rich and interesting dynamics.
While the interacting particle model in~\cite{PoschNarenhoferThirring1990} is deterministic and the entropy in the corresponding mean-field Vlasov equation is conserved, the model of interacting diffusions considered in this paper implies that the corresponding mean-field equation is a diffusion-aggregation equation that exhibits coarsening through mass exchange between clusters. This mass exchange leads to \emph{dynamical metastability} of the mean-field dynamics, which is related to the notion of slow motion of gradient flows studied in~\cite{Otto2007}.

We comment on some more general relations to similar effects observed in related models. First, we note that our identification of the two timescales for tunneling and coalescence is similar to the classical theory of \emph{Ostwald ripening}:  a configuration of near-equilibrium liquid droplets sitting on a precursor film that wets the entire substrate can coarsen in time by two different mechanisms: collapse or collision of droplets~\cite{GlasnerOttoRumpSlepcev2009}. Here, the phenomenological model is a gradient flow of Cahn--Hilliard type, where the dynamics is entirely deterministic. The authors provide a reduced model for the position and volume of single droplets in~\cite[Section~4]{GlasnerOttoRumpSlepcev2009}, which also allows them to study the crossover between the timescales of migration (leading to collapse) and ripening (through coalescence).
The relevant timescales can be identified using suitable energies, in this case related to surface tension, and it is possible to derive a differential inequality for them. This approach was pioneered in~\cite{KohnOtto2002} and later refined in~\cite{ContiNiethammerOtto2006,OttoRumpSlepcev2006,BrenierOttoSeis2011,OttoSeisSlepcev2013}.

In the physical literature, a study of diffusive aggregating particles can be found in~\cite{PototskyThieleArcher2014}, where the observed behaviour is studied on the PDE level. There, Ostwald ripening is related to volume modes induced by the escape of particles from clusters due to random fluctuations. However, a crucial difference from our study is that the translational modes causing movement and joining of clusters are again deterministically caused by long-range aggregation forces between the clusters, whereas in our model, the interaction potential has finite interaction range and clusters can move only due to Brownian motions.

We also comment on the link between the problem studied in our paper and Kramers' problem and its generalizations. Indeed, the mass leakage between clusters that occurs in the model that we derive in this paper is a generalization of Kramers' problem pioneered in~\cite{Kramers1940}, where a Brownian particle is confined in a double-well potential and the timescale is related to the energy barrier of the system. Our generalization requires considering potentials with multiple timescales, which in addition change dynamically as a result of the mass exchange.  In recent years, several different techniques for studying metastability in time-independent potentials have been developed, using both stochastic~\cite{mathieu1994zero,olivieri2005large,Berglund2013,MS2014,BH_2015} and analytic~\cite{Peletier2010,Herrmann2011,EvansTabrizian2016,Herrmann2019} methods. Our setting can be viewed as a combination of these, together with the analysis of Langevin-type dynamics for time-dependent potentials in the context of Brownian ratchets, e.g.,~\cite {Kinderlehrer2002}.

Finally, we briefly discuss the relevance of the Dean--Kawasaki SPDE for the study of the dynamics of clusters. This SPDE was used to study the statistics of the clusters~\cite{WSMSPW_2025}; see also~\cite{delfau2016pattern} for some earlier numerical experiments. The Dean--Kawasaki SPDE can be used, in principle, to study noise-induced metastability of our model by applying an infinite-dimensional extension of the Freidlin--Wentzell theory~\cite{freidlin1984random}, developed by Dawson and Gärtner~\cite{DG1989}. However, for the analysis of the relevant timescales in the models studied in this paper, we need to understand the combined effect of noise-induced cluster coalescence and of dynamical metastability due to slow motion for gradient flows~\cite{Otto2007}. Understanding the relationship between the different timescales induced by the different stochastic and deterministic effects is one of the main challenges that we address in this work.

\subsection{Setup of the model}
\label{sec:setup}

We consider the interacting particle system given by the system of $N$ stochastic differential equations
\begin{equation}\label{eq:microscopic-system}
    \d{X}_t^i = - \frac{1}{N}\sum_{j=1}^N \nabla W_{\gamma,\ell}\bra[\big]{{X_t^i-X_t^j}} \dx t + \sqrt{2} \dx B_t^i , \qquad\text{for } i =1,\dots ,N
\end{equation}
on the one-dimensional torus $\T \simeq [0,1)$. Here, $\gamma>0$ is the interaction strength, $\ell>0$ is the interaction range and $(B_t^i)_{i\in\N}$ are independent standard Brownian motions in $\R$.
For some fixed even function $w:\R\to\R$, the interaction potential $W_{\gamma,\ell}\colon\R\to\R$ is defined by
\begin{align}
    \label{eq:def:rescaledW}
    W_{\gamma,\ell}(x) := \gamma \ell w\bra*{\frac{x}{\ell}} \qquad \text{for } \abs{x}<\frac{1}{2} ,
\end{align}
and extended $1$-periodically to $\R$. We are interested in locally attractive interactions.
The results of our paper are valid under the following assumption on $w$:
\begin{assumption}
    \label{w-assumption}
    The function $w:\R \to (-\infty,0]$ satisfies:
    \begin{enumerate}[label=(\roman*)]
        \item $w$ is an even and Lipschitz continuous function which is symmetrically non-decreasing, that is, $w$ is non-decreasing on $[0,\infty)$.
        \item  $w$ has compact support, there exists some $s_w\in [0,\infty)$ such that $\supp w = [-s_w, s_w]$ and $w$ is bounded from above by $0$.
        \item $w$ is $C^2$ in a neighborhood of $0$ and $w$ attains its global minimum at $0$ with $w''(0) > 0$;
    \end{enumerate}
\end{assumption}
Typical examples include the Hegselmann--Krause family~\cite{GarnierEtAl2017}, where $w$ is such that
\begin{align}
    \label{eq:w-def:Hegselmann--Krause}
    w'(x) = \phi_0(\abs{x})x, \qquad \text{with } \phi_0:[0,\infty)\to [0,\infty) \text{ having compact support.}
\end{align}
The simplest example of this is
\begin{align}
    \label{eq:w-def:Hegselmann--Krause-special}
    w(x)= \frac{1}{2}({x}^2-1) \chi_{[0,1]}(\abs{x}),
\end{align}
that is, $\phi_0(r) = \chi_{[0,1]}(r)$ in~\eqref{eq:w-def:Hegselmann--Krause}.
We are especially interested in the \textit{strong local} interaction regime, i.e., $\gamma \gg 1$ and $\ell \ll 1$.

If $w$ is even, we can write~\eqref{eq:microscopic-system} in the form
\begin{align}
    \label{eq:microscopic-system:as-vector}
    \d \mathbf{X}_t = -\nabla\mathbf{W}_{\gamma,\ell}(\mathbf{X}_t) \dx t + \sqrt{2} \dx \mathbf{B}_t,
\end{align}
where $\mathbf{X}_t = (X_t^1, \dots, X_t^N)$,
$\mathbf{W}_{\gamma,\ell}(\mathbf{x}) = \frac{1}{2N}\sum_{j,k=1}^N W_{\gamma,\ell}(x_j-x_k)$ is $N$ times the microscopic interaction energy
and $\mathbf{B}_t = (B_t^1, \dots, B_t^N)$.

\subsubsection*{The mean-field limit}
If the initial positions $X^1_0, \dots, X^N_0$ are sampled i.i.d.\ from some $\rho_0\in\mathcal P(\R^d)$ with sufficiently many moments and if $w$ is sufficiently regular, then in the limit as the number of particles $N$ goes to infinity, the empirical measure $\mu^N_t:= \frac{1}{N}\sum_{i=1}^N \delta_{X^i_t}$ at time $t$ converges~\cite{MR1108185,ReviewChaintronI,ReviewChaintronII} to $\rho_t$, where $(\rho_t)_{t\ge 0}$ solves the mean-field McKean--Vlasov PDE
\begin{align}
    \label{eq:McKean-Vlasov-PDE}
    \partial_t \rho_t & = \Delta \rho_t + \nabla \cdot \bra*{\bra*{\nabla W_{\gamma,\ell} * \rho_t} \rho_t},
\end{align}
where $t\ge 0$ and $x\in\T$.
The PDE~\eqref{eq:McKean-Vlasov-PDE} may be seen as a $W_2$-gradient flow~\cite{CarrilloMcCannVillani2003} for the free energy
\begin{equation}
    \label{eq:def:free_energy}
    \cF_{\gamma,\ell}(\rho) = \int_{[0,1]} \rho(x) \log \rho(x)\dx x + \frac{1}{2} \iint_{[0,1]^2} W_{\gamma,\ell}\bra*{x-y} \rho(x) \rho(y) \dx x \dx y.
\end{equation}
The free energy acts as a Lyapunov functional for the evolution~\eqref{eq:McKean-Vlasov-PDE}:
\begin{align}
    \frac{\d}{\d t} \cF_{\gamma,\ell}(\rho_t) = - \int_{[0,1]} \rho_t \abs[\big]{\nabla \bra[\big]{\log \rho_t +  W_{\gamma,\ell} * \rho_t}}^2 \dx x \le 0.
\end{align}

\subsection{Summary of the effects}
\label{sec:sumnary}
In this work, we study the behaviour of various competing phenomena that are present in the interacting particle system~\eqref{eq:microscopic-system} and the McKean--Vlasov equation~\eqref{eq:McKean-Vlasov-PDE} and their associated timescales. We summarize these effects below:
\begin{figure}[ht!]
    \centering
    \begin{subfigure}[b]{0.49\textwidth}
        \centering
        \includegraphics[width=\textwidth]{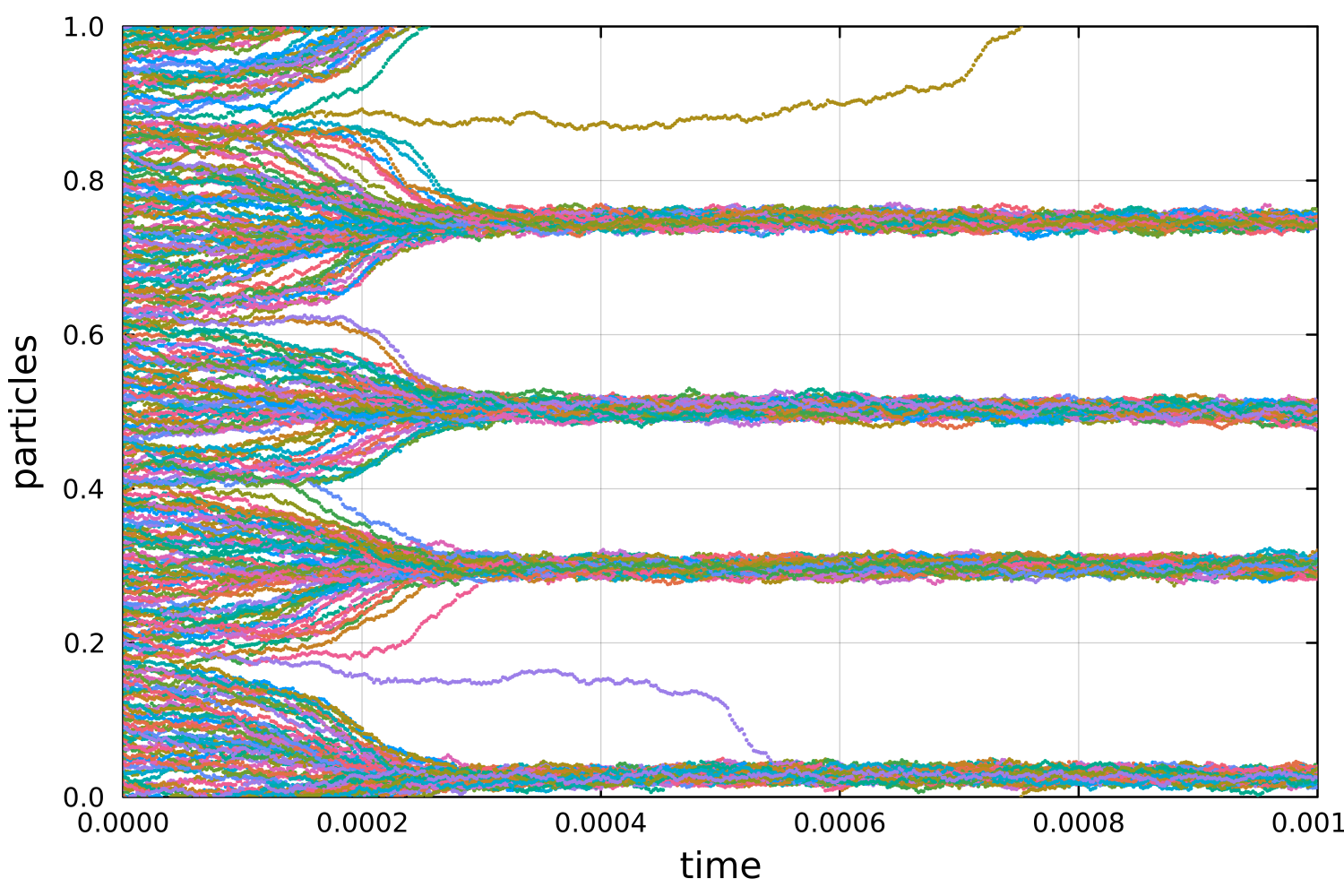}
    \end{subfigure}
    \hfill
    \begin{subfigure}[b]{0.49\textwidth}
        \centering
        \includegraphics[width=\textwidth]{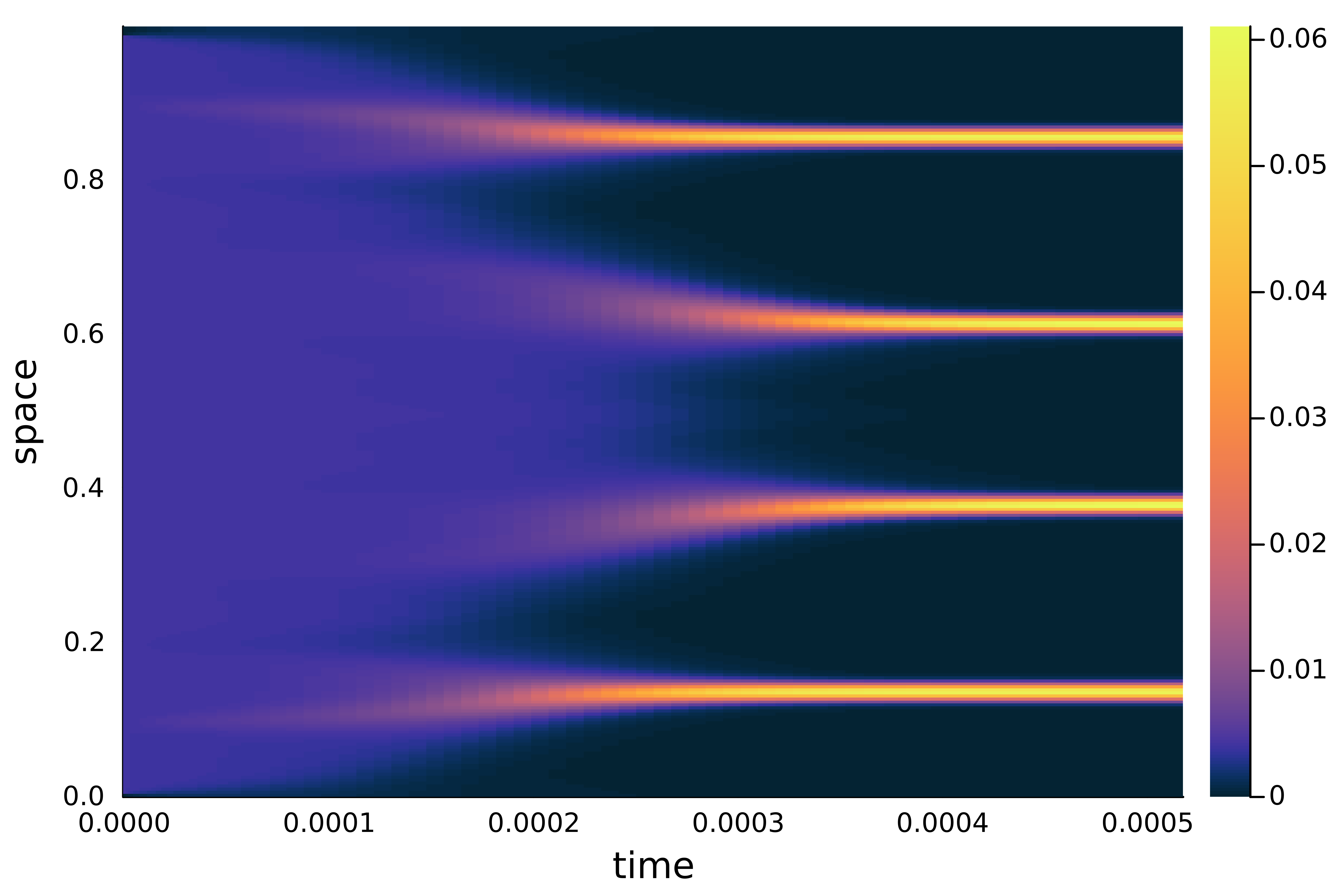}
    \end{subfigure}
    \caption{
    Initial clustering for the
    particle model (left) and for the mean-field PDE (right)  for the potential~\eqref{eq:w-def:Hegselmann--Krause-special} with $\ell = 0.1$ and $\gamma = 10^4$.
    For the particle model, initial positions are evenly spaced on $[0,1)$, that is $X^i_0 = \frac{i}{N}$ for $i=1,\dots,N$. The PDE is started from the state $\frac{100}{98}\chi_{[0.01, 0.99]}$ which is a small perturbation of the uniform state.
    The solution to the PDE is simulated using the Scharfetter--Gummel scheme~\cite{SS22}, please see \href{https://gitlab.uni-ulm.de/clustercoarsening/cluster-coarsening-paper}{https://gitlab.uni-ulm.de/clustercoarsening/cluster-coarsening-paper} for the code repository used to generate the figures in this paper.
    }
    \label{fig:initial-clustering:both}
\end{figure}

\begin{enumerate}[label=(E\arabic*)]
    \item \label{effect:initial-clustering}
          \textbf{Formation of clusters:} The particles, although initially spread out, quickly form multiple clusters driven by the attractive nature of the interaction potential (\cref{fig:initial-clustering:both}, left).
          Typically, this is the shortest timescale in the system. The mean-field PDE~\eqref{eq:McKean-Vlasov-PDE} exhibits a similar phenomenon (\cref{fig:initial-clustering:both}, right).
    \item \label{effect:coalesence}
          \textbf{Coalescence of clusters:} The clusters themselves perform Brownian motions and so can collide and merge into each other (\cref{fig:coalescence}). This phenomenon is absent in the mean-field system since its associated timescale grows linearly with the number of particles.
    \item \label{effect:mass-exchange}
          \textbf{Leakage of mass and dynamical metastability:} Once the clusters are formed, individual particles can leak from one cluster to another driven by excursions of the Brownian motions which drive the individual particles. This phenomenon is present both in the particle system and in the mean-field limit. The associated timescale is of the form $e^{\gamma \ell \Delta m}$, where $\Delta = - \inf w$ and $m$ is the cluster's macroscopic mass, that is, the number of particles in the cluster divided by $N$.
          In the case of the microscopic model, \cref{effect:mass-exchange} competes with~\cref{effect:coalesence} depending on the relative sizes of the timescales of the two phenomena, which are of the order $e^{\gamma \ell \Delta m}$ and $N$, respectively.
          ~\cref{fig:mass-exchange-three-clusters} shows an example for the mass exchange between clusters if the PDE is started initially with three clusters with different masses. Since the timescale for mass leakage $e^{\gamma \ell \Delta m}$ depends exponentially on $m$, the mass exchange among clusters leads to dynamical metastability in the PDE~\cite{Otto2007}: A multi-cluster state can remain stable over a long period of time until it suddenly collapses to a single-cluster state, see~\cref{fig:mass-exchange-three-clusters,fig:2-clusters-dynamical-metastab} for examples.
    \item \label{effect:reversibility}
          \textbf{Microscopic reversibility:} Finally, once a single-cluster state has been formed, it is always possible for it to spontaneously dissolve due to the effect of the noise. This is a large $N$ large deviation event of Dawson--G\"artner-type~\cite{DG1989} and occurs at a timescale of order $e^{N\Delta_{\mathcal F} }$, where $\Delta_{\mathcal F}$ is the free energy gap of the limiting free energy. Since it is driven by noise and  the timescale grows exponentially in the number of particles, this phenomenon is present only in the particle system and not the mean-field system.
\end{enumerate}

\begin{figure}
    \centering
    \includegraphics[width=0.6\textwidth]{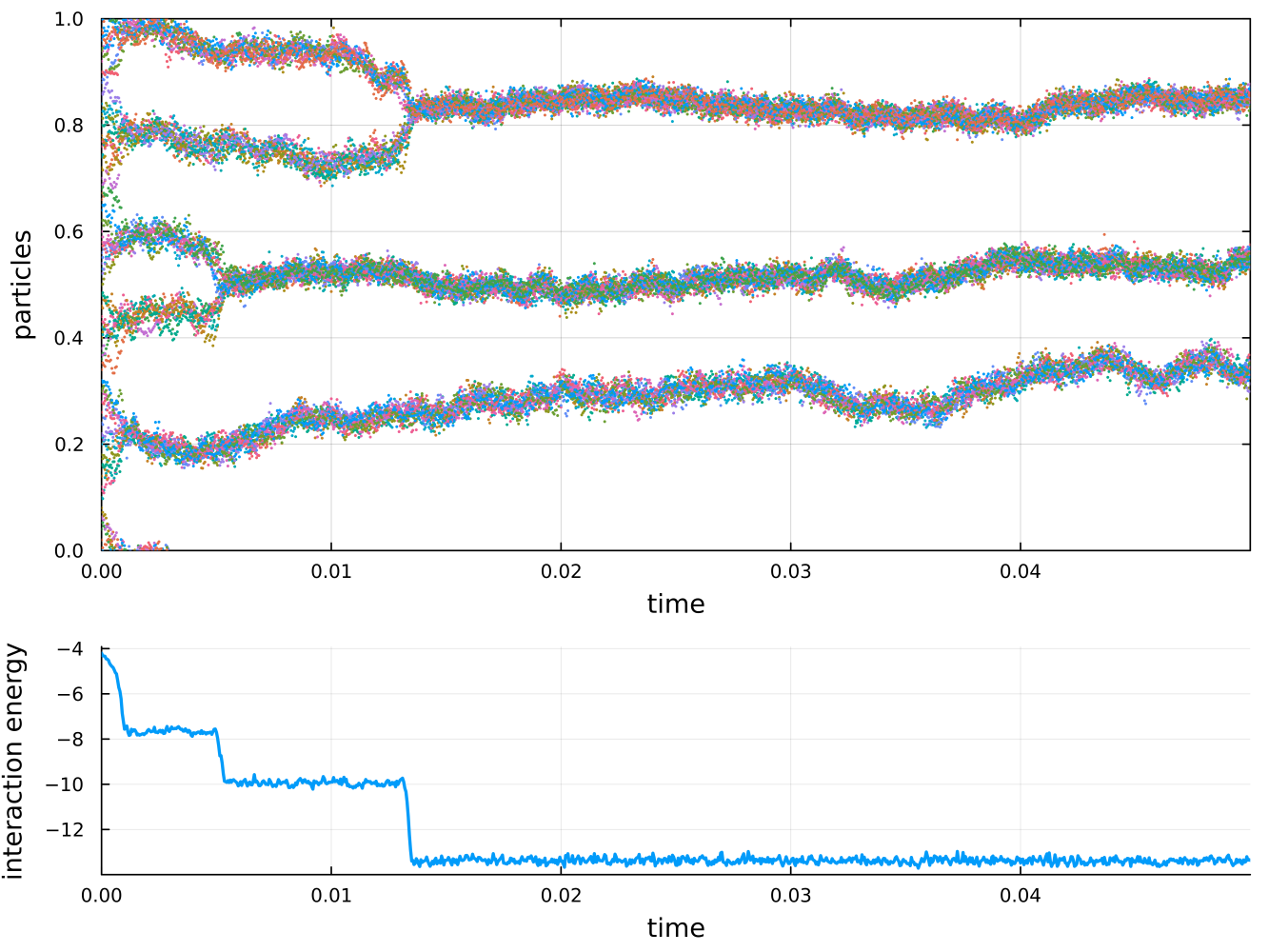}
    \caption{Coalescence of clusters in the particle system~\eqref{eq:microscopic-system} where $N=50$, $\gamma=2000$, $\ell=0.08$ and the particles start at $X^i_0=\frac{i}{N}$.
        ~\eqref{eq:microscopic-system} was simulated using the Euler--Maruyama scheme with time step $\Delta t = 0.00005$.
        The bottom plot shows the interaction energy
        $\frac{1}{2N^2}\sum_{i,j=1}^N W_{\gamma,\ell}(X^i_t - X^j_t)$
        as an order parameter.}
    \label{fig:coalescence}
\end{figure}

\begin{figure}
    \centering
    \begin{subfigure}[b]{0.49\textwidth}
        \centering
        \includegraphics[width=\textwidth]{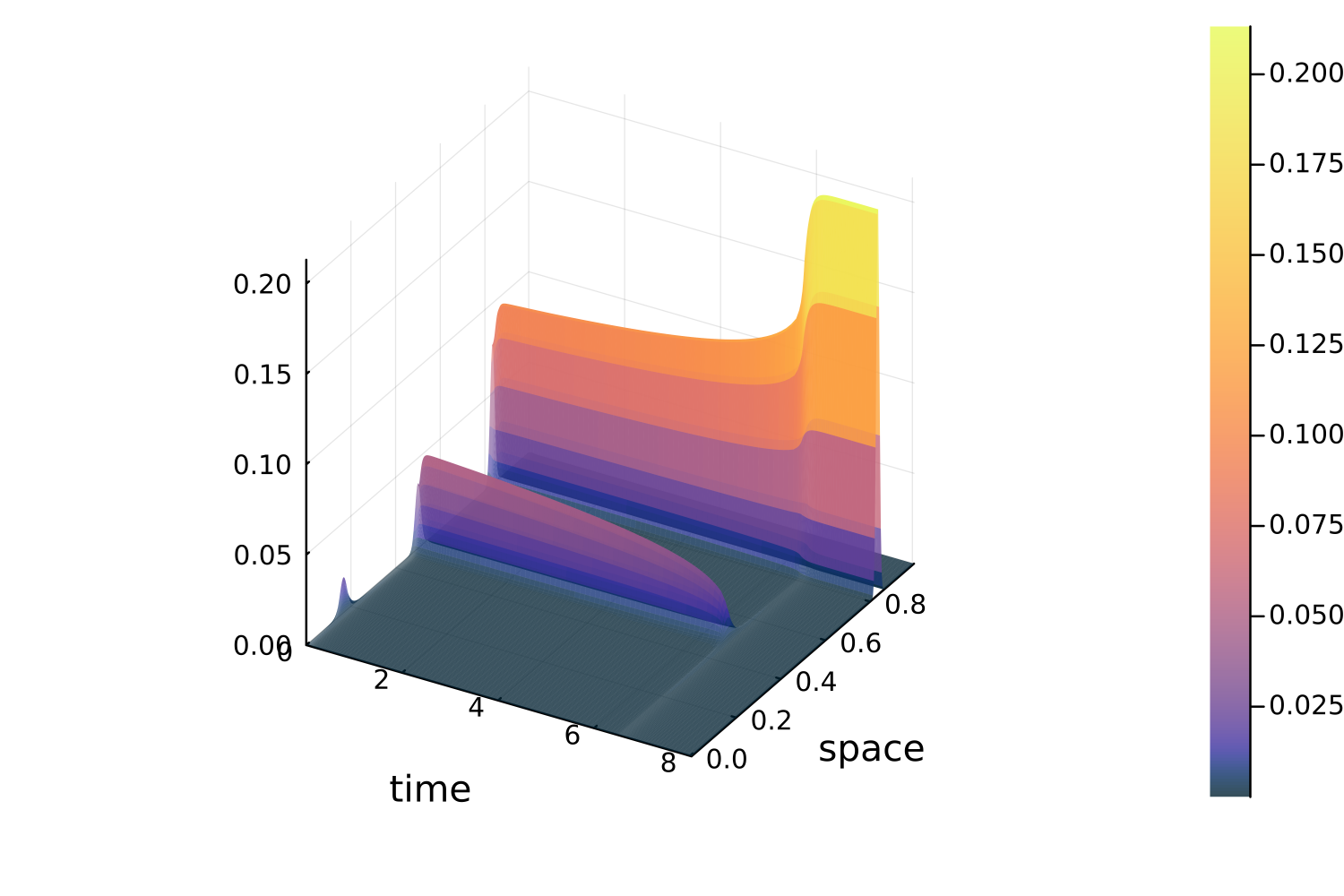}
    \end{subfigure}
    \hfill
    \begin{subfigure}[b]{0.49\textwidth}
        \centering
        \includegraphics[width=\textwidth]{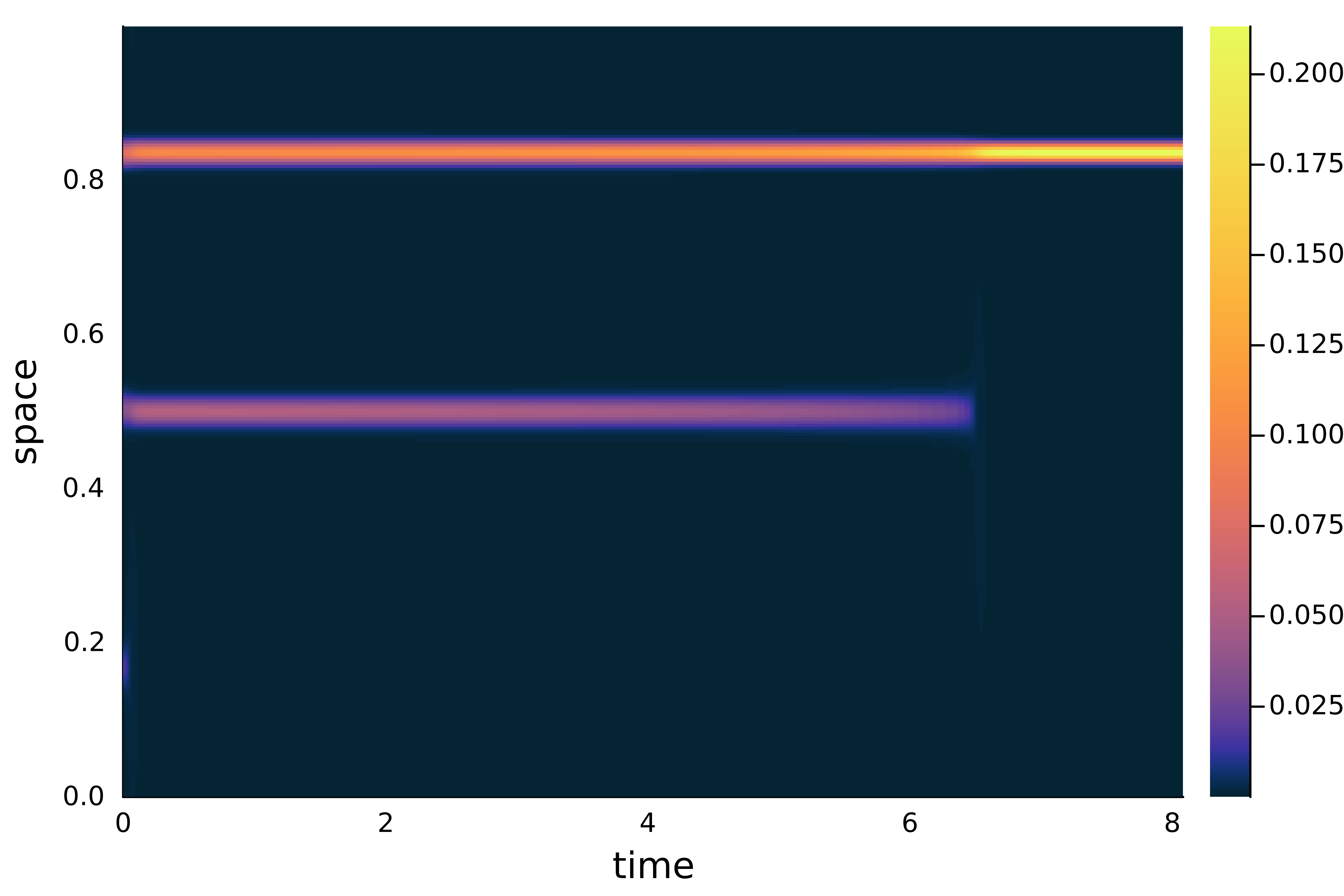}
    \end{subfigure}
    \caption{
        Exchange of mass between three clusters in the mean-field PDE. The potential is given by~\eqref{eq:w-def:Hegselmann--Krause-special} with $\gamma=1000$ and $\ell=0.05$. The initial state is given by a Gaussian mixture (truncated at $0$ and $1$ and normalized) with centres $(0.1, 0.3, 0.7)$ and variances $\sigma_k^2 = \frac{\ell}{\gamma m_k}$ for $k=1,2,3$ and $m_1=0.2, m_2=0.3, m_3=0.5$. This form of the variances will be explained in more detail in~\cref{sec:stationary-states:approximate-gaussian}.}

    \label{fig:mass-exchange-three-clusters}
\end{figure}
\begin{figure}
    \centering
    \includegraphics[width=0.7\textwidth]{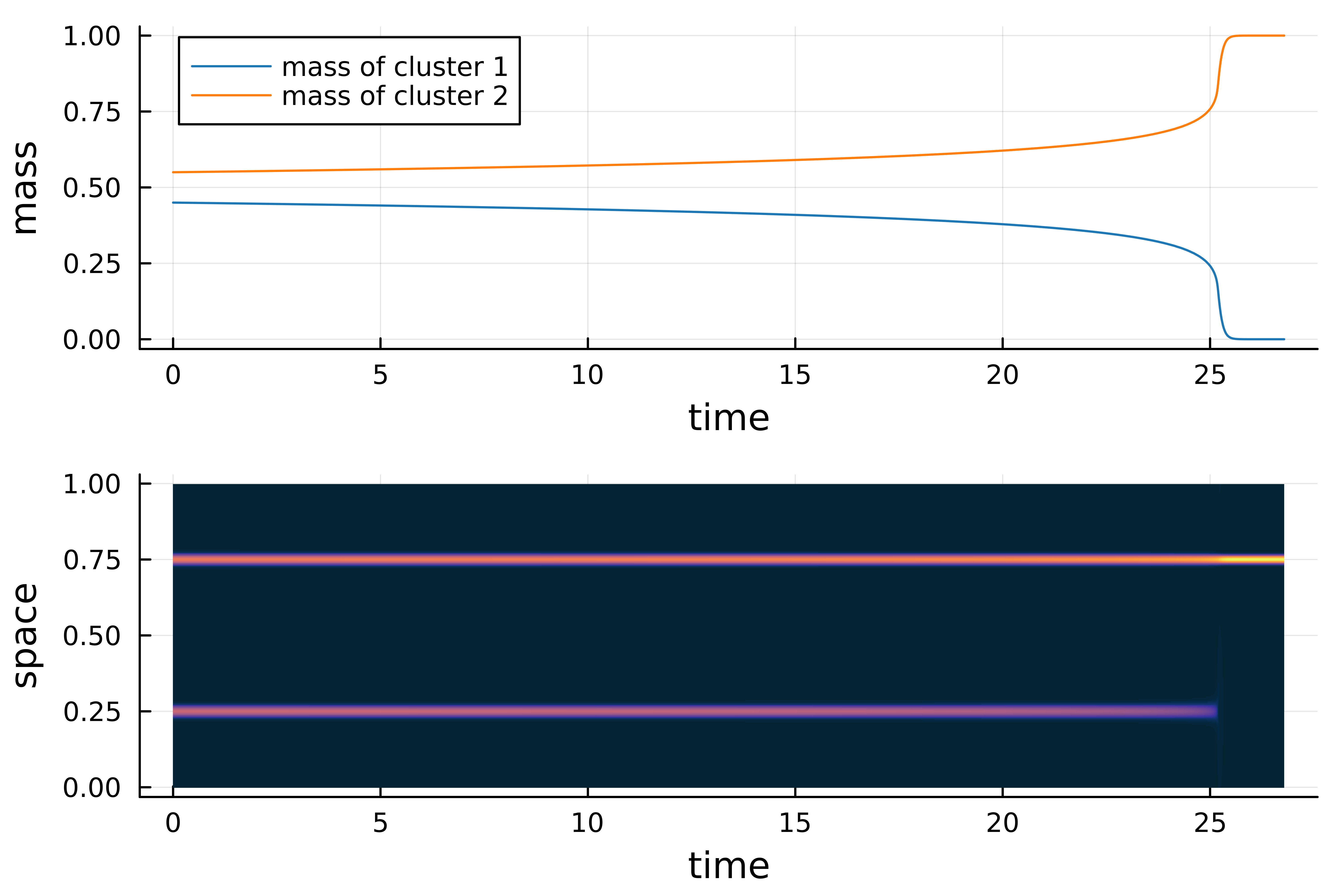}
    \caption{Dynamical metastability in the PDE: The PDE is started from a Gaussian mixture with centres $0.25$ and $0.75$ and initial masses $0.45$, $0.55$, respectively. $w$ is given by~\eqref{eq:w-def:Hegselmann--Krause-special} and $\gamma=1000$ and $\ell=0.05$.}
    \label{fig:2-clusters-dynamical-metastab}
\end{figure}
In what follows,  we will analyse each of these phenomena individually and study the dependence of the associated timescale on the parameters of the system, i.e.\ the strength of the interaction $\gamma$, the length scale of the interaction $\ell$, and the number of particles in the system $N$.
As already mentioned above, depending on the relative sizes of the timescales for~\cref{effect:coalesence} and~\cref{effect:mass-exchange}, one of the two effects can dominate the behaviour of the particle system. On the other hand, since the timescale for effect~\cref{effect:coalesence} roughly scales like $N$, the mean-field system does not show effect~\cref{effect:coalesence}.

As a by-product of our analysis, we derive a coarse-grained system which effectively describes both phenomena~\cref{effect:coalesence} and~\cref{effect:mass-exchange}. Moreover, we propose a suitable regime in which the merging for the clusters in the particle system (i.e. phenomenon~\cref{effect:coalesence}) is captured in terms of a stochastic process closely related to the modified massive Arratia flow~\cite{Konarovskyi2017,KonRenesse_2019}.

\subsection{Our contributions}

A similar model was already studied in~\cite{GarnierEtAl2017} for the particular case of the Hegselmann-Krause interaction potential; it was shown heuristically that if $\gamma$ is sufficiently large, the microscopic system exhibits initial clustering (effect~\cref{effect:initial-clustering}).
Subsequently, the authors explain how the clusters move like coalescing heavy Brownian motions (effect~\cref{effect:coalesence}). See also the numerical experiments and statistical analysis of the clusters in~\cite{WSMSPW_2025}.
Our contributions are as follows.
\begin{itemize}
    \item Generalising the heuristics of~\cite{GarnierEtAl2017} on the initial formation of clusters, we determine the timescale of initial clustering for the mean-field PDE as well as for the particle system under more general initial conditions than~\cite{GarnierEtAl2017}. Moreover, we determine the point of linear stability~$\gamma_\sharp$ explicitly.
    \item We prove for a large family of compactly supported potentials $w$ such as~\eqref{eq:w-def:Hegselmann--Krause-special} that global minimisers of the free energy~\eqref{eq:def:free_energy} are either the uniform state or single-cluster states in the sense that they are symmetrically decreasing with respect to their centre of mass. The latter is true beyond the phase transition, that is, for $\gamma > \gamma_c$, in which case global minimisers are non-trivial, that is, different from the uniform state. 
    The argument is based on a new variant of the strict Riesz rearrangement inequality which holds true for interaction kernels whose monotonicity is only locally strict (in contrast to the globally strictly decreasing kernels considered in~\cite{beckner1993sharp}) and can therefore be applied to certain compactly supported interaction potentials (see also~\cite[Chapter 5]{Shalova2025Thesis} for related results).
    \item By extending the description from~\cite{GarnierEtAl2017} for the coalescence of clusters, we introduce a new model to approximately describe both the coalescence of clusters~\cref{effect:coalesence} and the mass exchange between clusters~\cref{effect:mass-exchange} due to the excursions of single particles.
          In particular, we argue that the clusters in the mean-field PDE are unable to move and that the PDE cannot describe the merging of clusters, except for situations where two clusters are within the range of the interaction potential.
          Therefore, apart from such cases, coarsening in the PDE occurs only through mass exchange.

    \item Based on our coarse-grained model for the mass exchange, we argue and show numerically that the PDE exhibits dynamical metastability: A multi-cluster state can remain stable for a long period of time of order $e^{\gamma \ell \Delta m}$ until one of its clusters suddenly vanishes.
    \item Besides the timescales for~\cref{effect:initial-clustering},~\cref{effect:coalesence} (which is of order $N$) and~\cref{effect:mass-exchange} (which is of order $e^{\gamma \ell \Delta m}$), we also describe the timescale of microscopic reversibility~\cref{effect:reversibility}.
          In the mean-field limit $N\to\infty$, the timescales for coalescence and microscopic reversibility diverge to infinity, leaving the timescale for~\cref{effect:initial-clustering} and~\cref{effect:mass-exchange} as the ones which determine the behaviour of the PDE.
    \item We identify certain scaling regimes in which we conjecture the system~\eqref{eq:microscopic-system} to converge in a suitable limit towards the Modified Massive Arratia flow~\cite{Konarovskyi2017,KonRenesse_2019}.
\end{itemize}

We can summarise our observations by considering the empirical measure $\mu_t^N$ of the interacting particle system which can be formally interpreted as the solution of the following interacting Dean--Kawasaki equation:
\begin{equation}
    \partial_t \mu_t^N= \Delta \mu_t^N +\nabla\cdot(\mu_t^N\nabla W_{\gamma,\ell} * \mu_t^N) + \sqrt{2}N^{-\frac12}\nabla\cdot\bra*{\sqrt[\small]{\cramped{\mu_t^N}}\xi} \, ,
\end{equation}
where $\xi$ is vector-valued space-time white noise. The above SPDE has the structure of a slow gradient flow (corresponding to the deterministic part of the above equation) perturbed by a random process (corresponding to the transport noise), making the SPDE a reversible Markov process  with respect to the lifted Gibbs measure of the $N$-particle system. The competing timescales and the effects~\cref{effect:initial-clustering,effect:coalesence,effect:mass-exchange,effect:reversibility}  can then be interpreted as arising from the different terms in the SPDE.

The initial clustering~\cref{effect:initial-clustering} is driven by the short-time linear instability of the constant state under the influence of the deterministic part of the equation. The coalescence of clusters~\cref{effect:coalesence} is produced by the noise term driving the motion of clusters. Indeed, one can check that $\sqrt{2}N^{-\frac12}\nabla\cdot\bra*{\sqrt[\small]{\cramped{\mu_t^N}}\xi}$ has quadratic variation of order $N^{-1}t$ thus producing a linear timescale in $N$. The mass exchange~\cref{effect:mass-exchange} is produced by the slow motion of the deterministic part (the slow gradient flow) and arises from the competition between the Laplacian and the interaction term. The timescale for microscopic reversibility~\cref{effect:reversibility} can also be explained as the small-noise (large $N$) large deviation timescale of the above SPDE which is exponentially large in $N$.

\subsection{Notation}
For two functions $f, g:(0,\infty)\to(0,\infty)$, we write $f(x)\sim g(x)$ to indicate asymptotic equivalence for $x\to\infty$, that is, $\lim_{x\to\infty} \frac{f(x)}{g(x)}=1$. In addition, we write $f(x) \asymp g(x)$ to denote asymptotic proportionality for $x\to\infty$, i.e., $\limsup_{x\to \infty} \frac{f(x)}{g(x)}<\infty$ and $\liminf_{x\to \infty} \frac{f(x)}{g(x)}>0$. Given $\mu\in\R, \sigma>0$, we write $g(\dummy;\mu,\sigma^2)$ for the density of the one-dimensional Gaussian distribution with mean $\mu$ and variance $\sigma^2$. By abuse of notation, we will often use the same symbol for an absolutely continuous probability measure and its density. For particles $X^1_t, \dots, X^N_t$ at time $t$, we define the empirical measure as $\mu^N_t = \frac{1}{N}\sum_{i=1}^N \delta_{X^i_t}$.  The space of probability measures on a Polish space $E$ is denoted by $\mathcal P(E)$.

\subsection{Outline}
In~\cref{sec:initial-clustering}, we generalize the heuristics by~\cite{GarnierEtAl2017} on the initial clustering. The timescales for initial clustering for the particle system and the PDE can be found in~\cref{sec:initial-clustering-pde} and~\cref{sec:initial-clustering-particle-system}, respectively.

For generic initial conditions and if $\gamma \gg1$, numerical experiments show that the PDE converges to a single-cluster state.
In order to better understand the shape of $\rho_t$ as $t\to\infty$, we analyse the minimisers of the free energy~\eqref{eq:def:free_energy} more closely in~\cref{sec:stationary-states}: First, we prove that for $\ell \ll 1$, the mean-field PDE has a discontinuous phase transition for $\gamma\to\infty$, proving that for sufficiently large $\gamma$, there exist non-trivial global minimisers of the free energy apart from the trivial uniform state $\rhounif\equiv 1$ (see~\cref{sec:discontinuous-phase-transition}). Moreover, we use a strict version of the Riesz rearrangement inequality to show that all global minimisers of the free energy~\eqref{eq:def:free_energy} are either uniform or symmetrically decreasing, i.e., single-cluster states, see~\cref{ssec:Riesz}. 

In~\cref{sec:ReducedModel} we state and derive a reduced model which is able to describe both the coalescence~\cref{effect:coalesence} and the leakage of mass~\cref{effect:mass-exchange}.
We start by recalling the Markovian model from~\cite{GarnierEtAl2017} in~\cref{ssec:model:Garnier} and state our extended model with mass exchange in~\cref{ssec:model:MassExchange}. We show that our new model can explain dynamical metastability of the mean-field limit in~\cref{ssec:model:DynMetastable} and discuss its limitations in~\cref{ssec:model:limitations}. The heuristic derivation of our model is contained in~\cref{ssec:derivation}. As a byproduct, we also obtain a jump process description in~\cref{ssec:model:jumps}.

We conclude the paper by discussing possible future directions  in~\cref{sec:outlook}. Moreover, in~\cref{sec:ModArratia}, we formulate a conjecture on the convergence of~\eqref{eq:microscopic-system} to the Modified Massive Arratia flow~\cite{Konarovskyi2017,KonRenesse_2019} in a suitable scaling limit. Finally, \cref{app:proof-prop-laplace-hitting-time} provides a proof of the computation of the mass leakage times used in our extended model.

\section{Initial clustering}
\label{sec:initial-clustering}

\subsection{Background: Initial clustering  after \texorpdfstring{\cite{GarnierEtAl2017}}{[GPY17]} via linearization}\label{ssec:Initial:GPY}

We review here the heuristics on the initial clustering from~\cite{GarnierEtAl2017}, extending their analysis from the Hegselmann--Krause case~\eqref{eq:w-def:Hegselmann--Krause} to the more general~\cref{w-assumption} as well as examining the timescales of initial clustering for more general initial conditions (see~\cref{sec:initial-clustering-pde} and~\cref{sec:initial-clustering-particle-system}).

Assuming that $\rho_t = \rhounif + \rhopert$, where $\rhounif(x) \equiv 1$ is the uniform equilibrium state and $\rhopert$ is a small perturbation, we linearize the McKean--Vlasov PDE~\eqref{eq:McKean-Vlasov-PDE} around $\rhounif$ to obtain the linear PDE
\begin{align}
    \label{eq:linearized-PDE}
    \partial_t \rhopert = \partial_x^2 \rhopert + \partial_x \bigl(\rhounif W'_{\gamma,\ell} * \rhopert \bigr).
\end{align}
The Fourier coefficients $\rhopert[\hat](t,k) = \int_0^1 \rhopert(t,x) e^{-ikx} \dx x$ for $k \in \mathcal{K} := 2\pi \Z$ satisfy the system of linear ODEs
\begin{align}
    \partial_t\rhopert[\hat](t,k) = \pra[\bigg]{i\rhounif k\int_{[0,1]} W'_{\gamma,\ell}(x) e^{-ikx} \dx x - k^2}\rhopert[\hat](t,k).
\end{align}
Hence, we can read off the growth rates of the Fourier modes as
\begin{align}
    \psi(k) & := \operatorname{Re}\bra*{i\rhounif k\int_{[0,1]} W'_{\gamma,\ell}(x) e^{-ikx} \dx x - k^2} = -k^2 \bra*{ \hat{W}_{\gamma,\ell}(k) + 1},
\end{align}
where
\begin{align}
    \widehat{W}_{\gamma,\ell}(k) & := \int_0^1 W_{\gamma,\ell}(x) \cos(kx) \dx x =
    \gamma\ell \int_{-\frac{1}{2}}^{\frac{1}{2}} w\bra*{\frac{x}{\ell}} \cos(k x) \dx x \\ &=- \gamma \int_{-\frac{1}{2}}^{\frac{1}{2}} w'\bra*{\frac{x}{\ell}} \frac{\sin(kx)}{k} \dx x
\end{align}
are the Fourier modes of the interaction potential $W_{\gamma,\ell}$. By the Riemann--Lebesgue lemma, the growth rates $\psi(k)$ are bounded from above for $k\in \mathcal{K}$ and are therefore maximized at some $k_{\rm max} = \argmax_{k \in \mathcal K, k\ge0} \psi(k)$. The sign of the maximal growth rate $\psi(k_{\rm max})$ determines the stability of~\eqref{eq:linearized-PDE}:
For $\ell>0$ fixed and small; we consider $\gamma>0$ as bifurcation parameter, then the minimal $\gamma>0$ such that one of the growth rates $\psi(k)$ becomes positive is denoted by $\gamma_\sharp$, called the \textbf{point of linear stability}. It is given by
\begin{align}
    \label{eq:gamma-sharp}
    \gamma_\sharp & :
    = - \frac{1}{\min_{k \in \mathcal K \setminus \{0\}} \widehat{W}_{\gamma=1, \ell}(k)} \,,
\end{align}
see also~\cite{CP10}, where $\gamma_\sharp >0$ is true for sufficiently small $\ell>0$ since for those $\ell$ it holds that $\widehat{W}_{1,\ell}(2\pi) = \ell \int _{\supp w} w(x)\cos(2\pi\ell x)\dx x < 0$. Below in~\cref{lemma:gamma-sharp}, we will show that the minimum in~\eqref{eq:gamma-sharp} is always attained at the first Fourier mode $k=\pm 2\pi$.

In case $\gamma < \gamma_\sharp$, we have $\max_{k \in \mathcal K} \psi(k) < 0$, the linearized system~\eqref{eq:linearized-PDE} is stable and small perturbations of the uniform state $\rhounif$ will converge to the uniform state $\rhounif$, see~\cite{CGPS20} for exact results directly on the level of the nonlinear PDE~\eqref{eq:McKean-Vlasov-PDE}.

On the other hand, if $\gamma > \gamma_\sharp$, then
$\max_{k \in \mathcal K} \psi(k) > 0$ and the system~\eqref{eq:linearized-PDE} is unstable. Therefore, provided that the mode $k_{\rm max}$ with maximal growth rate is already present initially, that is,  $\rhopert[\hat](0,k_{\rm max})= \overline{\rhopert[\hat](0,-k_{\rm max})} \neq 0$, the short-time behaviour of~\eqref{eq:linearized-PDE} will be dominated by $e^{\pm i k_{\max}t}$, whose amplitude grows according to
    \begin{align}
        \label{eq:dominant-modes-growth}
        \abs{\rhopert[\hat](t, \pm k_{\rm max})}= e^{\psi(\pm k_{\rm max})t} \abs{\rhopert[\hat](0, \pm k_{\rm max})} \qquad \text{for $t>0$}.
    \end{align}
In particular, \eqref{eq:linearized-PDE} exhibits initial clustering by means of {$e^{\pm i k_{\max}t}$} on a timescale of order
    \begin{align}
        \label{eq:initial-clustering-time}
        \tcluster \asymp \frac{- \log\bra*{\abs{\rhopert[\hat](0, k_{\rm max})}}}{\psi(k_{\rm max})}.
    \end{align}
The estimated number of initial clusters is determined by $k_{\rm max}$ and given by
\begin{align}
    \label{eq:number-of-clusters}
    n_{\rm clusters} = \frac{k_{\rm max}}{2\pi} \in \N,
\end{align}
with the inter-cluster distance being $d_{\rm clusters}=1/n_{\rm clusters}=2\pi/k_{\rm max}$.
We will now describe the occurrence and the timescale for initial clustering for both the PDE~\eqref{eq:McKean-Vlasov-PDE} and the particle system.

\subsection{Initial clustering in the PDE}\label{ssec:InitialPDE}
\label{sec:initial-clustering-pde}
\newcommand{\rhostartmod}[1]{#1} 
Assume that $\gamma >\gamma_\sharp$ and that the PDE is started from a small perturbation $\rhostartmod{\rho}_0\in \mathcal{P}([0,1])\cap L^2([0,1])$ of the uniform state $\rhounif\equiv 1$ such that $\widehat{\rho_0}(0, k_{\rm max})\neq 0$. Then, the short-time behaviour of the PDE is governed by its linearization~\eqref{eq:linearized-PDE} around $\rhounif$ and the previous section shows that the PDE exhibits initial clustering with a number of $n_{\rm clusters}$ clusters. This is also observed numerically, see~(\cref{fig:initial-clustering:both}, right). Because of~\eqref{eq:initial-clustering-time}, the initial clustering timescale is then of the order
\begin{align}
    \label{eq:clustering-timescale:pde}
    \tcluster \asymp  \frac{- \log\bra*{\abs*{
                \widehat{\rhostartmod{\rho}_0}(k_{\rm max}) }}}{\psi(k_{\rm max})}.
\end{align}

\subsection{Initial clustering in the particle system}
\label{sec:initial-clustering-particle-system}
Let $\rhostartmod{\rho}_0\in \mathcal{P}([0,1])\cap L^2([0,1])$ be a probability measure which is close to or equal to the uniform distribution $\rhounif$ and assume that the initial positions of the particles~\eqref{eq:microscopic-system} are sampled i.i.d.\ from $\rhostartmod{\rho}_0$. By the central limit theorem, the fluctuations
\begin{equation}\label{eq:particle-rho-1-fluctuations}
    \rhopert^N (0,x) = \sqrt{N} \pra[\Bigg]{\frac{1}{N} \sum_{j=1}^N \delta_{X^j_0} - \rhostartmod{\rho}_0 }
\end{equation}
converge to a Gaussian-valued measure $\mu$, that is, for each $h\in C_b([0,1])$, the real-valued random variable $\int_{[0,1]} h(x) \mu(\d x)$ is $\mathcal{N}(0, \sigma^2(h))$-distributed, where $\sigma^2(h) =  \int \bigl(h(x)- \int h(y)  \rhostartmod{\rho}_0(\d y)\bigr)^2 \rhostartmod{\rho}_0(\d x)$. Writing the initial empirical measure $\mu^N_0$ in the form
\begin{align}
    \mu^N_0 = \frac{1}{N} \sum_{j=1}^N \delta_{X^j_0}\approx \rhounif + (\rhostartmod{\rho}_0- \rhounif) + \frac{\mu}{\sqrt{N}},
\end{align}
we find from~\eqref{eq:initial-clustering-time} that the initial clustering time is of the order
\begin{align}
    \tcluster \asymp \frac{- \log\bra[\big]{\abs[\big]{
                \widehat{\rhostartmod{\rho}_0}(k_{\rm max}) + \frac{\hat{\mu}(k_{\rm max})}{\sqrt{N}}}}}{\psi(k_{\rm max})},
\end{align}
where $\hat{f}(k)= \int_{[0,1]}f(x)e^{-ikx}\dx x $ denotes the $k$-th Fourier mode of $f$.

For the initial clustering timescale in the particle system, this implies the following: If the particles are sampled i.i.d.\ from the uniform distribution $\rhounif$, then the initial clustering time is of the order
\begin{align}
    \tcluster \asymp \frac{\log N}{\psi(k_{\rm max})}.
\end{align}
This was already described in~\cite{GarnierEtAl2017} and is in line with the fact that $\rhounif$ is a stationary state for the mean-field PDE~\eqref{eq:McKean-Vlasov-PDE}, so that when starting~\eqref{eq:McKean-Vlasov-PDE} at $\rhounif$, no clustering occurs.

On the other hand, if the particles are sampled i.i.d.\ from a distribution $\rhostartmod{\rho}_1$ which is close to $\rhounif$, but $\widehat{\rhostartmod{\rho}_1}(k_{\rm max}) \neq 0$ (in particular, $\rhostartmod{\rho}_1 \neq \rhounif$), then the initial clustering time is of the same order~\eqref{eq:clustering-timescale:pde} as for the limiting PDE~\eqref{eq:McKean-Vlasov-PDE}.

\subsection{Calculating the point of linear stability \texorpdfstring{$\gamma_\sharp$}{}}
    Next, we prove that the minimum in \eqref{eq:gamma-sharp} is always attained at the first Fourier mode. This gives an explicit expression for $\gamma_\sharp$, showing that $\gamma_\sharp \asymp \ell^{-2}$ as $\ell \downarrow 0$.
\begin{lemma}
    \label{lemma:gamma-sharp}
    Under~\cref{w-assumption}, the point of linear stability $\gamma_\sharp$ as defined in~\eqref{eq:gamma-sharp} equals
    \begin{align}
        \label{eq:gamma-sharp-explicit}
        \gamma_\sharp = - \frac{1}{\widehat{W}_{\gamma=1, \ell}(2\pi)}
    \end{align}
    whenever $0 < \ell < \frac{1}{4 s_w}$, where $\supp w = [-s_w, s_w]$. Moreover, we have the asymptotic equivalence
    \begin{align}
        \label{eq:gamma-sharp-asymptotics}
        \gamma_{\sharp} \sim \frac{1}{m_w}\ell^{-2} \qquad \text{as $\ell \downarrow 0$}, \qquad \text{where $m_w = -\int_\R w(x)\dx x>0$}.
    \end{align}
\end{lemma}
\begin{proof}
    \eqref{eq:gamma-sharp-explicit} follows once we have proven that $\widehat{W}_{\gamma=1, \ell}(k)$ is minimized on $(2\pi \Z)\setminus \{0\}$ at $k=\pm 2\pi$. To show this, fix $k\in 2\pi \Z$ such that $k \ge 4\pi$. We decompose the $k$-th Fourier mode of $W_{\gamma=1,\ell}$ as follows,
    \begin{align}
        \widehat{W}_{\gamma=1,\ell}(k) & =\ell\int_{-\frac12}^{\frac12} w\bra*{\frac{x}{\ell}} \cos(kx) \dx x = 2\ell\int_0^\infty w\bra*{\frac{x}{\ell}}\cos(kx)\dx x
        \\ & = 2 \ell
        \int_0^{\frac{\pi}{2k}} w\bra*{\frac{x}{\ell}}\cos(kx)\dx x
        + 2\ell\sum_{m=0}^\infty \int_{\frac{\pi}{2k}(1+4m)}^{\frac{\pi}{2k}(3+4m)} w\bra*{\frac{x}{\ell}}\cos(kx)\dx x                                                \\
                                       & \qquad\qquad
        + 2\ell\sum_{m=0}^\infty \int_{\frac{\pi}{2k}(3+4m)}^{\frac{\pi}{2k}(5+4m)} w\bra*{\frac{x}{\ell}}\cos(kx)\dx x,
    \end{align}
    where we split the integrals at the zeros of $\cos(kx)$. Since $w$ has compact support, the sums are finite.
    Using the substitution $x\mapsto x + \frac{\pi}{k}$, we have for each $m\in \N$ that
    \begin{align}
         & \int_{\frac{\pi}{2k}(1+4m)}^{\frac{\pi}{2k}(3+4m)} w\bra*{\frac{x}{\ell}}\cos(kx)\dx x + \int_{\frac{\pi}{2k}(3+4m)}^{\frac{\pi}{2k}(5+4m)} w\bra*{\frac{x}{\ell}}\cos(kx)\dx x
        \\={}& \int_{\frac{\pi}{2k}(1+4m)}^{\frac{\pi}{2k}(3+4m)} w\bra*{\frac{x}{\ell}}\cos(kx)
        + w\bra*{\frac{x}{\ell}+\frac{\pi}{\ell k}}\cos\bra*{kx + \pi} \dx x
        \\={}& \int_{\frac{\pi}{2k}(1+4m)}^{\frac{\pi}{2k}(3+4m)} \bra*{w\bra*{\frac{x}{\ell}}- w\bra*{\frac{x}{\ell}+\frac{\pi}{\ell k}}}\cos(kx)\dx x.
        \label{eq:integral-difference}
    \end{align}
    Since $w$ is non-decreasing on $[0,\infty)$, \eqref{eq:integral-difference} is non-negative and we obtain
                    \begin{align}
                        \widehat{W}_{\gamma=1,\ell}(k) & \ge \ell \int_{-\frac{\pi}{2k}}^{\frac{\pi}{2k}} w\bra*{\frac{x}{\ell}}\cos(kx) \dx x.
                    \end{align}
                    Moreover, since $\cos(kx) \le \cos(2\pi x)$ for all $x\in [0,\frac{\pi}{2k}]$, we have that
                    \begin{align}
                        \widehat{W}_{\gamma=1,\ell}(k) & \ge \ell \int_{-\frac{\pi}{2k}}^{\frac{\pi}{2k}} w\bra*{\frac{x}{\ell}}\cos(kx) \dx x > \ell \int_{-\frac{\pi}{2k}}^{\frac{\pi}{2k}} w\bra*{\frac{x}{\ell}}\cos(2\pi x) \dx x
                        \\ &\ge\ell \int_{-\frac{1}{4}}^{\frac{1}{4}} w\bra*{\frac{x}{\ell}}\cos(2\pi x) \dx x.
                    \end{align}
                    The inequality in the second step is strict since $w$ is negative on $(-s_w, s_w)$ and $\cos(kx) < \cos(2\pi x)$ for all $x\in (0,\frac{\pi}{2k}]$.
    Assuming that $0 < \ell < \frac{1}{4 s_w}$, we have that
    \begin{align}
         & \ell \int_{-\frac{1}{4}}^{\frac{1}{4}} w\bra*{\frac{x}{\ell}}\cos(2\pi x) \dx x  =\ell\int_{-\frac{1}{2}}^{\frac{1}{2}} w\bra*{\frac{x}{\ell}}\cos(2\pi x) \dx x = \widehat{W}_{\gamma=1,\ell}(2\pi).
    \end{align}
    In particular, $\widehat{W}_{\gamma=1,\ell}(k)$ is minimized on $(2\pi \Z )\setminus \{0\}$ at $k=\pm 2\pi$.

    The asymptotic behaviour \eqref{eq:gamma-sharp-asymptotics} follows from the fact that $\ell^{-1}w(\ell^{-1}\dummy) \to - m_w\delta_0$ in the sense of distributions as $\ell \downarrow 0$. In fact, we have that
    \begin{align}
        \ell^{-2}\widehat{W}_{\gamma=1,\ell}(2\pi) = \int_\R w\bra*{x} \cos(2\pi\ell x) \dx x \to  \int_\R w(x)\dx x \qquad \text{as $\ell \downarrow 0$}.
         & \qedhere
    \end{align}
\end{proof}
\subsection{Explicit number of initial clusters for the Hegselmann--Krause prototype}
For the Hegselmann--Krause prototype \eqref{eq:w-def:Hegselmann--Krause-special}, \cref{lemma:gamma-sharp} yields $\gamma_\sharp \sim \frac{3}{2}\ell^{-2}$ as $\ell \downarrow 0$. To find the number of initial clusters for~\eqref{eq:w-def:Hegselmann--Krause-special}, we calculate $\widehat{W}_{\gamma=1,\ell}(k)=k^{-2} 2 \cos(k\ell) - k^{-3}\ell^{-1}{2 \sin(k\ell)}$, therefore
\begin{align}
    \psi(k)= -k^2 + 2\gamma\frac{ \sin(k\ell)}{k \ell} -  2 \gamma\cos(k\ell).
\end{align}
Viewing $\psi$ as a function on $(0,\infty)$, we have that for large enough $\gamma$, $\psi(k)$ attains its maximum on $(0,\infty)$ in the interval $[0, \frac{2 \pi}{\ell}]$.
The derivative is given by
\begin{align}
    \psi'(k)
     & =
    -2k +
    \frac{2\gamma}{\ell}\bra*{
        \frac{\ell\cos(\ell k)}{k}
        - \frac{\sin(\ell k)}{k^2}
    }
    +  2\gamma \ell\sin(k\ell).
\end{align}
Setting $\psi'(k)=0$ leads to the equation
\begin{align}
    \frac{1}{\gamma} = \frac{1}{\ell} \bra*{ \frac{\ell \cos(\ell k)}{k^2} - \frac{\sin(\ell k)}{k^3}}
    +  \ell \frac{\sin(k\ell)}{k}.
\end{align}
As $\gamma \to\infty$, the limit $y:=\lim_{\gamma \to \infty} {\ell k_{\rm max}(\gamma)}$ solves the transcendental equation
\begin{align}
    0 = y \cos y - \sin y + y^2 \sin y \Leftrightarrow \tan y = \frac{y}{1- y^2} .
\end{align}
The smallest positive solution of this equation is at $y_\ast \approx 2.74$, so that
\begin{align}
    k_{\rm max} \sim \frac{y_\ast}{\ell} \qquad \text{as } \gamma \to \infty.
\end{align}
In particular, the number of initial clusters is given by
\begin{align}
    \label{eq:number-of-clusters:HK}
    n_{\rm clusters} \sim \frac{y_\ast}{2\pi} \frac{1}{\ell} \approx 0.44 \frac{1}{\ell} \qquad \text{as } \gamma \to \infty .
\end{align}

\section{Stationary and metastable states}
\label{sec:stationary-states}

Having analysed the initial formation of clusters in the previous section, we now want to understand the stationary states for the mean-field PDE~\eqref{eq:McKean-Vlasov-PDE}. While the particle system~\eqref{eq:microscopic-system} has a unique stationary state given by
\begin{align}
    \rho^N_{\infty} =
    \frac{e^{- \mathbf{W}_{\gamma,\ell}}}{Z_N} \in \mathcal{P}(\T^N),
    \qquad \text{where $Z_N:= \int_{\T^N} e^{- \mathbf{W}_{\gamma,\ell}(x^1, \dots, x^N)} \dx x^1 \dots \dx x^N$,}
\end{align}
the PDE~\eqref{eq:McKean-Vlasov-PDE} exhibits phase transitions and can have multiple stationary states which are characterized by~\cite[Proposition 2.4]{CGPS20}, see also~\cite{CP10}. In particular, a state $\rho$ is stationary for~\eqref{eq:McKean-Vlasov-PDE} if and only if it is a fixed point of the Kirkwood--Monroe map, that is, it satisfies
\begin{align}
    \label{eq:kirkwood-monroe}
    \rho = \frac{1}{Z(\rho)}e^{-W_{\gamma,\ell}\ast \rho} \in \mathcal{P}(\T), \qquad \text{where } Z(\rho):= \int_{\T} e^{-(W_{\gamma,\ell}\ast\rho)(x)}\dx x.
\end{align}
Note that the uniform state $\rhounif$ is always stationary. As proven in~\cite[Theorem~2.3]{CGPS20}, under reasonable assumptions on the interaction potential, all steady states are smooth and strictly positive.
In addition, any fixed point of~\eqref{eq:kirkwood-monroe} is a critical point of the free energy $\mathcal{F}_{\gamma,\ell}$ from~\eqref{eq:def:free_energy}.

Simulations suggest that for generic initial conditions, the PDE converges to a single-cluster state as $t\to\infty$ unless $\gamma$ is small enough, in which case the PDE converges to the uniform state $\rhounif$. This is in line with the fact that the mean-field PDE has a discontinuous phase transition at some $\gamma_c$, so that for $\gamma > \gamma_c$, there exist global minimisers of the free energy which are non-uniform (see~\cref{sec:discontinuous-phase-transition} below for a proof). Moreover, in~\cref{thm:glob-min-sym-decreasing} we prove that all global minimisers are either the uniform state or single-cluster states in the sense that they are symmetrically decreasing. The argument generalizes the strict version of the Riesz rearrangement inequality from~\cite{beckner1993sharp} to compactly supported kernels, crucially using the analyticity of stationary states which we show in~\cref{lemma:analyticity-stationary-states}.
A heuristic argument as well as numerical results as explained in~\cref{sec:stationary-states:approximate-gaussian} suggest that for large $\gamma$, the global minimiser of the free energy is in fact approximately (a circular version of) a Gaussian with variance $\sigma^2 = \frac{\ell}{\gamma w''(0)}$.

\begin{figure}
    \includegraphics[width=0.8\textwidth]{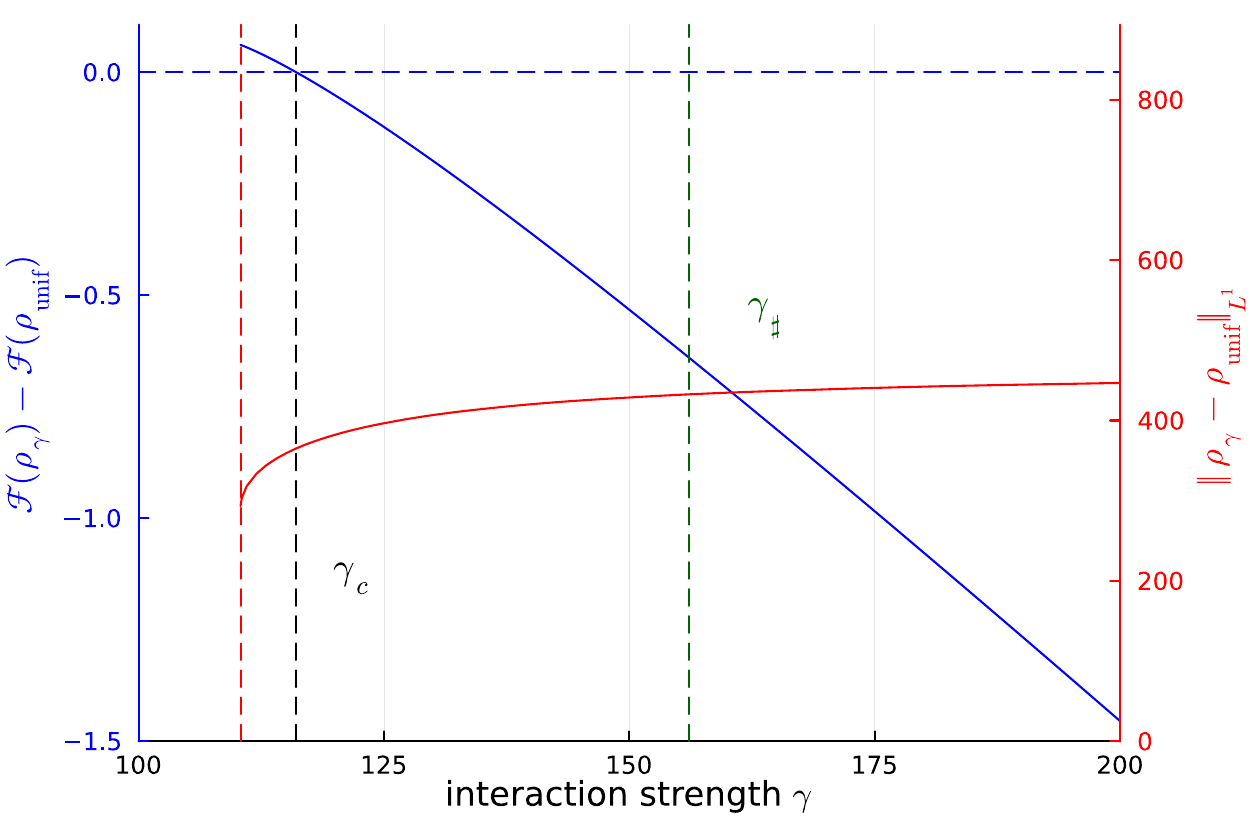}
    \caption{Bifurcation diagram for the Hegselmann--Krause potential~\eqref{eq:w-def:Hegselmann--Krause-special}.
    Holding $\ell=0.1$ fixed, the bifurcation diagram shows the single-cluster state $\rho_\gamma$ as a function of $\gamma$, continuing the single-cluster branch until it ceases to exist. The stationary state $\rho_\gamma$ is computed by using Newton's method to find fixed points of the Kirkwood--Monroe map. 
    The blue curve is the free energy gap $\cF_{\gamma,\ell}(\rho_\gamma) - \cF_{\gamma,\ell}(\rhounif)$, where $\rhounif\equiv  1$ is the uniform stationary state. The red curve shows the $L^1$-norm of $\rho_\gamma - \rhounif$. At $\gamma = \gamma_c$, there is a discontinuous phase transition: The uniform state $\rhounif$ is the unique global minimiser for $\gamma\in[0,\gamma_c)$ and becomes a local minimizer for $\gamma\in (\gamma_c , \gamma_\sharp)$, with $\gamma_\sharp$ the point of linear stability. Likewise, the single-cluster branch $\rho_\gamma$ exists for $\gamma<\gamma_c$ close to $\gamma_c$, where it is a local minimizer and it becomes the global minimiser only for $\gamma > \gamma_c$.}
    \label{fig:bifurcation-diagram}

\end{figure}

\subsection{Existence of a discontinuous phase transition}\label{sec:discontinuous-phase-transition}

We show that in the strong local interaction regime $\ell \ll 1$ and $\gamma \gg 1$, the free energy defined in~\cref{eq:def:free_energy} has a discontinuous phase transition, which is defined as follows~\cite{CP10,CGPS20}.
\begin{defn}[Transition point] \label{def:transitionpoint}
    A parameter value $\gamma_c >0$ is said to be a \emph{transition point} of~$\mathcal{F}_{\gamma,\ell}$ (for fixed $\ell>0$) from the uniform measure $\rhounif(\dx{x}) =  1 \dx{x}$ if it satisfies
    \begin{enumerate}
        \item For $0<\gamma< \gamma_c$, $\rhounif$ is the unique minimiser of $\mathcal{F}_{\gamma,\ell}$\,.
        \item For $\gamma=\gamma_c $, $\rhounif$ is a minimiser of $\mathcal{F}_{\gamma,\ell}$\,.
        \item For $\gamma>\gamma_c$, there exists $\rho_{\gamma} \in \cP(\T)\setminus \set{\rhounif}$, such that $\rho_\gamma$ is a minimiser of $\mathcal{F}_{\gamma,\ell}$\,.
    \end{enumerate}
    Additionally, a transition point  $\gamma_c >0$ is said to be \emph{continuous} for $\mathcal{F}_{\gamma,\ell}$ if
    \begin{enumerate}
        \item For $\gamma=\gamma_c$, $\rhounif$ is the unique minimiser of $\mathcal{F}_{\gamma,\ell}$\,.
        \item Given any family of minimisers $\{\rho_\gamma | \gamma> \gamma_c \}$, it holds that
              \begin{align}
                  \limsup_{\gamma \downarrow \gamma_c} \norm*{\rho_\gamma-\rhounif}_{\mathrm{TV}}=0 \, .
              \end{align}
    \end{enumerate}
    A transition point $\gamma_c$ which is not continuous is said to be \emph{discontinuous}.
\end{defn}
As shown in~\cite{CP10,CGPS20,DGPS23}, for small $\gamma$, the uniform state $\rhounif$ is the unique stationary state.
Also, for small enough $\ell$, we have a discontinuous phase transition~\cite[Def.\ 5.4]{CGPS20} as the following result shows.
\begin{proposition}
        \label{prop:discontinuous}
        Let $w\colon \R \to \R$ have compact support and moreover assume that $\int_\R w(x)\dx x <0$. Then, for all sufficiently small $\ell>0$, the potential $W_{\gamma,\ell}$ defined in~\eqref{eq:def:rescaledW} has a discontinuous transition point $\gamma_c< \gamma_\sharp$. In particular, for $\gamma>\gamma_c$ there exists a global minimiser $\rho_{\gamma,\ell}$ of $\calF_{\gamma,\ell}$ distinct from the uniform state $\rhounif(\dx{x}) = 1$.
    \end{proposition}
    In particular,~\cref{prop:discontinuous} holds true under~\cref{w-assumption} for $\ell$ sufficiently small. The existence of two energy minimisers for $\gamma=\gamma_c$ is illustrated in~\cref{fig:bifurcation-diagram}. Below in~\cref{subsec:gamma-crit}, we will prove that
    \begin{align}
        \limsup_{\ell \downarrow 0} \frac{\gamma_c}{2\Delta^{-1} \ell^{-1} \log\bra*{\ell^{-1}}} \le 1,
    \end{align}
    where $\Delta = -w(0)>0$.
    The proof of~\cref{prop:discontinuous} is based on the following result, which shows that sufficiently localised attractive interacting potentials lead to discontinuous phase transitions~\cite[Cor. 5.14]{CGPS20}.

    \begin{proposition}\label{prop:localized:discontinuous}~\cite[Corollary 5.14]{CGPS20}
        Let $\{W^{(n)} \}_{n \in \N}$ be a sequence of interaction potentials with $\|W^{(n)} \|_{L^1} = C >0$ for all $n \in \N$ such that $W^{(n)} \rightarrow - C \delta_0$ in the sense of distributions as $n \rightarrow \infty$.
        Then for $n$ large enough, the associated free energy
        \begin{align}
            \cF_{\gamma}^{(n)} (\rho) = \int_{\T} \rho(x) \log \rho(x) \dx x + \frac{1}{2} \int_{\T \times \T} W^{(n)}(x-y) \rho(x) \rho(y) \dx x \dx y
        \end{align}
        possesses a discontinuous transition point at some $\gamma_c^{(n)} < \gamma_{\sharp}^{(n)}$, where the point of linear stability $\gamma_{\sharp}^{(n)}$ is defined by replacing $W_{\gamma=1,\ell}$ in the definition ~\eqref{eq:gamma-sharp} by $W^{(n)}$.
    \end{proposition}

    \begin{proof}[Proof of~\cref{prop:discontinuous}]
        Let $(\ell_n)_n\in\N$ be a sequence such that $\ell_n \downarrow 0$. Then, $W^{(n)}:= W_{\gamma/\ell_n^2,\ \ell_n}$ satisfies the requirement of~\cref{prop:localized:discontinuous}.
    \end{proof}

\subsection{Non-trivial global minimisers of the free energy are symmetrically decreasing}\label{ssec:Riesz}

In this section, we prove that global minimisers are either uniform or single-cluster states in the sense that they are translates of their symmetric decreasing rearrangement. In the following, $S^n:=\set{x\in \R^{n+1}\st \abs{x}=1}$ denotes the $n$-dimensional hypersphere. Note also that $S^1\cong \T$. For a measurable function $f:S^n\to [0,\infty)$, we denote by $f^\ast$ the symmetric decreasing rearrangement of $f$.
That is, if $\overrightarrow{n}=(1,0,\dots,0)\in S^n$ and $d:S^n\times S^n\to [0,\infty)$ denotes the geodesic distance on $S^n$, then $f^\ast$ is the (up to sets of measure zero) unique measurable function $f^\ast:S^n\to[0,\infty)$ which depends only on $d(\dummy, \overrightarrow{n})$, is non-increasing in $d(\dummy, \overrightarrow{n})$ and whose level sets have the same measures as the level sets of $f$, that is, $\{x \in S^n: f(x) > t\}$ and $\{x \in S^n: f^\ast(x) > t\}$ have the same mass for all $t>0$.

\begin{theorem}[Global minimisers of the free energy are symmetrically decreasing]
    \label{thm:glob-min-sym-decreasing}
    Assume that $w$ satisfies~\cref{w-assumption} and is additionally $C^2$ on $(-s_w, s_w)$. Fix $\gamma>0$ and $\ell \in (0, 1/(2s_w))$.
    Let $\rho$ be a global minimiser of the free energy~\eqref{eq:def:free_energy}.
    Then there exists a translation $T:\T \to \T$ such that $\rho = \rho^\ast \circ T$.
\end{theorem}
The above~\cref{thm:glob-min-sym-decreasing} provides a non-trivial result if we know that the global minimiser is non-trivial, that is different from the uniform state $\rhounif(\dx{x})=  1$, which is ensured by the existence of a transition point from~\cref{prop:discontinuous}.
\begin{corollary}[Single-cluster minimiser]
    Assume that $w:\R \to \R$ is as in~\cref{thm:glob-min-sym-decreasing}. Then for sufficiently small $\ell>0$ and $\gamma>\gamma_c$, the global minimiser $\rho_{\gamma,\ell}$ of $\calF_{\gamma,\ell}$ is a nontrivial single-cluster state.
\end{corollary}
The proof of~\cref{thm:glob-min-sym-decreasing} relies on the classical Riesz rearrangement inequality~\cite{baernstein1976spherical}.
\begin{theorem}[Riesz rearrangement inequality]
    \label{thm:riesz-rearrangement}
    Consider the hypersphere $S^n$ for some $n\ge1$. Let $K:[0, \pi]\to[0,\infty)$ be a decreasing function. Then, for all measurable non-negative functions $f,g$ on $S^n$ it holds that
    \begin{align}
        \int_{S^n\times S^n} f(x) g(y) K(d(x,y)) \dx x \dx y
        \le \int_{S^n\times S^n} f^\ast(x) g^\ast(y) K(d(x,y)) \dx x \dx y.
    \end{align}
\end{theorem}

For the proof of~\cref{thm:glob-min-sym-decreasing}, we need to characterize the equality case in~\cref{thm:riesz-rearrangement}. For strictly symmetrically decreasing kernels $K$, this was done in~\cite{beckner1993sharp,burchard1994cases}, see also~\cite{lieb1977existence}. Since the interaction kernels $W_{\gamma,\ell}$ considered here are not strictly symmetrically decreasing, we prove the following variant of~\cite{beckner1993sharp,burchard1994cases}, which characterizes the equality case in~\cref{thm:riesz-rearrangement} for $f=g$ and kernels which can be constant outside a compact set.

\begin{lemma}[Variant of the strict Riesz rearrangement inequality]
    \label{lemma:riesz-rearrangement-equality-variant}
    Assume that $K:[0,\pi] \to [0,\infty)$ is decreasing, i.e., $K(x)\ge K(y)$ if $x\le y$ and that there exists $\delta >0$ such that
    \begin{align}
        \label{eq:assumption:K-strictly-decreasing-around-zero}
        K(x) > K(y)   \quad \text{whenever} \quad 0\le x < y \le \delta.
    \end{align}
    Let $\rho:S^n\to[0,\infty)$ be an analytic function\footnote{Analytic functions $f\colon S^n\to \R$ can be defined e.g.\ through two antipodal stereographic projections, which form a holomorphic atlas.} for which the equality
    \begin{align}
        \int_{S^n\times S^n} \rho(x) \rho(y) K(d(x,y)) \dx x \dx y
        = \int_{S^n\times S^n} \rho^\ast(x) \rho^\ast(y) K(d(x,y)) \dx x \dx y
    \end{align}
    holds.
    Then there exists a rotation $T:S^n\to S^n$ such that $\rho = \rho^\ast \circ T$.
\end{lemma}

\begin{proof}
    We use a variant of Beckner's polarization argument~\cite{beckner1993sharp}, see also~\cite{burchard2009short}. Let $\overrightarrow{n} = (1,0\dots,0)$. Let $H \subset \R^{n+1}$ be a hyperplane through the origin which does not pass through $\overrightarrow{n}$. The hyperplane $H$ divides $S^n$ into two parts. Let $M^+$ be the connected component of $S^n\setminus H$ that contains $\overrightarrow{n}$ and $M^-$ be the other connected component. Let $\sigma:S^n\to S^n$ be the reflection with respect to $H$ and $M^0=S^n\cap H$. We define the polarization $\rho^\sigma:S^n\to[0,\infty)$ by
    \begin{align}
        \rho^\sigma(x) := \begin{cases}
                              \max\{\rho(x), \rho(\sigma x)\} & \text{if } x\in M^+  \\
                              \rho(x)                         & \text{if } x\in M^0  \\
                              \min\{\rho(x), \rho(\sigma x)\} & \text{if } x\in M^-.
                          \end{cases}
    \end{align}

    Similarly to~\cite[Lemma~2.11]{burchard2009short}, it follows that if
    \begin{align}
        \label{eq:rho-is-rho-sigma-or-rho-circ-sigma}
        \rho = \rho^\sigma \quad \text{or} \quad \rho = \rho^\sigma \circ \sigma \qquad \text{for all reflections $\sigma$},
    \end{align}
    then there exists a rotation $T:S^n\to S^n$ such that $\rho = \rho^\ast \circ T$.

    Now, fix some reflection $\sigma$. Similarly to~\cite[Lemma 2.6]{burchard2009short} (see also~\cite{beckner1993sharp}), we write
    \begin{align}
        I(\rho) & :=\int_{S^n\times S^n} \rho(x) \rho(y) K(d(x,y)) \dx x \dx y
        \\
                & =   \int_{M^+\times M^+} \pra[\Big]{\rho(x)\rho(y)+\rho(\sigma x)\rho(\sigma y)} K(d(x,y)) \dx x \dx y
        \\&\qquad + \int_{M^+\times M^+} \pra[\Big]{\rho(x)\rho(\sigma y)+\rho(\sigma x)\rho(y)} K(d(x,\sigma y)) \dx x \dx y
        \\&=
        \int_{M^+\times M^+} \pra[\Big]{\rho(x)\rho(y)+\rho(\sigma x)\rho(\sigma y)} \set[\Big]{K(d(x,y))- K(d(x, \sigma y))}\dx x\dx y
        \\ & \qquad \begin{multlined}
            + \int_{M^+\times M^+}
            \Bigl[\rho(x)\rho(y)+  \rho(\sigma x)\rho(\sigma y)+ \rho(x)\rho(\sigma y) \\
                \qquad  + \rho(\sigma x)\rho(y)\Bigr]K(d(x, \sigma y)) \dx x \dx y.
        \end{multlined}
    \end{align}
    If we replace $\rho$ by $\rho^\sigma$, the second integral does not change.
    Fixing $x,y\in M^+$, we examine the effect of replacing $\rho$ by $\rho^\sigma$ in the first integral. If $\rho(x)\ge \rho(\sigma x)$ and $\rho(y) \ge \rho(\sigma y)$, or if $\rho(x)\le \rho(\sigma x)$ and $\rho(y) \le \rho(\sigma y)$, then the integrand $\rho(x)\rho(y) + \rho(\sigma x)\rho(\sigma y)$ does not change. Assume that $\rho(x) \ge \rho(\sigma x)$ and $\rho(y) < \rho(\sigma y)$. Then, the integrand increases by
    \begin{multline}
        \rho(x) \rho(\sigma y) + \rho(\sigma x)\rho(y)
        - \rho(x)\rho(y) - \rho(\sigma x)\rho(\sigma y) \\
        = \bra[\big]{
            \rho(\sigma x) - \rho(x)
        }
        \bra[\big]{
            \rho(y) - \rho(\sigma y)
        } \ge 0. \nonumber
    \end{multline}
    The same holds true in case $\rho(x) < \rho(\sigma x)$ and $\rho(y) \ge \rho(\sigma y)$.
    We obtain
    \begin{equation}\label{eq:polarization-increase}
        \begin{multlined}
            I(\rho^\sigma) - I(\rho)
            ={} \int_{M^+\times M^+} \max\set[\Big]{\bra[\big]{
                    \rho(\sigma x) - \rho(x)
                }
                \bra[\big]{
                    \rho(y) - \rho(\sigma y)
                }, 0} \, \cdot
            \\
            \set[\Big]{K(d(x,y))- K(d(x, \sigma y))}
            \dx x \dx y.
        \end{multlined}
    \end{equation}
    From $(\rho^\sigma)^\ast = \rho^\ast$ and Riesz' rearrangement inequality~\cref{thm:riesz-rearrangement}, it follows that $I(\rho) \le I(\rho^\sigma) \le I((\rho^\sigma)^\ast)=I(\rho^\ast)$.
    Since $I(\rho)=I(\rho^*)$ by assumption, and $I(\rho) \leq I(\rho^\sigma) \leq I(\rho^*)$, it follows that $I(\rho^\sigma) = I(\rho)$.
    Hence, the integral on the right-hand side of~\eqref{eq:polarization-increase} is zero.
    Since $d(x, \sigma y) < d(x,y)$ for all $x,y\in M^+$ and $K$ is a decreasing function, the integrand in~\eqref{eq:polarization-increase} is non-negative everywhere, therefore it is zero everywhere. Moreover, if $x,y\in M^+$ with $d(x,y)<\delta$, we have from
    $d(x,y) < d(x,\sigma y)$ and~\eqref{eq:assumption:K-strictly-decreasing-around-zero} that
    \begin{align}
        K(d(x,y)) > K(\min\set{\delta, d(x,\sigma y)}) \ge K(d(x,\sigma y)),
    \end{align}
    so we obtain in particular
    \begin{align}
        \label{eq:riesz-rearrangement-equality-variant-condition}
        \bra[\big]{
            \rho(x) - \rho(\sigma x)
        }
        \bra[\big]{
            \rho(y) - \rho(\sigma y)
        } \ge 0 \quad \text{for all } x,y\in M^+ \quad\text{with } d(x,y) < \delta.
    \end{align}
    The idea is now that \eqref{eq:riesz-rearrangement-equality-variant-condition} implies that the function $\rho - \rho\circ \sigma$ cannot change sign within balls of radius $\delta$. Since $\rho$ is analytic, this will imply that $\rho -\rho\circ \sigma$ cannot change sign globally on $M^+$.

    To make this precise, define $A:= \{x\in M^+: \rho(x) < \rho(\sigma x)\}$ and $B:= \{y\in M^+: \rho(y) > \rho(\sigma y)\}$.
    If $A$ is empty, then $\rho \ge \rho\circ\sigma$ on $M^+$ which implies $\rho = \rho^\sigma$. On the other hand, if $B$ is empty, then $\rho \le \rho\circ\sigma$ on $M^+$ which implies $\rho = \rho^\sigma \circ \sigma$. Below, we will show that it is not possible for $A$ and $B$ to be both non-empty. Together with~\eqref{eq:rho-is-rho-sigma-or-rho-circ-sigma}, this implies~\cref{lemma:riesz-rearrangement-equality-variant}. 

    Assume now that both $A$ and $B$ are non-empty.
    If $x\in A$ and $y\in B$, then $d(x,y) \ge \delta$ by~\eqref{eq:riesz-rearrangement-equality-variant-condition}.
    Let $x_0\in \overline{A}$ and $y_0 \in \overline{B}$ such that $\inf\set*{d(x,y): x\in A, y\in B} = d(x_0,y_0)$. Then, $d(x_0, y_0)\ge \delta$ and the intersection of the open balls $B(x_0, d(x_0,y_0))$ and $B(y_0, d(x_0,y_0))$ is a subset of $\{\rho = \rho\circ\sigma\}$. In particular, $\{\rho = \rho\circ\sigma\}$ contains an open ball. By the identity theorem for real analytic functions, we then have $\rho = \rho\circ\sigma$ on $M^+$, which implies that $A=B=\emptyset$, which is a contradiction.
\end{proof}

\begin{lemma}[Analyticity of stationary states]
    \label{lemma:analyticity-stationary-states}
    Assume that $w$ satisfies~\cref{w-assumption} and is additionally $C^2$ on $(-s_w, s_w)$. Fix $\gamma>0$ and $\ell \in (0, 1/(2s_w))$.
    Then, any stationary state $\rho$ of~\eqref{eq:McKean-Vlasov-PDE}  is analytic.
\end{lemma}

\begin{proof}
    ~\cref{w-assumption} and $s_w\ell < 1/2$ imply that $W_{\gamma,\ell} \in \mathcal W^{1,2}([0,1])$ where $\mathcal W^{1,2}([0,1])$ is the periodic Sobolev space. Therefore, by~\cite[Theorem~2.3]{CGPS20}, $\rho$ is smooth.
    By Lagrange's form of the remainder in Taylor's theorem, analyticity follows once we have shown  that for some $C>0$ we have
    \begin{align}
        \norm[\big]{\rho^{(n)}}_\infty \le C^{n+1} n! \qquad \text{for all } n\in\N.
    \end{align}
    Since $\rho$ is stationary, it is a fixed point of the Kirkwood--Monroe map~\eqref{eq:kirkwood-monroe}, that is,
    \begin{align}
        \label{eq:analyticity:kirkwood}
        \rho = \frac{1}{Z(\rho)} e^{-W_{\gamma,\ell}\ast \rho} \qquad \text{where } Z(\rho) = \int_{[0,1]} e^{-(W_{\gamma,\ell}\ast \rho)(x)}\dx x.
    \end{align}

    Observe that $\supp W_{\gamma, \ell} = [-\ell s_w, \ell s_w]$. Let $f:\R\to\R$ be a smooth $1$-periodic function.
    By partial integration, we obtain
    \begin{align}
        (W_{\gamma,\ell}\ast f)'(x) & =
        \int_0^1 W_{\gamma,\ell}(y) f'(x-y) \dx y
        = \int_{-\ell s_w}^{\ell s_w} W_{\gamma,\ell}(y) f'(x-y) \dx y
        \\
                                    & =  \pra[\Big]{ - W_{\gamma,\ell}(y) f(x-y) }_{-\ell s_w}^{\ell s_w}
        + \int_{-\ell s_w}^{\ell s_w} W_{\gamma,\ell}'(y) f(x-y) \dx y
        \\ & = \int_{-\ell s_w}^{\ell s_w} W_{\gamma,\ell}(y) f(x-y) \dx y.
    \end{align}
    Similarly, we have
    \begin{align}
        (W_{\gamma,\ell}\ast f)''(x)
         & = \pra[\Big]{ - W_{\gamma,\ell}'(y) f(x-y) }_{-\ell s_w}^{\ell s_w}
        + \int_{-\ell s_w}^{\ell s_w} W_{\gamma,\ell}''(y) f(x-y) \dx y.
    \end{align}
    In particular, there exists a constant $c_1>0$ (depending on $\gamma,\ell$ and $w$) such that we have for all smooth $1$-periodic functions $f:\R\to\R$ that
    \begin{align}
        \label{eq:analyticity:proof:conv-bound}
        \norm*{W_{\gamma,\ell}\ast f''}_\infty \le c_1 \norm{f}_\infty \quad \text{and} \quad
        \norm*{W_{\gamma,\ell}\ast f'}_\infty \le c_1 \norm{f}_\infty.
    \end{align}

    By Faà di Bruno's formula~\cite{johnson2002curious}, the $n$-th derivative of the composition of two $n$ times differentiable functions $f$ and $g$ can be written as
    \begin{align}
        \label{eq:faa-di-bruno}
        (g\circ f)^{(n)}(x)
        = \sum \frac{n!}{m_1! \cdot m_2! \cdots m_n!}
        g^{(k)}(f(x)) \prod_{j=1}^n \bra*{\frac{f^{(j)}(x)}{j!}}^{m_j},
    \end{align}
    where the sum is over all non-negative integers $m_1, \dots, m_n$ such that $1\cdot m_1 + 2\cdot m_2 + \dots + n\cdot m_n = n$.
    Define
    \begin{align}
        C := \max\set*{1, \frac{\norm{\rho}_\infty}{Z(\rho)}, \norm{\rho}_\infty, \norm{\rho'}_\infty, c_1, c_1 \norm{\rho}_\infty}.
    \end{align}
    Assume inductively that for some $n\ge 2$ it holds for all $j=0,\dots, n-1$ that
    \begin{align}
        \norm[\big]{\rho^{(j)}}_\infty \le C^{j+1} j!.
    \end{align}
    By~\eqref{eq:analyticity:kirkwood} and Faà di Bruno's formula~\eqref{eq:faa-di-bruno}, we obtain that
    \begin{align}
        \rho^{(n)} = \sum \frac{n!}{m_1! \cdot m_2! \cdots m_n!}
        \frac{\rho}{Z(\rho)} \prod_{j=1}^n \bra*{\frac{-(W_{\gamma,\ell}\ast \rho)^{(j)}}{j!}}^{m_j}.
    \end{align}
    By the induction hypothesis and the first part of~\eqref{eq:analyticity:proof:conv-bound}, we have for all $j=2, \dots, n$
    \begin{align}
        \norm[\big]{ W_{\gamma, \ell} \ast \rho^{(j)}}_\infty \le c_1 \norm[\big]{\rho^{(j-2)}}_\infty \le c_1 C^{j-1} (j-2)! \le C^{j} (j-1)!.
    \end{align}
    By the second part of~\eqref{eq:analyticity:proof:conv-bound}, this formula holds for $j=1$ as well. We obtain
    \begin{align}
        \norm[\big]{\rho^{(n)}}_\infty
         & \le \sum \frac{n!}{m_1! \cdot m_2! \cdots m_n!} \frac{\norm*{\rho}_\infty}{Z(\rho)}
        \prod_{j=1}^n \frac{\norm*{W_{\gamma,\ell}\ast \rho^{(j)}}_\infty^{m_j}}{j!^{m_j}}     \\
         & \le
        \sum \frac{n!}{m_1! \cdot m_2! \cdots m_n!} \frac{\norm*{\rho}_\infty}{Z(\rho)}
        \prod_{j=1}^n \frac{
        C^{j m_j}(j-1)!^{m_j}}{j!^{m_j}}                                                       \\
         & \le
        \frac{\norm*{\rho}_\infty}{Z(\rho)} C^n \mathbf{Y}_n\bra*{0!, 1!, \dots, (n-1)!},
        \label{eq:rho-nth-derivative-bound}
    \end{align}
    where the \textit{complete Bell polynomial} $\mathbf{Y}_n$ is given by~\cite[Section~3.3]{comtet2012advanced}
    \begin{align}
        \mathbf{Y}_n(x_1, \dots, x_n) := \sum \frac{n!}{m_1!\cdot m_2! \cdots m_n!} \prod_{j=1}^n \frac{x_j^{m_j}}{j!^{m_j}}.
    \end{align}
    By~\cite[Section 3.3, (3i)]{comtet2012advanced}, we obtain
    \begin{equation}
        \label{eq:bell-poly-at-factorials}
        \mathbf{Y}_n(0!, 1!, \dots, (n-1)!) = \sum_{k=1}^n \abs*{s(n,k)},
    \end{equation}
    where $s(n,k)$ are the Stirling numbers of the first kind~\cite[Section~1.14]{comtet2012advanced}. By~\cite[Section~1.14, (14p)]{comtet2012advanced}, the sum~\eqref{eq:bell-poly-at-factorials} equals $n!$. In particular, we obtain from~\eqref{eq:rho-nth-derivative-bound} that
    \begin{align}
        \norm{\rho^{n}}_\infty \le C^{n+1} n!,
    \end{align}
    which proves the claim.
\end{proof}

\begin{proof}[Proof of~\cref{thm:glob-min-sym-decreasing}]
    Define the entropy and the interaction energy by
    \begin{align}
        S(\rho) := \int_{[0,1]} \rho(x) \log(\rho(x)) \dx x \quad \text{and} \quad
        I(\rho) := \iint_{[0,1]^2}  \rho(x) \rho(y) W_{\gamma, \ell}(x-y) \dx x \dx y,
    \end{align}
    respectively. The free energy~\eqref{eq:def:free_energy} is then given by the sum $\mathcal F(\rho) = S(\rho) + I(\rho)$. The entropy depends only on the masses of the levels sets of $\rho$, so $S(\rho) = S(\rho^\ast)$. Also note that we are minimizing the free energy and that $W_{\gamma,\ell}(x)$ is increasing as $x$ increases from $0$ to $\frac{1}{2}$, so we can apply~\cref{thm:riesz-rearrangement} to $-I(\rho)$ to obtain $I(\rho) \ge I(\rho^\ast)$. Since $\rho$ is a global minimiser of $\mathcal{F}$, this implies that
        $I(\rho)=I(\rho^\ast)$. By~\cref{lemma:analyticity-stationary-states}, $\rho$ is analytic.
    From~\cref{w-assumption} we have that $w''(0)>0$ which implies~\eqref{eq:assumption:K-strictly-decreasing-around-zero}, so by applying~\cref{lemma:riesz-rearrangement-equality-variant}  to $-I(\rho)$, we obtain the claim.
\end{proof}

\begin{remark}[Generalization to higher dimensions]
    We expect a variant of~\cref{thm:glob-min-sym-decreasing} to hold true on $S^d$ and on $\mathbb{T}^d$ 
    for $d>1$ as well. The key part to prove this is to generalize the analyticity result~\cref{lemma:analyticity-stationary-states} to multiple dimensions, something which we believe can be done for example using the multivariate generalization of Faà di Bruno's formula from~\cite{Hardy2006}.
\end{remark}

\subsection{Approximate Gaussian steady states}
\label{sec:stationary-states:approximate-gaussian}

Numerically, we observe that the single-cluster states in the PDE are Gaussians with variance $\sigma^2 = \frac{\ell}{\gamma w''(0)}$. This is in line with the following heuristic argument. Under~\cref{w-assumption},
we can use Laplace's method around $x= 0$ to find
\begin{align}
    \label{eq:kirkwood-monroe:Laplace}
    e^{-(W_{\gamma,\ell}\ast \rho)(x)} \sim
    e^{-\gamma \ell w(0)- \gamma w''(0) (x^2\ast \rho)(x)/\ell} \qquad \text{as }\gamma \to \infty.
\end{align}
Therefore, a Gaussian with variance $\sigma^2=\frac{\ell}{\gamma w''(0)}$ satisfies the Kirkwood--Monroe fixed point equation~\eqref{eq:kirkwood-monroe} approximately for large $\gamma$.
Therefore, we expect the single-cluster state to have variance $\sigma^2$, which is confirmed by our numerical simulations (see~\cref{fig:single-cluster-state}).
That single-cluster states are approximately Gaussian can also be seen through the particle system by approximating the dynamics of the particles within a cluster by an Ornstein--Uhlenbeck process as in~\cite[Setion~5.3]{GarnierEtAl2017}.

\begin{figure}
    \includegraphics[width=0.8\textwidth]{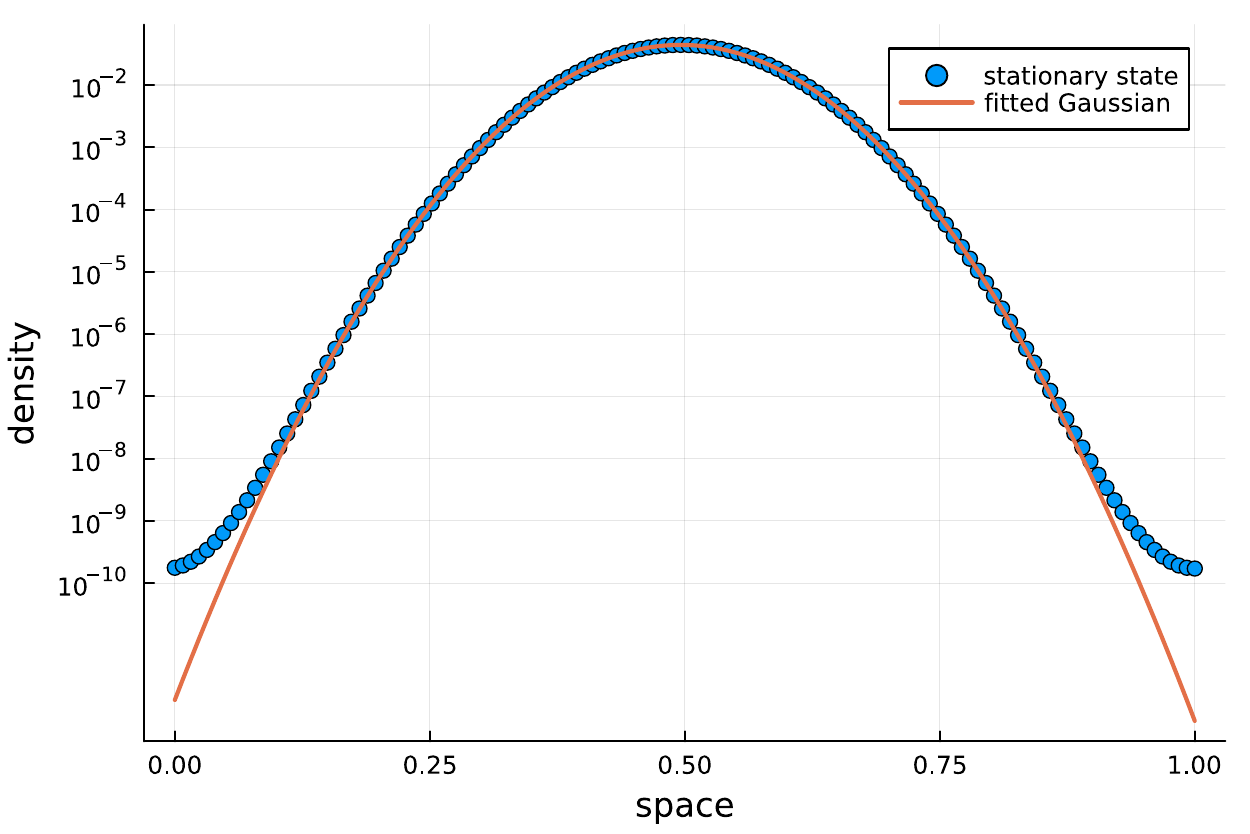}
    \caption{Approximate Gaussian single-cluster states are stationary for $\gamma \gg 1$ (logarithmic $y$-axis). \textbf{Blue:} A numerical steady state of~\eqref{eq:McKean-Vlasov-PDE}, computed as a fixed point of the Kirkwood--Monroe map~\eqref{eq:kirkwood-monroe}  using Newton's method. The potential is given as in~\eqref{eq:w-def:Hegselmann--Krause-special} with parameters~$\gamma=100$ and~$\ell= 0.5$. \textbf{Red:} A Gaussian least-squares fit to the numerical steady state. Its variance is~$\sigma^2 = 0.0051$ which approximately equals $\frac{\ell}{\gamma}=0.005$.}
    \label{fig:single-cluster-state}
\end{figure}

\subsection{Upper bound on the critical interaction strength \texorpdfstring{$\gamma_c$}{}}
    \label{subsec:gamma-crit}
    Next, we give a rigorous upper bound on the critical interaction strength $\gamma_c$ by comparing the free energy of the uniform state with the free energy of a rectangular function.

    \begin{proposition}
        \label{prop:gamma-crit-bound}
        Assume that $w$ satisfies~\cref{w-assumption} and that $\Lambda>0$ is such that
        \begin{align}
            \label{eq:w-quadratic-upper-bound}
            w(x) \le  -\Delta + \frac{\Lambda}{2} x^2 \quad \text{for all } x \in \R,
        \end{align}
        where $\Delta := -w(0) > 0$.
        (For example, for~\eqref{eq:w-def:Hegselmann--Krause-special} we can take $\Lambda=1$.)
        Then, for all $\ell>0$ such that $\ell<\min\set*{\frac{1}{2},\frac{\Delta}{m_w}, \frac{1}{2s_w}, \sqrt{\frac{\Lambda}{12 \Delta}}}$ where $m_w = -\int_\R w(x) \d x>0$ , we have that
        \begin{align}
            \label{eq:gamma-sharp-heuristics:less}
            \gamma_c \le \gamma_{c+}(\ell) := - \frac{W_{-1}\bra*{-\lambda(\ell)\ell\exp(c(w))}}{\lambda(\ell)},
        \end{align}
        where $\lambda(\ell):= \ell \Delta - \ell^2 m_w>0$ and $c(w):= \log\bra*{\frac{12}{\Lambda}}-1$ and $W_{-1}$ is the lower branch of the Lambert W function.
        In particular,
        \begin{align}
            \label{eq:gamma-sharp-heuristics:limsup}
            \limsup_{\ell \downarrow 0} \frac{\gamma_c}{2\Delta^{-1}\ell^{-1}\log\bra*{\ell^{-1}}} \le 1.
        \end{align}
    \end{proposition}

    \begin{proof}
        For simplicity, we work with the domain $[-\tfrac12,\tfrac12]$ rather than $[0,1]$. For some fixed $a\in(0,\ell/2)$, we consider the rectangular function $\varrho_a:[-\tfrac12,\tfrac12]\to[0,\infty)$ defined by
        \begin{align}
            \varrho_a(x) := \begin{cases}
                                \frac{1}{2a} & \text{if } - a \le x \le  a, \\
                                0            & \text{otherwise.}
                            \end{cases}
        \end{align}
        We want to compare the free energy of $\varrho_a$ with the free energy of the uniform state $\rhounif$.
        The entropy of $\varrho_a$ is given by
        \begin{align}
            \int_{-\frac12}^{\frac12} \varrho_a(x) \log(\varrho_a(x)) \d x = - \log(2a).
        \end{align}
        Under the assumption that $a < \frac{\ell}{2}$, we can bound the interaction energy from above using \eqref{eq:w-quadratic-upper-bound} as follows:
        \begin{align}
                & \frac{1}{2}\iint_{[-\frac12,\frac12]^2} W_{\gamma, \ell}(x-y) \varrho_a(x) \varrho_a(y) \dx x \dx y = \frac{1}{8a^2}
            \iint_{[-a,a]^2} W_{\gamma, \ell}(x-y) \dx x \dx y                                                                                                                                  \\
            ={} & \frac{\gamma \ell}{8a^2} \iint_{[-a,a]^2} w\bra*{\frac{x-y}{\ell}} \dx x \dx y \le \frac{\gamma \ell}{8a^2} \iint_{[-a,a]^2} \bra*{-\Delta + \frac{\Lambda}{2} \frac{(x-y)^2}{\ell^2}} \dx x \dx y \\
            ={} &
            - \frac{\gamma \ell \Delta }{2} +
            \frac{\Lambda}{6} \frac{\gamma}{\ell}a^2.
        \end{align}
        In particular, we have that
        \begin{align}
            \calF_{\gamma,\ell}(\varrho_a) & \le - \log(2a) +
            \frac{\Lambda}{6} \frac{\gamma}{\ell}a^2 - \frac{\gamma \ell \Delta }{2}.
        \end{align}
        On the other hand, the free energy~\eqref{eq:def:free_energy} of the uniform state $\rhounif$ is given by
        \begin{align}
            \calF_{\gamma,\ell}(\rhounif) & = \frac{1}{2}  \int_{[0,1]^2} W_{\gamma,\ell}(x-y) \d x \d y = \frac{1}{2}  \int_{[0,1]} W_{\gamma,\ell}(z) \d z \\&= \frac{\gamma \ell}{2}\int_\R w\bra*{\frac{x}{\ell}} \d x = \frac{\gamma \ell^2}{2} \int_\R w(x) \d x = - \frac{\gamma \ell^2}{2} m_w,
        \end{align}
        where $m_w = -\int_\R w(x) \d x>0$ and we used that $W_{\gamma,\ell}$ is $1-$periodic and that $\ell s_w<\frac{1}{2}$.
        In particular, we have the upper bound on the free energy gap
        \begin{align}
            \label{eq:free-energy-gap-upper-bound-heuristics}
            \calF_{\gamma,\ell}(\varrho_a) - \calF_{\gamma,\ell} (\rhounif) \le - \log(2a) + \frac{\Lambda}{6} \frac{\gamma}{\ell}a^2 - \frac{\gamma \ell \Delta }{2} + \frac{\gamma \ell^2}{2} m_w.
        \end{align}
        The minimum of the right-hand side in  $a\in(0, \ell/2)$ is given by $a_0:=\sqrt{3\ell/(\Lambda \gamma)}$ and we obtain
        \begin{align}
            \calF_{\gamma,\ell}(\varrho_{a_0}) - \calF_{\gamma,\ell} (\rhounif) \le
            \frac{1}{2}\phi(\gamma, \ell) := \frac{1}{2}\bra*{-\gamma \lambda(\ell) + \log \gamma - \log \ell -c(w)},
        \end{align}
        where $\lambda(\ell):= \ell \Delta - \ell^2 m_w>0$ and $c(w):= \log\bra*{\frac{12}{\Lambda}} -1$.

        Note that $\phi(\gamma, \ell)  < 0$ is equivalent to
        \begin{align}
            \gamma\exp\bra*{-\gamma \lambda(\ell)} < \ell \exp(c(w)) \Leftrightarrow (-\gamma\lambda(\ell))\cdot\exp\bra*{-\gamma \lambda(\ell)}
        > -\lambda(\ell)\ell \exp(c(w)).
        \end{align}
        The function $x\mapsto x \exp(x)$ is strictly decreasing on $(-\infty, -1)$ with inverse function $W_{-1}:[-e^{-1}, 0)\to (-\infty, -1]$, the lower real branch of the Lambert $W$-function. Therefore, we have $\phi(\gamma,\ell) < 0$ for all $\gamma > \gamma_{c+}(\ell)$, where
        \begin{align}
            \gamma_{c+}(\ell) := - \frac{W_{-1}\bra*{-\lambda(\ell)\ell \exp(c(w))}}{\lambda(\ell)} > 0.
        \end{align}
        Note that $\lambda(\ell)\ell \exp(c(w)) = (\ell \Delta - \ell^2 m_w) \ell \frac{12}{\Delta} e^{-1} \le \ell^2 \frac{12 \Delta}{\Lambda} e^{-1} \le e^{-1}$ since we assume $\ell \le \sqrt{\Lambda/(12 \Delta)}$, so the argument of $W_{-1}$ is in $[-e^{-1},0)$ as required.
        By means of \eqref{eq:free-energy-gap-upper-bound-heuristics} and \cref{def:transitionpoint}, this proves~\eqref{eq:gamma-sharp-heuristics:less}.

        We will now prove that
        \begin{align}
            \label{eq:gamma-sharp-plus-asymp-equivalent}
            \lim_{\ell \downarrow 0} \frac{\gamma_{c+}(\ell)}{2 \Delta^{-1}\ell^{-1}\log\bra*{\ell^{-1}}} = 1,
        \end{align}
        which will then imply~\eqref{eq:gamma-sharp-heuristics:limsup}. To show \eqref{eq:gamma-sharp-plus-asymp-equivalent}, fix $\ell$ sufficiently small and let $x:= \lambda(\ell) \gamma_{c+}(\ell)$ and $y:= \lambda(\ell) \ell \exp(c(w))$. Then, $x\exp(-x) = y$ and so in particular
        \begin{align}
            \label{eq:lambertW-x-y}
            x = \log(y^{-1}) + \log(x).
        \end{align}
        Since $x\ge 1$ due to $-x = W_{-1}(-y)$, \eqref{eq:lambertW-x-y} implies that $x \ge \log(y^{-1})$.
        On the other hand, we have that
        \begin{align}
            3 \log(y^{-1}) \exp(-3\log(y^{-1})) = 3 \log(y^{-1}) y^3 < 3y^2 < y,
        \end{align}
        if $y$ is small enough. Since $t\mapsto t \exp(-t)$ is strictly decreasing on $(1,\infty)$, this implies that $x < 3 \log(y^{-1})$. Plugging this into~\eqref{eq:lambertW-x-y}, we obtain that $x \le \log(y^{-1}) + \log(\log(y^{-1})) + \log 3$. In particular, we have that
        \begin{align}
            -\frac{W_{-1}(-y)}{\log\bra*{y^{-1}}} = \frac{x}{\log\bra*{y^{-1}}} \in \bra*{1, 1+ \frac{\log(\log(y^{-1})) + \log 3}{\log(y^{-1})}}
        \end{align}
        which proves~\eqref{eq:gamma-sharp-plus-asymp-equivalent}.
    \end{proof}
    \begin{remark}
        Instead of approximating the single-cluster stationary states of the PDE with rectangular functions, one could alternatively use analogues of the Gaussian on the circle, such as the von Mises distribution or the wrapped Gaussian. While these choices would provide a more accurate representation of the cluster shape, the resulting upper bound for $\gamma_{c+}$ would remain of the same order in $\ell$, and the calculations would become more technical.
    \end{remark}

\subsection{Multi-cluster states}
\label{sec:multi-cluster-states}
When, after initial clustering, the $N$ particles in~\eqref{eq:microscopic-system} are divided into $K\in\N$ clusters of $N^{(1)}, \dots, N^{(K)}$ particles, respectively, we can assign to each cluster a mass $m_i=\frac{N^{(i)}}{N}$.
If the clusters are sufficiently localised, the particles within each cluster behave approximately like the system~\eqref{eq:microscopic-system} with effective interaction strengths $\gamma_i= \frac{N^{(i)}}{N} \gamma= m^{(i)} \gamma$.
Therefore, the particles within cluster $i$ are approximately normally distributed with variances $\sigma_i^2=\frac{\ell}{\gamma w''(0) m^{(i)}}$.
Moreover, the empirical density for a multi-cluster state with masses $m^{(1)}, \dots, m^{(K)}$ and cluster centres $X^{(1)}, \dots, X^{(K)}$ is approximately a Gaussian mixture
\begin{align}
    \label{eq:multi-cluster-state}
    \mu^N \approx \sum_{k=1}^{K} m^{(k)}g(\dummy; X^{(k)}, \sigma_k^2) =
    \sum_{k=1}^{K} m^{(k)} \frac{1}{\sqrt{2\pi \sigma^2_k}} \exp\bra*{-\frac{1}{2\sigma^2_k}{\bra*{\dummy-X^{(k)}}^2}},
\end{align}
where $g(x;\mu, \sigma^2)$ is the density of a normal distribution with mean $\mu$ and variance~$\sigma^2$.

In~\cref{sec:ReducedModel}, we will provide evidence that multi-cluster states are metastable states for the PDE~\eqref{eq:McKean-Vlasov-PDE} and that the stability timescale for a cluster of mass $m$ within a multi-cluster state is of order $e^{\gamma \ell \Delta m}$.

\subsection{The critical minimal mass for a cluster inside a multi-cluster state}
\label{sec:critical-mass}

Consider again a multi-cluster state in the particle system~\eqref{eq:microscopic-system}. 
In order for a cluster of mass $m\in(0,1)$ of this multi-cluster state to be stable at least for some time, the width $\sigma_{\rm cluster}= (\ell/(\gamma w''(0) m))^{\frac{1}{2}}$ needs to be small compared to the interaction range $2s_w\ell$ of the potential since otherwise the cluster would dissolve quickly. 
We define the critical minimal mass $\mcrit \in (0,1)$ as the smallest mass for which a cluster within a multi-cluster state does not dissolve immediately. Certainly, if the cluster width $\sigma_{\rm cluster}$ is much smaller than the interaction range $2s_w\ell$ of $W_{\gamma,\ell}$, we expect the cluster to be stable for some time. This corresponds to the condition
\begin{align}
    \sigma_{\rm cluster} \ll 2s_w\ell \quad \Rightarrow \quad m \gg \frac{1}{\gamma \ell},
\end{align}
giving the heuristic lower bound
\begin{align}
    \label{eq:m-crit}
    \mcrit \gtrsim \frac{1}{\gamma \ell}
\end{align}
on the critical minimal mass. The existence of a critical minimal mass $\mcrit$ can also be seen through the following scaling argument for the free energy. For all $m\in(0,1)$ and $\rho\in\calP(\T)$, we have that
\begin{align}
    \mathcal F_{\gamma,\ell}(m\rho) & = m \int \rho(x) (\log \rho(x) + \log(m))\dx x + m^2 \frac{\gamma \ell}{2} \iint w\bra*{\frac{x-y}{\ell}} \rho(x)\rho(y) \dx x\dx y \\
                                    & =m \mathcal F_{m\gamma, \ell}(\rho) + m \log m.
\end{align}
Therefore, whenever $\rho_{m,\gamma} \in \mathcal P(\T)$ is such that
\begin{align}
    \mathcal{F}_{\gamma,\ell}(m\rho_{m,\gamma}) =
    \inf_{\rho \in \mathcal{P}(\T)} \mathcal{F}_{\gamma, \ell}(m\rho),
\end{align}
we have that
\begin{align}
    \mathcal{F}_{m\gamma,\ell}(\rho_{m,\gamma}) = \inf_{\rho \in \mathcal{P}(\T)} \mathcal{F}_{m\gamma,\ell}(\rho).
\end{align}
Therefore, $\rho_{m,\gamma}$ can only be a stable cluster state provided that $m\gamma > \gamma_c$, where $\gamma_c$ is defined in~\cref{sec:discontinuous-phase-transition}.
This suggests that the critical mass $\mcrit \in (0,1)$ is of the form $\mcrit = \frac{\gamma_c}{\gamma}$, so that~\eqref{eq:m-crit} gives the heuristic lower bound $\gamma_c \gtrsim \ell^{-1}$. Recall that \cref{prop:gamma-crit-bound} gives a rigorous upper bound on $\gamma_c$ which is of the order $\ell^{-1}\log(\ell^{-1})$.

\section[The reduced model]{The reduced model: Coalescing heavy Brownian motions with mass exchange}\label{sec:ReducedModel}

In this section, we will first recall the Markovian dynamics of the clusters from~\cite{GarnierEtAl2017}. Then we will present a model for the mass exchange of clusters.

\subsection{The Markovian dynamics of the clusters in the particle model after~\texorpdfstring{\cite{GarnierEtAl2017}}{GPY17}}\label{ssec:model:Garnier}
We now describe the coalescence of the clusters for the particle model~\eqref{eq:microscopic-system}, following closely~\cite{GarnierEtAl2017}.
Assuming that $w$ satisfies~\cref{w-assumption}, we consider again the dynamics~\eqref{eq:microscopic-system} in the regime $\ell \ll 1$ and $\gamma \gg 1$.
Let us assume that initial clustering has already occurred (see~\cref{sec:initial-clustering}) so that the particles are organised in  $K(\tau_0)$ clusters at some time $\tau_0$. Let $I^j \subset \{1,\dots,N\}$ be the set of indices of the particles belonging to the $j$-th cluster.
Each cluster $j$ has then a mass of  $m^{(j)}=\abs{I^j}/N$ and a corresponding width $\sigma_j$ defined by $\sigma^2_j = \frac{\ell}{\gamma w''(0) m^{(j)}}$.
Since $w$ is even, it follows from~\eqref{eq:microscopic-system} that the cluster centres
\begin{align}
    X^{(j)}_t:= \frac{1}{\abs{I^j}}\sum_
    {i\in I^j} X^i_t 
\end{align}
follow independent Brownian motions
\begin{equation}
    \label{eq:clust_cent_mass}
    X^{(j)}_t = X^{(j)}_{\tau_0} + \frac{\sqrt{2}}{\sqrt{N m^{(j)}}} W^{(j)}_t,
\end{equation}
where $W^{(j)}_t := \frac{1}{\sqrt{\abs{I^j}}} \sum_{i\in I^j} W^i_t$ is a standard Brownian motion. The cluster centres follow~\eqref{eq:clust_cent_mass} until two clusters meet at time
\begin{align}
    \tau_1 = \inf \left\{t  > \tau_{0} \, : \, |X^{(j)} -  X^{(k)}(t)| = d \; \mbox{for some} \; j \neq k \right\},
\end{align}
where the characteristic length scale $d$ is given by $d \asymp \ell s_w$.
In this case, the two clusters $j$ and $k$ merge except for some set of realizations which is exponentially small in $\gamma$. It follows from~\eqref{eq:microscopic-system} that the merging process can be described approximately by the following SDE
\begin{align}
    \begin{split}
        \d \bra[\big]{
            X^{(j)}_t - X^{(k)}_t
        }
         & = - \frac{\gamma w''(0)}{\ell} \bra[\big]{ m^{(j)} + m^{(k)}}\bra[\big]{X^{(j)}_t - X^{(k)}_t} \d t  \\
         & \qquad+\frac{\sqrt{2}}{\sqrt{N m^{(j)}}} \d W^{(j)}_t - \frac{\sqrt{2}}{\sqrt{N m^{(k)}}} \d W^{(k)}_t.
    \end{split}
\end{align}
In particular, since the martingale term has quadratic variation of order $\mathcal{O}(t/N)$, for $N$ large enough, the drift term dominates and the cluster centres converge exponentially fast. The new cluster with index set $I^j \cup I^k$ has mass
\begin{align}
    m{(j)} + m{(k)}
\end{align}
and follows a Brownian motion of the form
\begin{align}
    X^{(j,k)}_t
    = X^{(j,k)}_{\tau_1} + \frac{\sqrt{2}}{\sqrt{N (m^{(j)} + m^{(k)})}} W^{(j,k)}_t,
\end{align}
where $W^{(j,k)}_t := \frac{1}{\sqrt{2}} \left( W^{(j)}_t + W^{(k)}_t \right)$ is a standard Brownian motion and the centre of the new cluster starts at some time $\tau_1$ at the weighted average
\begin{align}
    \label{eq:cluster-center-after-merging}
    X^{(j,k)}_{\tau_1}
    =  \frac{m^{(j)} X^{(j)}_{\tau_1}  + m^{(k)}X^{(k)}_{\tau_1} }{m^{(j)}+ m^{(k)}},
\end{align}
where we neglected the influence of the Brownian motions during the duration of merging as well as the time the merging takes.

Relabeling the clusters to take the merging into account, we now obtain $K(\tau_1)=K(\tau_0)-1$ clusters.
The coalescing process continues until a unique cluster with mass $1$ and variance $\frac{\ell}{\gamma w''(0)}$ forms.
Neglecting the set of realizations on which merging of touching clusters does not occur, or mass leakage occurs, it follows from standard properties\footnote{For a standard Brownian motion $W_t$ and some $\sigma>0$, the hitting time $\tau_{a,-b}:=\inf\set{t>0\st \sigma W_t \in \{a, -b\}}$ (with $a,b>0$) satisfies $\expect \pra*{\tau_{a,-b}}=ab/\sigma^2$. Therefore, since by \eqref{eq:clust_cent_mass} the difference of two cluster centres behaves (before merging) as $\sigma B_t$ with $\sigma^2 = \frac{2}{N}\bra*{\frac{1}{ m^{(j)}} + \frac{1}{ m^{(k)}}}$ and $B_t$ a standard Brownian motion, we see that for two clusters on a periodic domain $[0,1]$, the timescale for the two clusters to merge is of order $N$.} of the one-dimensional Brownian motion in a bounded domain that a unique cluster will form almost surely in finite time on a timescale of order $\tcoalesc\asymp N$.

\subsection{The extended Markovian dynamics with mass exchange}\label{ssec:model:MassExchange}

We now extend the model of the Markovian dynamics of the clusters by allowing the clusters to exchange mass.
Under~\cref{w-assumption}, we propose the following model, which we derive in~\cref{ssec:derivation}: Consider the regime $\gamma \ell \gg 1$ and $\ell\ll 1$.
We assume that at time $t$, there are $K(t)$ clusters with centres $X^{(1)}< \dots< X^{(K(t))}$ and masses $m^{(1)}, \dots, m^{(K(t))}\in(0,1)$ such that $\sum_{j=1}^{K(t)}m^{(j)}=1$.
The positions and the masses of the clusters behave as follows:
\begin{subequations}
    \label{eq:joint-model}
    \begin{empheq}[box=\fbox]{align}
        \label{eq:joint-model:X}
        \d X^{(j)}_t &= \frac{\sqrt{2}}{\sqrt{N m^{(j)}}} \d W^{(j)}_t,  \\
        \label{eq:joint-model:m}
        \frac{\d}{\d t} m^{(j)} &
        = \phi^{(r)}_{j-1} m^{(j-1)}
        + \phi^{(l)}_{j+1} m^{(j+1)}
        - \bra*{
            \phi^{(r)}_{j}+ \phi^{(l)}_{j}} m^{(j)
            },
    \end{empheq}
\end{subequations}
for $j=1,\dots, K(t)$, where the rates $\phi^{(r)}_j$ and $\phi^{(l)}_j$ at which cluster $j$ with mass $m^{(j)}$ loses mass to the cluster to the right and to the left, respectively, are given by
\begin{align}
    \phi^{(r)}_j := \phi^{(r)}(m^{(j)}, X^{(j-1)}, X^{(j)}, X^{(j+1)}), \quad  \phi^{(l)}_j := \phi^{(l)}(m^{(j)}, X^{(j-1)}, X^{(j)}, X^{(j+1)}),
\end{align}
where we define
\begin{align}
    \label{eq:def:phi-r}
    \phi^{(r)}(m, x_l, x, x_r)
    = \frac{1}{d_{\rm dir}(x, x_r) - 2 s_w\ell}
    \sqrt{\frac{e \gamma w''(0)}{2\pi \ell}}\sqrt{m}e^{-\gamma \ell \Delta m}
\end{align}
as well as
\begin{align}
    \label{eq:def:phi-l}
    \phi^{(l)}(m, x_l, x, x_r)
    = \frac{1}{d_{\rm dir}(x_l, x) - 2 s_w\ell}
    \sqrt{\frac{e \gamma w''(0)}{2\pi \ell}}\sqrt{m}e^{-\gamma \ell \Delta m}
\end{align}
Here, $\Delta =  - \inf w$ and $d_{\rm dir}$ denotes the directed distance on the torus $\T \simeq [0,1]$:
\begin{align}
    \label{eq:def:directed-distance}
    d_{\rm dir}(x_l, x_r) := \begin{cases}
                                 x_r - x_l    & \text{if } x_r \ge x_l, \\
                                 x_r +1 - x_l & \text{if } x_r < x_l.
                             \end{cases}
\end{align}
Throughout this paper, we will write $X^{(0)}:=X^{(K(t))}$ and $X^{(K(t)+1)}:=X^{(1)}$ in order to simplify the notation. This makes sense since our model is defined on the torus $\T$.

The model~\eqref{eq:joint-model} is subject to the following boundary conditions:
\begin{itemize}
    \item \textbf{Cluster merging:} If two clusters with centres $X^{(j)}$ and $X^{(k)}$ meet, they merge into a single cluster of mass $m^{(j)} + m^{(k)}$ and position given by~\eqref{eq:cluster-center-after-merging}, so that the system~\eqref{eq:joint-model} starts anew with the clusters $X^{(j)}$ and $X^{(k)}$ being replaced by a single cluster with starting position $X^{(j,k)}$ and mass $m^{(j)} + m^{(k)}$.
    \item \textbf{Cluster dissolution:}
          If the mass of a cluster approaches zero (this will happen if the cluster mass has fallen below the critical mass $\mcrit$), the cluster is removed from the dynamics~\eqref{eq:joint-model}.
\end{itemize}
This process continues until only one cluster remains.

As the derivation of the model~\eqref{eq:joint-model} in~\cref{ssec:derivation} below shows, $1/\phi^{(r)}_j$ and $1/\phi^{(l)}_j$ are the waiting times for a particle inside cluster $j$ to leave the cluster and join the cluster to the right and left, respectively. These timescales are of the order $e^{\gamma \ell \Delta m^{(j)}}$ by the Eyring--Kramers law.

\subsection{Application to the mean-field equation and dynamical metastability}\label{ssec:model:DynMetastable}

The behaviour of the mean-field PDE~\eqref{eq:McKean-Vlasov-PDE} after initial clustering in the regime $\gamma\ell \gg 1$ and $\ell \ll 1$ is governed by the limit $N\to\infty$ of~\eqref{eq:joint-model}, so that the cluster centres are frozen in space and the mass exchange ODE system~\eqref{eq:joint-model:m} determines the dynamics.
This is also reflected by the coalescence timescale $\tcoalesc \asymp N$.
This suggests that after initial clustering, coarsening in the PDE~\eqref{eq:McKean-Vlasov-PDE} occurs by mass exchange between clusters which is described by~\eqref{eq:joint-model:m}.
Indeed, if we start the PDE with a multi-clustered initial state, then the ODE~\eqref{eq:joint-model:m} approximately describes the PDE evolution, as shown in~\cref{fig:3-clusters-ode-comparison} for the potential~\eqref{eq:w-def:Hegselmann--Krause-special}.~\cref{fig:piecewise-parabolic:3-clusters-ode-comparison} shows a similar comparison for the truncated piecewise parabolic potential
with parameters $\alpha, \beta > 0$, $a \in (0,1)$ given by
\begin{align}
    \label{eq:w-def:piecewise-parabolic}
    w(x) = \begin{cases}
               \alpha (x^2 - a^2)/2 + \beta (a^2 - 1)/2 & \text{if } |x| \le a,     \\
               \beta (x^2 - 1)/2                        & \text{if } a < |x| \le 1, \\
               0                                        & \text{if } |x| > 1 .
           \end{cases}
\end{align}

\subsubsection*{Dynamical metastability}For $\gamma \ell \gg 1$, the timescales $e^{\gamma\ell \Delta m}$ (see~\eqref{eq:clusters-masses:Eyring-Kramers} below) for a particle to leave a cluster are exponentially large in the cluster masses $m^{(j)}$, so that we expect dynamical metastability for the cluster mass dynamics. Indeed, assume, for instance, that we start the PDE with two clusters of masses $m_a, m_b \approx \frac{1}{2}$, with $m_a>m_b$.
In the regime $\gamma\ell  \gg 1$, $\ell \ll 1$, we find from~\eqref{eq:joint-model:m} that mass is exchanged at a rate that is exponentially small in $\gamma \ell m_a$ and $\gamma \ell m_b$. Therefore, it will take exponentially long for the stronger cluster to absorb the smaller one. Indeed, we observe both for the PDE and the mass exchange ODE~\eqref{eq:joint-model:m} that the masses change only very little over a large amount of time to then suddenly collapse to a state in which only the initially heavier cluster remains, see~\cref{fig:3-clusters-ode-comparison}.

\begin{figure}
    \centering
    \includegraphics[width=\textwidth]{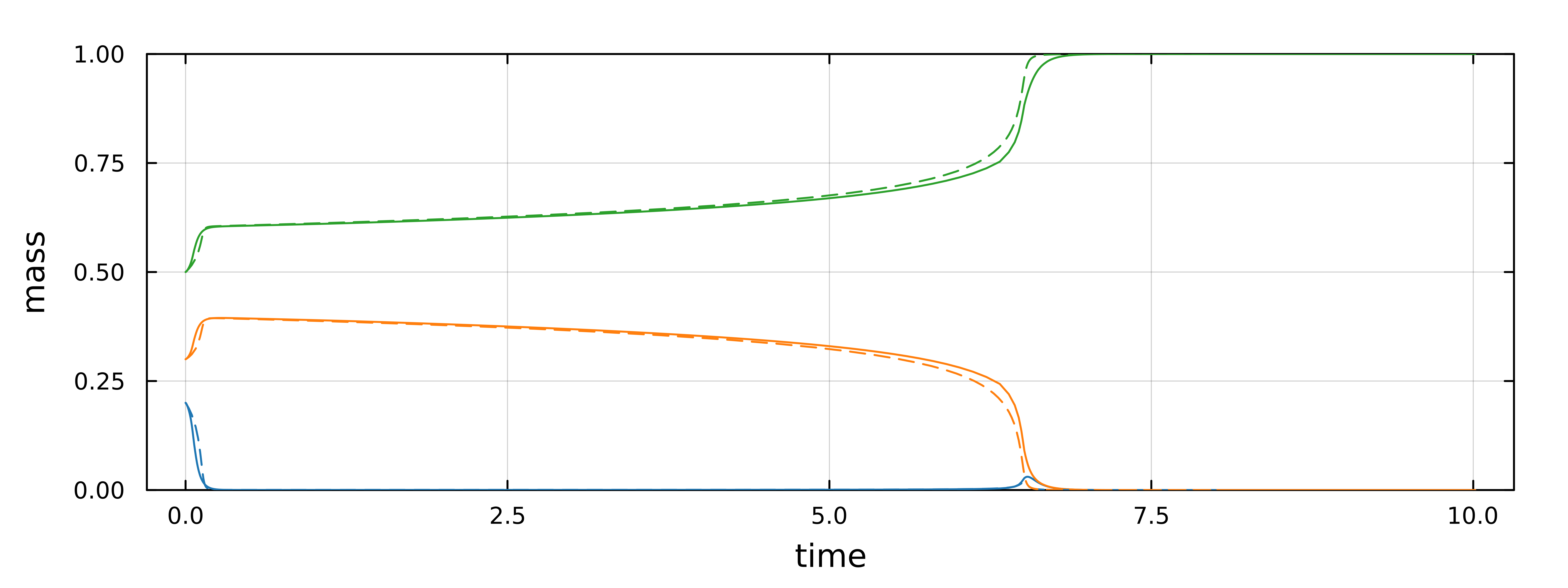}

    \caption{
        Comparison of the masses from the PDE simulation from~\cref{fig:mass-exchange-three-clusters} (solid lines) with the ODE model~\eqref{eq:joint-model:m} (dashed lines).
    }
    \label{fig:3-clusters-ode-comparison}
\end{figure}
\begin{figure}
    \centering
    \begin{subfigure}[c]{0.3\textwidth}
        \centering
        \includegraphics[width=\textwidth]{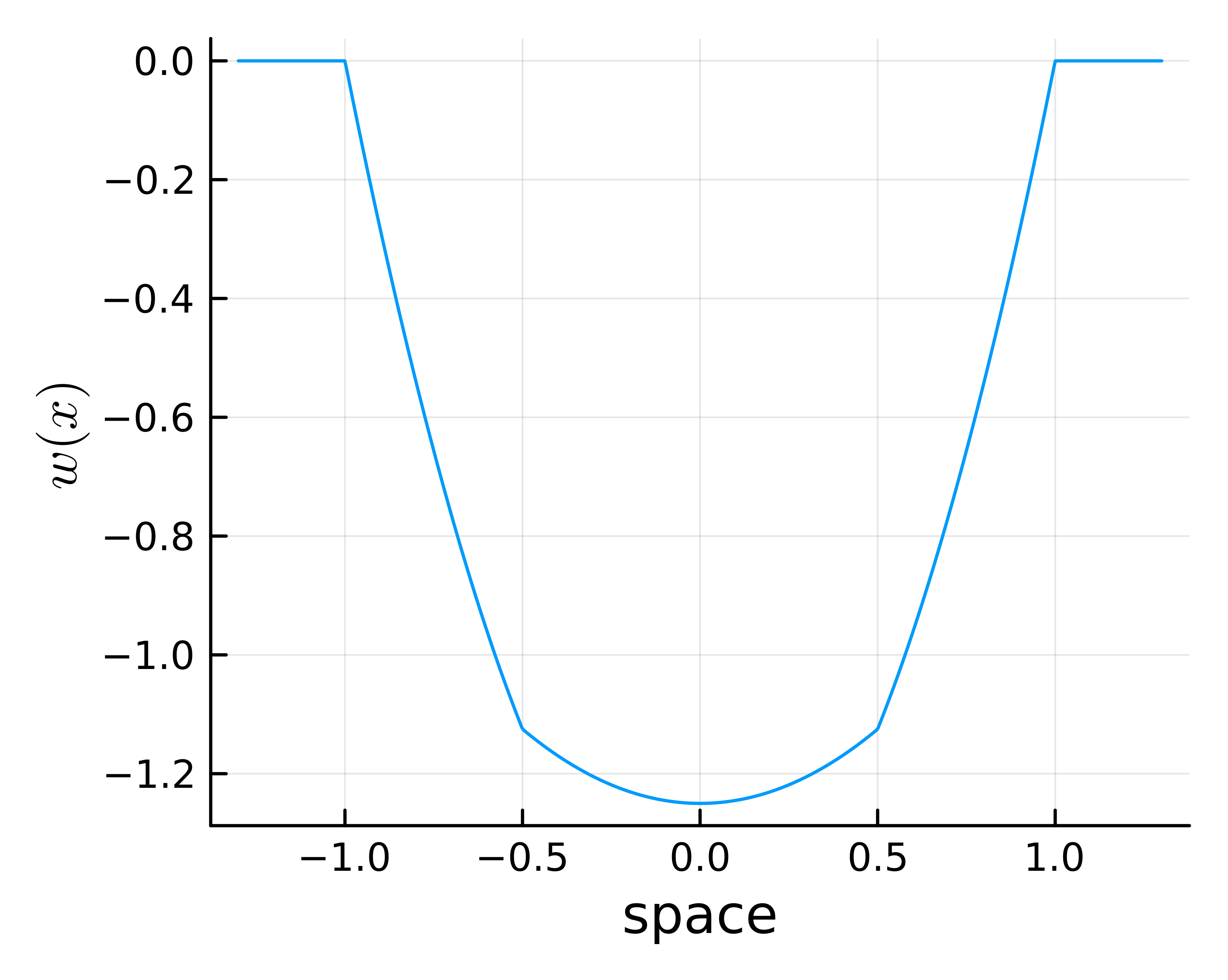}
    \end{subfigure}
    \hfill
    \begin{subfigure}[c]{0.68\textwidth}
        \centering
        \includegraphics[width=\textwidth]{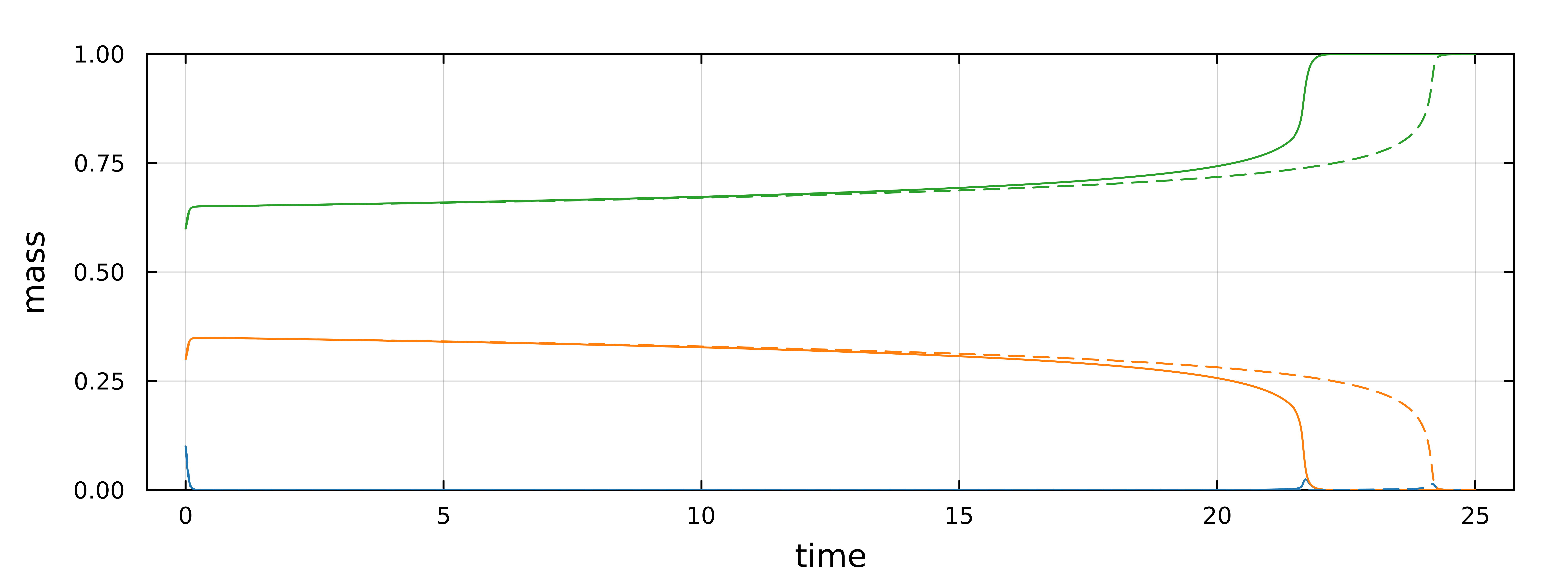}
    \end{subfigure}
    \caption{
        \textbf{Left:} The piecewise parabolic potential~\eqref{eq:w-def:piecewise-parabolic}
        with $\alpha = 1.0$, $\beta = 3.0$, and $a = 0.5$.
        \textbf{Right:} Comparison of the masses from the PDE simulation from~\cref{fig:mass-exchange-three-clusters} (solid lines) with the ODE model~\eqref{eq:joint-model:m} (dashed lines) for the same potential~\eqref{eq:w-def:piecewise-parabolic}, where $\ell = 0.05$, $\gamma = 500$, initial masses $(0.1, 0.3, 0.6)$ and initial positions $(\frac{1}{6}, \frac{3}{6}, \frac{5}{6})$.
    }
    \label{fig:piecewise-parabolic:3-clusters-ode-comparison}
\end{figure}
\subsection{Limitations of the reduced model\texorpdfstring{~\eqref{eq:joint-model}}{}}\label{ssec:model:limitations}

Observe that the model is only a valid approximate description of the PDE with a multi-clustered initial state if $X^{(j+1)}-X^{(j)} \gg \ell s_w$, that is, the initial centres of the clusters are sufficiently far apart, since otherwise deterministic merging phenomena may occur.

Moreover, the model~\eqref{eq:joint-model:m} is in general not a good approximation of the dissolution of clusters. Once one of the clusters has a mass smaller than the critical mass $\mcrit$ (see~\cref{sec:critical-mass}), then the dynamics of the particles of this cluster are not described by the Ornstein--Uhlenbeck approximation from~\cite{GarnierEtAl2017} and the cluster shape is in general not well approximated by~\eqref{eq:multi-cluster-state}. Moreover, in this situation, not only the behaviour of $w$ around zero but its general shape determines the behaviour and the exact evolution of mass. In particular, the description of the ODE~\eqref{eq:joint-model:m} is in general only valid until the first dissolution of a cluster.

\subsection{Heuristic derivation of the reduced models\texorpdfstring{~\eqref{eq:joint-model} and~\eqref{eq:joint-model-jumps}}{}}\label{ssec:derivation}

\subsubsection*{The timescale for a particle to leave a cluster}
As before, we assume that for some fixed time $t\ge 0$,
the particles~\eqref{eq:microscopic-system} are arranged in clusters of masses $m^{(1)}, \dots, m^{(K(t))}\in (0,1)$ and centres $X^{(1)} < \dots < X^{(K(t))}$. For simplicity, let us assume that the particles move on the real line $\R$ instead of the torus.
Each cluster $X^{(j)}$ consists of approximately $m^{(j)} N$ particles.  We would like to compute the timescale of a particle $X^i_t$ in the $j$-th cluster to leave the cluster and join another one, where $i\in\{1,\dots, N\}$ and $j\in\{1,\dots,K(t)\}$.
The empirical measure $\mu^N_t:=\frac{1}{N}\sum_{i=1}^N \delta_{X^i_t}$ can be approximated by a linear combination of Gaussian densities centered at the cluster centres $X^{(k)}_t$ with variances $\sigma^2_k = \frac{\ell}{\gamma w''(0) m^{(k)}}$ and weights $m^{(k)}$:
\begin{align}
    \mu^N_t \approx \sum_{k=1}^{K(t)} m^{(k)}g(\dummy; X^{(k)}, \sigma_k^2) =
    \sum_{k=1}^{K(t)} \frac{ m^{(k)}}{\sqrt{2\pi \sigma^2_k}} \exp\bra*{-\frac{1}{2\sigma^2_k}{\bra*{\dummy-X^{(k)}_t}^2}},
\end{align}
where $g(x;\mu, \sigma^2)$ is the density of a normal distribution with mean $\mu$ and variance $\sigma^2$.
We are interested in the event that the particle  $X^i_t$ of cluster $j$ leaves this cluster and joins another cluster which can be assumed to be one of the neighboring clusters $j-1$ or $j+1$, since the probability that a particle jumps to a non-neighboring cluster is of lower order.
The particle $X^i_t$ moves in the effective potential
\begin{align}
    \label{eq:effective-potential-approx}
    \bra*{W_{\gamma,\ell}\ast \mu^N_t}(x) \approx V_{\rm eff}(x) :=
    \sum_{k=1}^{K(t)} m^{(k)} \bra*{\gamma \ell w\bra*{\frac{\dummy}{\ell}}\ast g(\dummy; X^{(k)}, \sigma_k^2)}(x),
\end{align}
that is, $X^i_t$ approximately follows the SDE $\d X^i_t = -  V'_{\rm eff}(X^i_t) \d t + \sqrt{2} \d B^i_t$.

In order to approximate $V_{\rm eff}$ around its local minimum $X^{(j)}$, we replace $w(x)$
in~\eqref{eq:effective-potential-approx}
by the quadratic approximation $w(x) \approx -\Delta + w''(0)\frac{x^2}{2}$ for small $\abs{x}$, where $\Delta :=  - \inf w= -w(0)$. This is justified by the fact that $\sigma_k^2$ is of order $\ell/\gamma$, while the interaction range of the potential $W_{\gamma,\ell}$ is of order $\ell \gg \revfirst {{\sqrt{\ell/\gamma}}}$.
Using this approximation, we obtain then
\begin{align}
    \label{eq:V-eff-approx}
    \begin{split}
        V_{\rm eff}(x) & \approx m^{(k)}  \bra*{\bra*{-\gamma \ell\Delta + \frac{\gamma w''(0)}{2\ell}{(\dummy)^2}} \ast g(\dummy; X^{(k)}, \sigma_k^2)}(x) \\
                       & = - \gamma \ell \Delta m^{(k)} + \frac
        {\gamma w''(0)m^{(k)}}{2\ell}
        \bra*{\bra*{x-X^{(k)}}^2 + \sigma_k^2}
        \\ & = - \gamma \ell \Delta m^{(k)} +  \frac
        {1}{2\sigma_k^2}{\bra*{x-X^{(k)}}^2} + \frac{1}{2},
    \end{split}
\end{align}
for $x$ close to $X^{(k)}$, see also~\cref{fig:V-eff-approximation}.
Fix some $a\in \R$ inside the potential well of $V_{\rm eff}$ around $X^{(j-1)}$ which is close enough to the right boundary of this potential well and some $b\in \R$ inside the well of $V_{\rm eff}$ around $X^{(j+1)}$ which is close enough to the left boundary of this well.
Let $\tau_{a,b}$ be the time of $X^i_t$ to hit $(-\infty, a]\cup [b,\infty)$ and set $u(x)= \condexpect*{\tau_{a,b} \given X^i_0 = x}$. By classical results~\cite[Theorem 5.7.3]{dembo2001large}~\cite[Sec. 7.3]{pavliotis2014book}, $u$ satisfies the boundary value problem
\begin{align}
    \label{eq:clusters-masses:BVP}
    u''(x) - V_{\rm eff}'(x) u'(x) & = -1 \qquad \text{in }(a,b), \\
    u(a) = 0, \quad u(b)           & = 0,
\end{align}
which has the solution
\begin{align}
    \label{eq:hitting-time-explicit}
    u(x) & = - \int_a^x \int_a^y e^{V_{\rm eff}(y)- V_{\rm eff}(z)} \dx z \dx y  + \frac{
        \int_a^x e^{V_{\rm eff}(y)} \dx y \cdot
        \int_a^b \int_a^y e^{V_{\rm eff}(y)- V_{\rm eff}(z)} \dx z \dx y
    }{
        \int_a^b e^{V_{\rm eff}(y)} \dx y
    }
\end{align}
for $x\in[a,b]$.
\begin{figure}
    \includegraphics[width=0.9\textwidth]{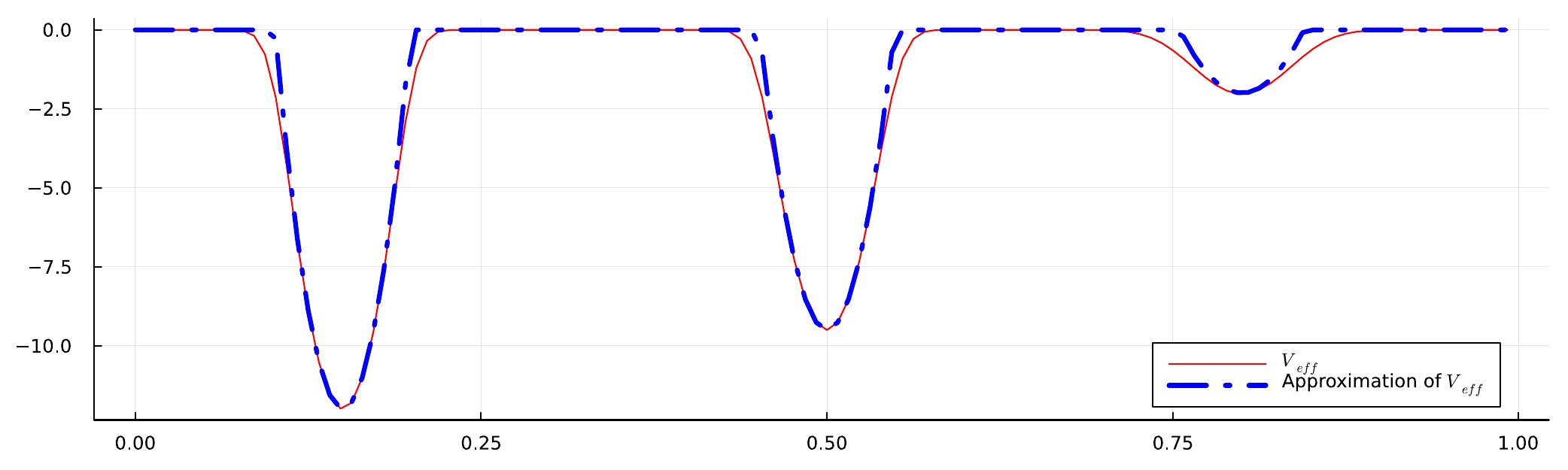}
    \caption{The effective potential $V_{\rm eff}$ for the Hegselmann--Krause prototype~\eqref{eq:w-def:Hegselmann--Krause-special}. The solid curve shows the convolution of $W_{\gamma,\ell}$ with a Gaussian mixture, the dashed curve shows the locally quadratic approximation \eqref{eq:V-eff-approx} around the minima $X^{(j)}$ which is truncated by zero from above. }
    \label{fig:V-eff-approximation}
\end{figure}
We also need the probability that a particle $X^i_t$ which starts at $x$ hits $a$ before hitting $b$ which is given by the equilibrium potential
    $h_{a,b}(x):= \condprob*{\tau_a < \tau_b \given X^i_0 = x}$. Then, $h_{a,b}$ satisfies the boundary value problem given by (see e.g.~\cite{Berglund2013})
\begin{align}
    \label{eq:clusters-masses:BVP-equilibrium-potential}
    h_{a,b}''(x) - V_{\rm eff}'(x) h_{a,b}'(x)  = 0 \qquad \text{in }(a,b), \\
    h_{a,b}(a) = 1, \quad h_{a,b}(b) = 0.
\end{align}
Therefore, we obtain
\begin{align}
    \label{eq:clusters-masses:equilibrium-potential-solution}
    h_{a,b}(x) & = \frac{
        \int_x^b e^{V_{\rm eff}(y)} \dx y
    }
    {
        \int_a^b e^{V_{\rm eff}(y)} \dx y
    }
    \qquad \text{for all } x\in[a,b].
\end{align}
By carefully examining the shape of the effective potential $V_{\rm eff}$, Laplace's method yields the following result for the limit $\gamma \to \infty$.
    \begin{proposition}
        \label{prop:prop-laplace-hitting-time}
        In addition to~\cref{w-assumption}, suppose that $w$ is $C^3$ on $[-r_w, r_w]$ for some $r_w\in(0,s_w]$ and that there exists some
        $\delta_w\in(0, s_w)$ and $\tau>0$ such that
        \begin{align}
            \label{eq:w-assumption:linear-upper-bound-around-sw}
            w(z) \le \tau\cdot (z - s_w) \quad \text{for all } z \in [ \delta_w, s_w].
        \end{align}
        Let $K\ge 3$ be a natural number, $\gamma>0$, $\ell>0$ and fix $j\in\{2,\dots, K-1\}$.
        Let the cluster positions $X^{(1)} < \dots < X^{(K)} \in \R$ and masses $m^{(1)}, \dots, m^{(K)} \in (0,1)$ be given such that $\sum_{j=1}^K m^{(j)} = 1$ and such that $X^{(k+1)} - X^{(k)} > 2s_w\ell$ for all $k=1,\dots,K-1$. Let $\alpha \in (0,s_w)$ such that
        \begin{align}
            \label{eq:w-assumption:alpha-condition}
            \frac{\abs{w(\alpha)}}{\Delta} < \min\set*{\frac{m^{(j)}}{m^{(j-1)}}, \frac{m^{(j)}}{m^{(j+1)}}}
        \end{align}
        and define\footnote{The condition \eqref{eq:w-assumption:alpha-condition} ensures that $a$ is sufficiently close to the right boundary of the potential well around $X^{(j-1)}$ and that $b$ is sufficiently close to the left boundary of the potential well around $X^{(j+1)}$.} $a:= X^{(j-1)} + \alpha \ell$ and $b = X^{(j+1)} - \alpha \ell$. 

        Define the Gaussian mixture
        \begin{align}
            \label{eq:gaussian-mixture}
            \rho(\dummy) := \sum_{k=1}^K m^{(k)} g(\dummy; X^{(k)}, \sigma_k^2), \quad \text{where } \sigma_k^2 := \frac{\ell}{\gamma w''(0) m^{(k)}},
        \end{align}
        and define the effective potential
        \begin{align}
            \label{eq:prop-laplace:effective-potential}
            V_{\rm eff}(x) := \bra*{\gamma \ell w\bra*{\frac{\dummy}{\ell}}\ast \rho}(x).
        \end{align}
        Consider a diffusive particle $X^i_t$ on the real line moving according to $\d X^i_t = -  V'_{\rm eff}(X^i_t) \d t + \sqrt{2} \d B^i_t$ in the effective potential \eqref{eq:prop-laplace:effective-potential} which is formed by the multi-cluster state $\rho$. Consider the hitting time
        \begin{align}
            \tau_{a,b} := \inf\{t\ge 0: X^i_t \in (-\infty, a]\cup [b,\infty)\}.
        \end{align}
        Fix $x_0 \in (X^{(j)}- s_w\ell, X^{(j)} + s_w\ell)$. Then, we have the asymptotic equivalence
        \begin{equs}
            \label{eq:clusters-masses:Eyring-Kramers}
            {}& \qquad\condexpect*{\tau_{a,b} \given X^i_0 = x_0} \\
            &\sim{}
            \frac{
                \bra*{X^{(j+1)} - X^{(j)} -2s_w\ell} \cdot
                \bra*{X^{(j)} - X^{(j-1)} -2s_w\ell}
            }{
                \bra*{X^{(j+1)} - X^{(j-1)} - 2s_w\ell}
            }
            \sqrt{\frac{2\pi \ell}{\gamma w''(0) m^{(j)}}} e^{\gamma \ell \Delta m^{(j)} - \frac12}
        \end{equs}
        as $\gamma\to\infty$, which is an Eyring--Kramers formula for the expected hitting time of the $j-1$-th and the $j+1$-th cluster when starting inside the $j$-th cluster. Moreover, we obtain for the equilibrium potential
        \begin{align}
            \label{eq:clusters-masses:equilibrium-potential}
            \condprob*{\tau_a < \tau_b \given X^i_0 = x} \sim p^{(j)}:=\frac{
                X^{(j+1)} - X^{(j)} - 2s_w\ell
            }
            {
                X^{(j+1)} - X^{(j-1)} - 4s_w\ell
            } \quad \text{as } \gamma\to\infty.
        \end{align}
    \end{proposition}
    We will prove \cref{prop:prop-laplace-hitting-time} in~\cref{app:proof-prop-laplace-hitting-time}. The central idea is to use a rigorous version of~\eqref{eq:V-eff-approx} to approximate $V_{\rm eff}$ around its minima $X^{(1)}, \dots, X^{(K)}$. Once a particle has left a valley formed by one of the clusters, the time it takes to reach a neighbouring cluster is determined by the widths of the potential ridges between two clusters. As we will show, these widths are given by $X^{(j+1)} - X^{(j)} - 2s_w\ell$ in the limit $\gamma\to\infty$ for the ridge between cluster $j$ and cluster $j+1$. This explains the prefactors in~\eqref{eq:clusters-masses:Eyring-Kramers}.

    With  \eqref{eq:clusters-masses:Eyring-Kramers} and \eqref{eq:clusters-masses:BVP-equilibrium-potential}, we can describe the mass exchange as follows: A particle $X^i_t$ of cluster $j$ waits for an exponentially distributed time with expectation $\condexpect*{\tau_{a,b} \given X^i_0 = x}$ inside the cluster $j$. When it leaves the cluster after this random time, it joins the cluster $j-1$ with probability $p^{(j)}$ and the cluster $j+1$ with probability $1-p^{(j)}$. This gives the following waiting times for jumps of the particle $X^i_t$ inside cluster $j$ to the cluster $j-1$ and $j+1$, respectively:
\begin{align}
    T_{\rm left}^{(j)}  & = \frac{  \condexpect*{\tau_{a,b} \given X^i_0 = x}}{p^{(j)}}
    \approx \bra*{d_{\rm dir}\bra*{X^{(j-1)}, X^{(j)}}- 2 s_w \ell}
    \sqrt{\frac{2\pi \ell}{e \gamma w''(0) m^{(j)}}} e^{\gamma \ell \Delta m^{(j)}}       \\
    T_{\rm right}^{(j)} & = \frac{  \condexpect*{\tau_{a,b} \given X^i_0 = x}}{1-p^{(j)}}
    \approx \bra*{d_{\rm dir}\bra*{X^{(j)}, X^{(j+1)}}- 2 s_w \ell}
    \sqrt{\frac{2\pi \ell}{e \gamma w''(0) m^{( j)}}} e^{\gamma \ell \Delta m^{(j)}}
\end{align}
where we used the directed distances~\eqref{eq:def:directed-distance} to account for the fact that we are on the one-dimensional torus. Now the waiting times $T_{\rm left}^{(j)}$ and $T_{\rm right}^{(j)}$ are just the inverses of the rates $\phi^{(l)}_j$ and $\phi^{(r)}_j$ defined in~\eqref{eq:def:phi-l} and~\eqref{eq:def:phi-r}, respectively. From this, we obtain the model~\eqref{eq:joint-model:m}. More precisely, we can use the jump process described in the next section to model the mass exchange, from which we can then derive~\eqref{eq:joint-model:m}.

\subsection{Modeling the mass exchange via jumps of particles}\label{ssec:model:jumps}
Based on the waiting times for particles from one cluster to join adjacent clusters, we can write down a continuous-time stochastic Markov process for the mass exchange of the clusters which features jumps of particles from one cluster to another. Assuming again that the particles are arranged in $K(t)\in\N^+$ clusters at time $t\ge 0$ with masses $m^{(j)}=M^{(j)}/N$ and cluster centres $X^{(j)}$ where $j=1,\dots, K(t)$, we  model the mass exchange of these clusters by a continuous time Markov chain with state space
\begin{align}
    \mathfrak{M}(N, K(t)) :=
    \set*{\bra*{M^{(1)}, \dots, M^{(K(t))}} \in \N^{(K(t))} \colon M^{(1)} + \dots + M^{(K(t))}=N},
\end{align}
where $N$ is the total number of particles as before.

For $i,j\in\{1,\dots, K(t)\}$, we denote that two clusters $i$ and $j$ are neighbors on $\T$ by writing
\begin{align}
    i \sim_{K(t)} j \qquad\vcentcolon\Longleftrightarrow \qquad \abs{i-j}=1 \text{ or }(i,j) \in \{(1,K(t)), (K(t),1)\} \,.
\end{align}
We say that the state $\mathbf{\widetilde{M}} = \bra[\big]{\widetilde{M}^{(1)}, \dots, \widetilde{M}^{(K(t))}}\in \mathfrak{M}(N, K(t))$  arises from the state $\mathbf{M} = \bra[\big]{M^{(1)}, \dots, M^{(K(t))}} \in \mathfrak{M}(N, K(t))$
by a jump of a particle from cluster $i$ to cluster $j$ 
if $\widetilde{M}^{(i)}= M^{(i)}-1$, $\widetilde{M}^{(j)}= M^{(j)}+1$ and $\widetilde{M}^{(k)}= M^{(k)}$ for all $k\in \{1,\dots, K(t)\}\setminus\{i,j\}$.
The $Q$-matrix of the Markov chain $\bra*{M^{(1)}_t, \dots, M^{(K(t))}_t}$ is defined as follows. If $ \mathbf{M} = \mathbf{\widetilde{M}}$, it is defined by
\begin{align}
    Q\bra*{\mathbf{M}, \mathbf{M}} :=
    \sum_{j=1}^{K(t)}
    M^{(j)}\bra*{\phi_j^{(r)}+\phi_j^{(l)}}\bra*{
        \frac{M^{(j)}}{N},
        X^{(j-1)}, X^{(j)}, X^{(j+1)}}
\end{align}
Moreover, if $\mathbf{\widetilde{M}}$ arises from $\mathbf{M}$ by a jump of a particle from cluster $i$ to cluster $j$ where $i\sim_{K(t)} j$ are neighbors, then we set
\begin{align}
    Q\bra*{\mathbf{M}, \mathbf{\widetilde{M}}}
    :=   \frac{M^{(i)}}{N} \bra*{\phi_i^{(r)}+ \phi_i^{(l)}}
    \bra*{
        \frac{M^{(i)}}{N}, X^{(i-1)}, X^{(i)}, X^{(i+1)}}.
\end{align}
In all other cases, we set $Q\bra*{\mathbf{M}, \mathbf{\widetilde{M}}} = 0$.
The corresponding chemical master equation (Kolmogorov forward equation) reads
\begin{align}
    \label{eq:clusters-masses:master-equation}
     & \frac{\partial}{\partial_t}
    \condprob*{ \mathbf{M}_t = \mathbf{M} \given \mathbf{M}_0 = \widebar{\mathbf{M}}}
    = \sum_{\widetilde{\mathbf{M}} \in  \mathfrak{M}(N, K(t))}
    \condprob*{\mathbf{M}_t = \widetilde{\mathbf{M}} \given \mathbf{M}_0 = \widebar{\mathbf{M}}}
    \cdot
    Q\bra*{\mathbf{M}, \mathbf{\widetilde{M}}}
\end{align}
With this, we can write down the following model:
\begin{subequations}
    \label{eq:joint-model-jumps}
    \begin{empheq}[box=\fbox]{align}
        \label{eq:joint-model-jumps:X}
        &\d X^{(j)}_t = \frac{\sqrt{2}}{\sqrt{N m^{(j)}}} \dx W^{(j)}_t \qquad\text{for $j=1,\dots, K(t)$} \\
        \label{eq:joint-model-jumps:m}
        &\parbox{0.7\textwidth}{The numbers of particles in the clusters $(M^{(1)}, \dots, M^{(K(t))})$ evolve as a continuous time Markov chain with state space $\mathfrak{M}(N, K(t))$ and $Q$-matrix as defined above.}
    \end{empheq}
\end{subequations}
This model is augmented by the same boundary conditions as the model~\eqref{eq:joint-model}, that is, clusters merge when getting close and if the number of particles $M^{(j)}$ in a cluster approaches zero, the cluster is removed from the system. 

In the limit $N\to\infty$, we expect $\frac{M^{(j)}}{N} \to m^{(j)}$ and we can formally replace the chemical master equation~\eqref{eq:joint-model-jumps:m} by the reaction rate equation~\eqref{eq:joint-model:m}, see~\cite{gillespie2000chemical} for details.

Note that the ODE~\eqref{eq:joint-model:m} exhibits dynamical metastability in the same way as the continuous-time jump-diffusion process given by the Q-matrix defined above exhibits (stochastic) metastability.

\begin{figure}
    \includegraphics[width=0.9\textwidth]{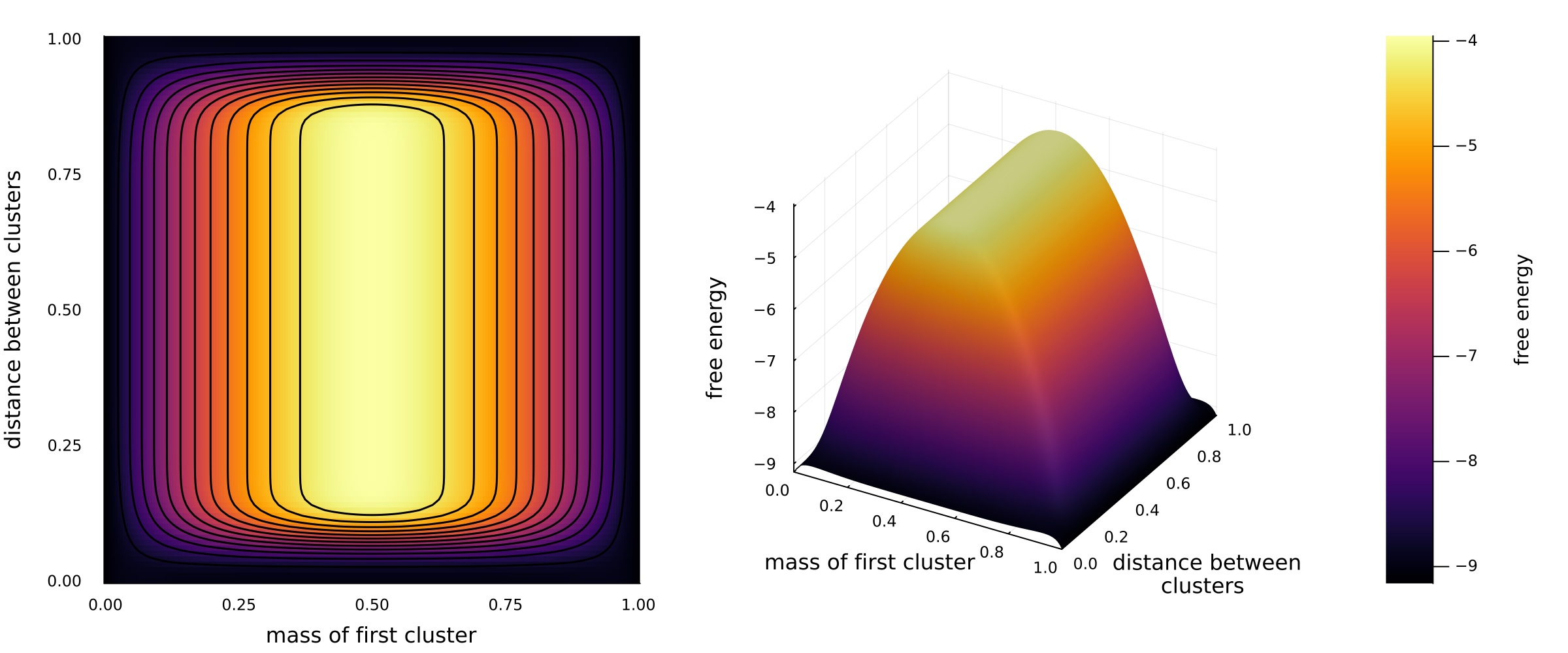}
    \caption{Free energy landscape for a two-cluster state~\eqref{eq:multi-cluster-state} as a function of the two masses $m_1\in [0,1]$, $m_2 = 1- m_1$ with distance between cluster centres $d:=X^{(1)}- X^{(2)}\in [0,1]$.
    For fixed $d$, the free energy is maximized at $m_1 = m_2 = 1/2$, which is however an unstable state. 
    The free energy was computed by discretizing the torus with $2^9$ grid points. Parameters: $\gamma = 150, \ell = 0.1$, $w$ as in~\eqref{eq:w-def:Hegselmann--Krause-special}.}
    \label{fig:free-energy-landscape}
\end{figure}

\begin{remark}[Existence of multi-cluster steady states]
    \label{rem:multi-cluster-steady-states}
    A multi-cluster state with $K$ clusters of the same masses $\frac{1}{K}$ and at equidistant positions on the torus corresponds to a stationary state of the ODE~\eqref{eq:joint-model:m}. Indeed, we can prove the existence of stationary states of~\eqref{eq:McKean-Vlasov-PDE} of this form:
    By~\cref{sec:discontinuous-phase-transition}, for $\ell\ll1$ and $\gamma \gg 1$, there exists a non-trivial global minimiser $q_{\gamma,\ell}$ of the free energy. Moreover,  $q_{\gamma,\ell}$ must be a single-cluster state by~\cref{thm:glob-min-sym-decreasing}.
    We show now how to construct multi-cluster stationary states from~$q_{\gamma,\ell}$.

    We view $q_{\gamma,\ell}$ as a $1$-periodic function $\R\to\R$. Then, for every $K\in\N^+$ and assuming that $s_w\ell < 1/2$, the state $\tilde{q}(x):=q_{\gamma,\ell}(Kx)$ is a fixed point of~\eqref{eq:kirkwood-monroe} for $W_{\tilde{\gamma}, \tilde{\ell}}$ where $\tilde{\gamma} = K^2\gamma$ and $\tilde{\ell} = \ell/K$, since
    \begin{align}
            & \bra*{W_{K^2\gamma, \ell /K} \ast \tilde{q}}(x)
        =K\gamma \ell \int_{-\frac{1}{2}}^{\frac{1}{2}} w\bra*{\frac{K y}{\ell}} q_{\gamma,\ell}\bra*{K(x-y)} \dx y
        \\
        ={} & \gamma\ell\int_{-\frac{KL}{2}}^{\frac{KL}{2}} w\bra*{\frac{y}{\ell}} q_{\gamma,\ell}\bra*{Kx-y} \dx y
        = \bra*{
            W_{\gamma,\ell} \ast q_{\gamma,\ell}}(Kx),
    \end{align}
    where we used that $\ell \supp w \subset [-1/2 ,1/2]$ in the last step by assumption. This can be also seen by calculating the Fourier modes: On the one hand, we have that $\tilde{q}(x) = \frac{1}{2\pi} \sum_{k \in \mathcal{K}} \hat{q}_{\gamma,\ell}(k) e^{i K k x}$, where $\hat{q}_{\gamma,\ell}(k)= \int_{[-\frac{1}{2},\frac{1}{2}]} q_{\gamma,\ell}(x) e^{-i k x} \dx x$. On the other hand, the Fourier modes of $W_{\tilde{\gamma}, \tilde{\ell}}$ are given by
    \begin{align}
        \widehat{W}_{K^2\gamma, \ell/K}(k)
         & = K \gamma \ell \int_{[-\frac{1}{2},\frac{1}{2}]} w\bra*{\frac{K y}{\ell}} e^{-i K k y} \dx y
        = \gamma \ell \int_{-K\frac{1}{2}}^{K\frac{1}{2}} w\bra*{\frac{y}{\ell}} e^{-i k y} \dx y
        \\&= \gamma \ell \int_{\frac{-1}{2}}^{\frac{1}{2}} w\bra*{\frac{y}{\ell}} e^{-i k y} \dx y
        \\ &= \widehat{W}_{\gamma,\ell}(Kk).
    \end{align}
    In particular, these calculations show that multi-cluster stationary states exist. However, these states are unstable as can be seen from the free energy landscape in~\cref{fig:free-energy-landscape}.
\end{remark}

\section{Discussion and outlook}
\label{sec:outlook}
In this paper, the dynamics of weakly interacting diffusions in an attractive localized interaction potential was studied. It was shown, through a combination of heuristics, rigorous analysis, and careful numerical experiments, that the system exhibits very rich dynamical behaviour that the mean-field description cannot fully capture. In particular, we identified several timescales and dynamical phenomena: formation of clusters, merging/coalescence, mass transfer between clusters, and microscopic reversibility, which we summarize as follows:
\begin{itemize}
    \item \textbf{Initial clustering:} The timescale $\tcluster$ in~\eqref{eq:initial-clustering-time} for the onset of clustering is determined solely by the activation of the dominant unstable mode of the linearization of the mean-field PDE~\eqref{eq:linearized-PDE}.
          As shown in~\cref{ssec:Initial:GPY}, in the case of the mean-field PDE and for a generic initial datum $\rho_0 \neq \rhounif$, the clustering time $\tcluster$ is of order one. The same applies to the particle system when its initial positions are sampled i.i.d.\ from a generic probability measure $\rho_0 \neq \rhounif$. However, if the initial positions of the particles are sampled i.i.d.\ from the uniform measure $\rhounif$, then the timescale for clustering is of order $\tcluster\asymp \log N$.
    \item \textbf{Coalescence:} The merging of clusters only happens in the particle system on the timescale $\tcoalesc$ of order $N$ as argued in~\cref{ssec:model:Garnier}.
    \item \textbf{Dynamic metastability:} Each cluster of macroscopic mass $m\in (\mcrit,1]$ is metastable on a timescale $\tmetastab \asymp e^{\gamma \ell \Delta m}$ independent of $N$ as implied by the proposed extended model in~\cref{ssec:model:MassExchange}.
    \item \textbf{Microscopic reversibility:} The finite particle system~\eqref{eq:microscopic-system} is ergodic due to the presence of Brownian motion for each particle. However, the presence of a discontinuous phase transition (see~\cref{prop:discontinuous}) implies that the free energy~\eqref{eq:def:free_energy} has a critical point. Since the clustered state is the global minimiser, its energy barrier $\Delta_\cF>0$ to this critical point determines the timescale of microscopic reversibility $\tmicrorev \asymp e^{N\Delta_{\mathcal F}}$ (see~\cite[Section 5]{GvalaniSchlichting2020}).
\end{itemize}
Consequently, we obtain the following behaviour for the particle model~\eqref{eq:microscopic-system}: As long as $\log N \ll \gamma \ell \Delta $, the movements of clusters~\eqref{eq:joint-model:X} and~\eqref{eq:joint-model-jumps:X} dominate the mass exchange~\eqref{eq:joint-model:m} and~\eqref{eq:joint-model-jumps:m}. In this case, the clusters will merge before the mass exchange has a significant effect. On the contrary, for $\log N \gg \gamma \ell \Delta$, the mass exchange dominates the cluster movements.

The identification of the timescales obtained in this paper raises several interesting questions that we plan to address in future work. The rigorous proof of dynamical metastability of the mean-field PDE, in the sense of Otto--Reznikoff, is still open. Furthermore, the rigorous justification of the limit to the massive Arratia flow in the appropriate scaling regime is a challenging question. A precise form of the conjecture can be found in~\cref{sec:ModArratia}. Finally, the rigorous analysis of similar questions for the kinetic Langevin dynamics, in particular in the weak friction regime, is also of interest.

\appendix

\section{Convergence towards the Modified
  Massive Arratia flow} \label{sec:ModArratia}

As explained above, in the regime $\gamma\gg 1$ and $\ell \ll 1$, once clusters have arisen, the system~\eqref{eq:microscopic-system} behaves essentially as a system of coalescing heavy Brownian motions. It is therefore natural to expect that appropriate rescalings of~\eqref{eq:microscopic-system} converge in the large particle limit to the Massive Modified Arratia Flow~(MMA)~\cite{KonRenesse_2019}, which is a random element $\{y(u,t),\ u\in[0,1],\ t\in[0,T]\}$ in the Skorohod space $D([0,1],C[0,T])$ characterized by the following set of axioms.

\begin{theorem}[Massive Arratia flow \protect{\cite[Theorem~1.1]{Konarovskyi2017}}]
    \label{thm:arratia-flow-axioms}
    There exists a random element $\{y(u,t),\ u\in[0,1],\ t\in[0,T]\}$ in the Skorohod space $D([0,1],C[0,T])$ such that
    \begin{enumerate}[label=(C\arabic*)]
        \item for all $u\in[0,1]$, the process $y(u,\cdot)$ is a continuous square integrable martingale with respect to the filtration
              \begin{equation}\label{f_filtration_y}
                  \mathcal{F}_t=\sigma(y(u,s),\ u\in[0,1],\ s\leq t),\quad t\in[0,T];
              \end{equation}

        \item for all $u\in[0,1]$, $y(u,0)=u$;

        \item for all $u<v$ from $[0,1]$ and $t\in[0,T]$, $y(u,t)\leq y(v,t)$;

        \item for all $u\in[0,1]$, the quadratic variation has the form
              $
                  \langle y(u,\cdot)\rangle_t=\int_0^t\frac{ds}{m(u,s)},
              $
              where $m(u,t)=\mathcal{L}^1 \{v:\ \exists s\leq t\ y(v,s)=y(u,s)\}$, $t\in[0,T]$;

        \item for all $u,v\in[0,1]$ and $t\in[0,T]$,
              $
                  \langle y(u,\cdot),y(v,\cdot)\rangle_{t\wedge\tau_{u,v}}=0,
              $
              where $\tau_{u,v}=\inf\{t:\ y(u,t)=y(v,t)\}\wedge T$.
    \end{enumerate}
\end{theorem}

Based on~\eqref{eq:microscopic-system}, we define the random variables $Z^N_{\gamma,\ell}$ with values in $\R^{[0,1]\times [0,\infty)}$ by
\begin{align}
    Z^N_{\gamma, \ell}(u,t) := X^{\lceil N u\rceil}_{t}.
\end{align}
Comparing~\eqref{eq:joint-model:X} with the axioms for the MMA~\cref{thm:arratia-flow-axioms} and with the construction of the MMA by means of the finite particle system from~\cite[Section 2]{Konarovskyi2017}, we see that we have to rescale the noise in~\eqref{eq:microscopic-system} by a factor of $\frac{\sqrt{N}}{\sqrt{2}}$, or equivalently, rescale time by a factor of $\frac{N}{2}$. Having this in mind, we formulate the following conjecture:
\begin{conjecture*}
    Let $(\gamma_n)$, $(\ell_n)$ be sequences in $(0,\infty)$ and $(N_n)$ be a sequence in $N^+$ such that
    \begin{align}
        \label{eq:Scaling:MMA}
        N_n \to \infty, &  & \ell_n \downarrow 0, &  & \frac{\gamma_n \ell_n}{\log N_n} \gg 1.
    \end{align}
    Then we have
    \begin{align}
        Z^{N_n}_{\gamma_n,\ell_n}\bra*{u,\tfrac{N_n t}{2}} \to y(u,t) \qquad \text{as } n \to \infty
    \end{align}
    in a suitable sense.
\end{conjecture*}
The condition $\gamma_n\ell_n\to\infty$ in~\eqref{eq:Scaling:MMA} ensures that clusters are asymptotically stable as  $n\to\infty$, while
$\ell\downarrow 0$ ensures that clusters merge in the limit only when they have exactly the same position. Finally, we want the coalescence dynamics to dominate over the mass exchange in the model~\eqref{eq:joint-model} which explains the condition $\frac{\gamma_n\ell_n}{\log N_n} \gg 1$.

\begin{remark}
    The scaling in~\eqref{eq:Scaling:MMA} is related to the moderately interacting diffusions considered by Oelschläger~\cite{oelschlager1985law}, but his regime is $\gamma_n\ell_n\asymp 1$, while in our regime $\gamma_n$ grows much faster.
    In the direction of more localized attractive scalings, the work~\cite{ChenGoettlichKnapp2018} provides a first result. The authors consider the regime, in which the interaction becomes an attractive Dirac $W_{\gamma_n,\ell_n} \rightharpoonup - \delta_0$. The derived mean-field limit has, besides the linear diffusion, a quadratic \emph{backward} porous medium. The system is only well-posed for initial data close to the uniform state, and larger states show a blow-up like a reverse Barenblatt solution. In this sense, the mean-field limit can show the initial clustering, but the dynamics of clusters cannot be observed in the mean-field limit, in agreement with our result.
\end{remark}

\section{Proof of Proposition~\ref{prop:prop-laplace-hitting-time}}
\label{app:proof-prop-laplace-hitting-time}
This section presents the proof of~\cref{prop:prop-laplace-hitting-time}. In the following, we use the notation $\calW_{\gamma,\ell}(y) := \gamma \ell w\bra*{\frac{y}{\ell}}$ for all $y\in\R$, which is different from $W_{\gamma,\ell}$ as defined in~\eqref{eq:def:rescaledW} since we do not periodize $\calW_{\gamma,\ell}$, working here on $\R$ instead of $\T$. We also use the notation $\rho_k(\dummy) := m^{(k)} g(\dummy; X^{(k)}, \sigma_k^2)$ for the $k$-th cluster as well as
$L_k:= X^{(k)} - s_w\ell$ and $R_k := X^{(k)} + s_w\ell$ for $k=1,\dots, K$ for the left and right boundary of the $k$-th valley, respectively. For all of the following, we assume the assumptions of \cref{prop:prop-laplace-hitting-time} to hold.

For the proof of~\cref{prop:prop-laplace-hitting-time}, we need several auxiliary results. Let's first prove that outside of $[L_k, R_k]$, the convolution $\calW_{\gamma,\ell}\ast \rho_k$ converges to zero as $\gamma \to \infty$.

\begin{lemma}
    \label{lemma:W-convolution-outside-valley-bound}
    Fix $k\in\{1,\dots, K\}$. For all $x\in\R\setminus[L_k, R_k]$, we have that
    \begin{align}
        \label{eq:W-convolution-outside-valley-bound}
        \abs*{\bra*{\calW_{\gamma, \ell}\ast \rho_k}(x)}  \le \gamma \ell \Delta m^{(k)} \exp\bra*{-\frac{ \eps^2w''(0)m^{(k)}}{2\ell }\gamma},
    \end{align}
    where $\eps:= \abs*{x-X^{(k)}} - s_w \ell >0$.
\end{lemma}

\begin{proof}
    We have the standard bound for Gaussian tails
    \begin{align}
        \label{eq:gaussian-tail-bound}
        \P\pra*{ \abs{Z} \ge C} \le \exp\bra*{-\frac{C^2}{2\sigma^2}} \quad \text{whenever } Z \sim \calN(0,\sigma^2) \text{ and } C>0.
    \end{align}
    This follows from the moment generating function of a Gaussian and Markov's inequality.
    Therefore, we obtain
    \begin{align}
               & \abs*{\bra*{\calW_{\gamma, \ell}\ast \rho_k}(x)}  =\abs*{\int_{-s_w\ell}^{s_w\ell} \calW_{\gamma,\ell}(y) \rho_k(x-y)\dx y}                                             \\
        \le{}  & \gamma \ell \Delta \int_{-s_w\ell}^{s_w \ell} \rho_k(x-y) \dx y =\gamma \ell \Delta m^{(k)} \P_{Z\sim \calN(0, \sigma_k^2)}\pra*{\abs*{Z-\bra*{x-X^{(k)}}} \le s_w \ell} \\
        \le {} & \gamma\ell\Delta m^{(k)}
        \P_{Z \sim \calN(0, \sigma_k^2)}\pra[\big]{\abs{Z} \ge \eps}  \le \gamma\ell\Delta m^{(k)} \exp\bra*{-\frac{\eps^2}{2\sigma_k^2}}                                                 \\
        ={}    & \gamma \ell \Delta m^{(k)} \exp\bra*{-\frac{\eps^2w''(0)m^{(k)} }{2\ell }\gamma},
    \end{align}
    which shows the claim.
\end{proof}

Next, we provide a bound on the error of the approximation~\eqref{eq:V-eff-approx} within the interval $(X^{(k)}-r_w\ell, X^{(k)}+r_w\ell)$. Although the error includes a quartic term in $x-X^{(k)}$ that grows with $\gamma$, this does not affect the application of Laplace's method later, as only the leading quadratic approximation by $f_k$ is relevant for the asymptotics.
\begin{lemma}
    \label{lemma:V-eff-parabolic-approximation-rigorous}
    Fix $k\in\{1,\dots, K\}$ and define the parabola $f_k(x):= -\gamma \ell \Delta m^{(k)} + \frac{(x - X^{(k)})^2}{2\sigma_k^2} + \frac{1}{2}$.
    Then, there exists a constant 
    \begin{align}
        C_{\rm par}=C_{\rm par}(\ell, w, X^{(1)}, \dots, X^{(K)}, m^{(1)}, \dots,m^{(K)})<\infty,
    \end{align} such that for all $\gamma >0$ and for all $x\in (X^{(k)}- r_w\ell, X^{(k)}+r_w\ell)$, we have that
    \begin{multline}
        \label{eq:V-eff-parabolic-approximation-rigorous}
        \abs*{ V_{\rm eff}(x) - f_k(x)}
        \le C_{\rm par} \gamma \exp\bra*{ - \frac{\eps^2 w''(0) m^{(k)}}{4 \ell } \gamma }
        \\ +
        C_{\rm par} \cdot L''' \cdot \bra*{\frac{1}{\gamma} + \gamma \bra*{x- x^{(k)}}^4},
    \end{multline}
    where $\eps := r_w\ell - \abs*{x-X^{(k)}}>0$. Here, $L'''<\infty$ denotes the Lipschitz constant of $w'''$ on $[-r_w, r_w]$.
\end{lemma}
\begin{proof}
    Assume that $\eps:=r_w\ell - \abs{x- X^{(k)}}>0$. We have that
    \begin{align}
        \calE^{(k)}(x) & := \abs*{ \bra*{\calW_{\gamma,\ell}\ast \rho_k}(x) -\bra*{-\gamma \ell \Delta m^{(k)} + \frac{(x - X^{(k)})^2}{2\sigma_k^2} + \frac{1}{2}}}                \\
                       & = \abs*{\int_\R \bra*{\calW_{\gamma,\ell}(y) + \gamma \ell \Delta - \frac{\bra*{x-X^{(k)}}^2}{2\sigma_k^2 m^{(k)}}-\frac{1}{2m^{(k)}}} \rho_k(x-y) \dx y}.
        \label{eq:V-eff-approx-error}
    \end{align}
    Since $w$ is $C^3$ on $[-r_w, r_w]$ and even, for all $y\in [-r_w\ell, r_w\ell]$, there exists $\xi_y \in [-\abs{y}, \abs{y}]$ such that
    \begin{align}
        \calW_{\gamma,\ell}(y) = - \gamma \ell \Delta + \frac{\gamma w''(0)}{2\ell} y^2 + \frac{\gamma}{6\ell^2}w'''\bra*{\frac{\xi_y}{\ell}} y^3.
    \end{align}
    In particular, we have
    \begin{align}
        \int_{-r_w\ell}^{r_w\ell} \calW_{\gamma,\ell}(y)\rho_k(x-y) \dx y =
        \int_{-r_w\ell}^{r_w\ell} \bra*{- \gamma \ell \Delta + \frac{\gamma w''(0)}{2\ell} y^2} \rho_k(x-y) \dx y + \calE_1,
        \label{eq:V-eff-approx-error-1}
    \end{align}
    where we introduced the error term
    \begin{align}
        \calE_1 & := \int_{-r_w\ell}^{r_w\ell} \frac{\gamma}{6\ell^2}w'''\bra*{\frac{\xi_y}{\ell}} y^3 \rho_k(x-y) \dx y.
    \end{align}
    Moreover, we have from \eqref{eq:V-eff-approx-error-1} that
    \begin{align}
         & \int_\R \calW_{\gamma,\ell}(y)\rho_k(x-y) \dx y
        =  \int_\R \bra*{- \gamma \ell \Delta + \frac{\gamma w''(0)}{2\ell}  y^2} \rho_k(x-y) \dx y + \calE_1 + \calE_2,
    \end{align}
    where
    \begin{align}
        \abs*{\calE_2} & \le \int_{\R\setminus [-r_w\ell, r_w\ell]} \bra*{\gamma\ell \Delta m^{(k)} + \frac{y^2}{2\sigma_k^2}} \frac{\rho_k(x-y)}{m^{(k)}} \dx y.
    \end{align}
    Together with \eqref{eq:V-eff-approx-error}, this gives
    \begin{align}
        \calE^{(k)}(x) & \le \abs*{\int_\R\bra*{
        \frac{y^2 - \bra*{x-X^{(k)}}^2}{2\sigma_k^2 } - \frac{1}{2}} \frac{\rho_k(x-y)}{m^{(k)}} \dx y
        } + \abs{\calE_1} + \abs*{\calE_2}.
        \label{eq:V-eff-approx-error-2}
    \end{align}
    Observe that the first term on the right-hand side of~\eqref{eq:V-eff-approx-error-2} is zero since $\rho_k(x-\dummy)/m^{(k)}$ is a Gaussian with mean $x-X^{(k)}$ and variance $\sigma_k^2 $.

    It remains to control the error terms $\calE_1$ and $\calE_2$.
    Let $L'''$ be the Lipschitz constant of $w'''$ on $[-r_w, r_w]$. From $w'''(0)=0$ we have then
    \begin{align}
        \abs{\calE_1} & \le
        \frac{\gamma L'''}{6\ell^3}
        \int_{-r_w\ell}^{r_w\ell} {y}^4 \rho_k(x-y) \dx y
        \le \frac{\gamma L'''}{6\ell^3}\int_\R y^4 \rho_k(x-y)\dx y                                                                                                   \\
                      & \le \frac{4 \gamma L''' m^{(k)}}{3\ell^3} \int_\R \pra*{\bra*{y - \bra*{x-X^{(k)}}}^4 + \bra*{x-X^{(k)}}^4} \frac{\rho_k(x-y)}{m^{(k)}} \dx y
        \\&= \frac{4 \gamma L''' m^{(k)}}{3\ell^3}\bra*{3\sigma_k^4 + \bra*{x-X^{(k)}}^4}
        \\&= \frac{4 L'''}{(w''(0))^2 m^{(k)}\ell}\frac{1}{\gamma} + \frac{4 L''' m^{(k)}\gamma}{3\ell^3} \bra*{x-X^{(k)}}^4.
        \label{eq:V-eff-approx-error-calE-1}
    \end{align}
    Moreover, we have that
    \begin{align}
        \abs*{\calE_2} & \le
        \expect_{Z\sim \calN(x-X^{(k)}, \sigma_k^2)}
        \pra*{\bra*{
            \gamma\ell\Delta m^{(k)} + \frac{Z^2}{2\sigma_k^2}
        } \mathbbm{1}_{\abs{Z}\ge r_w\ell}}
        \\& \le  \expect_{Z\sim \calN(x-X^{(k)}, \sigma_k^2)}
        \pra*{\bra*{
            \gamma\ell\Delta m^{(k)} + \frac{Z^2}{2\sigma_k^2}
        } \mathbbm{1}_{\abs{Z-(x-X^{(k)})}\ge \eps}}
        \\ &=
        \expect_{Z\sim \calN(0, \sigma_k^2)}
        \pra*{\bra*{
        \gamma\ell\Delta m^{(k)} + \frac{\bra*{x-X^{(k)}}^2 + Z^2}{2\sigma_k^2}
        } \mathbbm{1}_{\abs{Z}\ge \eps}}
    \end{align}
    Using \eqref{eq:gaussian-tail-bound} to bound
    \begin{align}
        \expect_{Z\sim\calN(0,\sigma_k^2)}\pra*{Z^2 \mathbbm{1}_{\abs{Z}\ge \eps}} & \le \bra*{\expect_{Z\sim\calN(0,\sigma_k^2)}\pra*{Z^4}}^{1/2} \P\pra*{\abs{Z}\ge \eps}^{1/2} = \sqrt{3}\sigma_k^2 \exp\bra*{-\frac{\eps^2}{4\sigma_k^2}},
    \end{align}
    we obtain
    \begin{align}
        \abs*{\calE_2} & \le \bra*{\gamma\ell \Delta m^{(k)} + \frac{\bra*{x-X^{(k)}}^2}{2\sigma_k^2}} \exp\bra*{-\frac{\eps^2}{2\sigma_k^2}} + \frac{\sqrt{3}}{2} \exp\bra*{-\frac{\eps^2}{4\sigma_k^2}}.
        \label{eq:V-eff-approx-error-calE-2}
    \end{align}
    Together,~\eqref{eq:V-eff-approx-error-2},~\eqref{eq:V-eff-approx-error-calE-1} and~\eqref{eq:V-eff-approx-error-calE-2} give
    \begin{multline}
        \abs*{ \bra*{\calW_{\gamma,\ell}\ast \rho_k}(x) -f_k(x)}
        \\\le C_{\rm par} \gamma \exp\bra*{ - \frac{\eps^2 w''(0) m^{(k)}}{4 \ell } \gamma }
        +       C_{\rm par} \cdot L''' \cdot \bra*{\frac{1}{\gamma} + \gamma \bra*{x- X^{(k)}}^4}
    \end{multline}
    for all $x\in (X^{(k)}-r_w\ell, X^{(k)}+r_w\ell)$. Here, the constant $C_{\rm par}<\infty$ depends on $\ell$, $w$ and on $m^{(1)},\dots,m^{(K)}$ as well as $X^{(1)}, \dots, X^{(K)}$, but not on $\gamma$. By applying~\eqref{eq:W-convolution-outside-valley-bound} to the other clusters $1,\dots, k-1, k+1,\dots,K$ and by enlarging the constant $C_{\rm par}<\infty$ (which still does not depend on $\gamma$), we obtain~\eqref{eq:V-eff-parabolic-approximation-rigorous}.
\end{proof}

The next lemma shows that $\calW_{\gamma,\ell}\ast \rho_k$ has a global minimum at $X^{(k)}$.
\begin{lemma}
    \label{lemma:W-convolution-monotonicity}
    For all $k\in\{1,\dots, K\}$, the function
    $\calW_{\gamma,\ell}\ast \rho_k:\R\to(-\infty,0]$ is strictly increasing for $x > X^{(k)}$ and strictly decreasing for $x < X^{(k)}$.
\end{lemma}
\begin{proof}
    For all $x\in\R$, we have that
    \begin{align}
             \bra*{\calW_{\gamma,\ell}\ast \rho_k}'(x)
        &= \int_{-s_w\ell}^{s_w\ell} \calW_{\gamma,\ell}(y) \rho_k'(x-y) \dx y                                              
        = \int_{-s_w\ell}^{s_w\ell} \calW_{\gamma,\ell}'(y) \rho_k(x-y) \dx y \\
          &= \int_{0}^{s_w \ell} \calW_{\gamma,\ell}'(y)  \rho_k(x-y) \dx y
        -\int_{0}^{s_w \ell} \calW_{\gamma,\ell}'(y) \rho_k(x+y) \dx y                                                                                    \\
         &= \int_{0}^{s_w \ell} \calW_{\gamma,\ell}'(y) \bra*{\rho_k(x-y)-\rho_k(x+y)} \dx y.
    \end{align}
    Note that $\calW_{\gamma,\ell}'(y) \ge 0$ for all $y\in[0,s_w\ell]$ and that $\rho_k(x-y)-\rho_k(x+y)$ is positive for $x > X^{(k)}$ and negative for $x < X^{(k)}$ which shows the claim.
\end{proof}

While \cref{lemma:V-eff-parabolic-approximation-rigorous} provides a good approximation near the minima of the effective potential, the next lemma establishes a coarser approximation which has the advantage that it holds globally for all $x \in \R$.

\begin{lemma}
    \label{lemma:V-eff-close-to-Wgl-sqrt-gamma}
    There exists a constant $C_{\rm coarse} = C_{\rm coarse}(\ell, w,K)<\infty$
    such that for all $x\in \R$, for all $\gamma>0$  and all $k\in\{1,\dots,K\}$ we have that
    \begin{align}
        \label{eq:V-eff-close-to-Wgl-sqrt-gamma-one-cluster}
        \abs*{\bra*{\calW_{\gamma,\ell}\ast \rho_k}(x) - m^{(k)} \calW_{\gamma,\ell}\bra*{x - X^{(k)}}} \le C_{\rm coarse} \sqrt{\gamma}
    \end{align}
    and
    \begin{align}
        \label{eq:V-eff-close-to-Wgl-sqrt-gamma}
        \abs*{V_{\rm eff}(x) - \sum_{k=1}^K m^{(k)} \calW_{\gamma,\ell}\bra*{x - X^{(k)}}} \le C_{\rm coarse} \sqrt{\gamma}.
    \end{align}
\end{lemma}

\begin{proof}
    Let $L<\infty$ be the Lipschitz constant of $w$. Then, we have that
    \begin{align}
              & \abs*{\bra*{\calW_{\gamma,\ell}\ast \rho_k}(x) - m^{(k)} \calW_{\gamma,\ell}\bra*{x - X^{(k)}}}
        \\={}& \abs*{\int_\R \bra*{\calW_{\gamma,\ell}(y) - \calW_{\gamma,\ell}\bra*{x - X^{(k)}}} \rho_k(x-y) \dx y}\\
        \le{} & L \gamma \int_\R \abs*{y - \bra*{x - X^{(k)}}} \rho_k(x-y) \dx y
        = L\gamma m^{(k)} \int_\R \abs{y} g(y; 0, \sigma^2_k) \d y
        \\={}&2L\sigma_k \gamma m^{(k)} \int_{[0,\infty)} y g(y; 0, 1) \d y = \sqrt{2 \pi^{-1}}L\sigma_k \gamma m^{(k)} =\frac{L \sqrt{2\ell}}{\sqrt{\pi w''(0) }} \sqrt{\gamma m^{(k)}},
    \end{align}
    where we used that $\sigma_k^2 = \frac{\ell}{w''(0)m^{(k)}\gamma}$. Using that $m^{(k)}< 1$, this proves \eqref{eq:V-eff-close-to-Wgl-sqrt-gamma-one-cluster}. By applying the triangle inequality and enlarging the constant $C_{\rm coarse}=C_{\rm coarse}(\ell, w, K)$ if necessary, we obtain~\eqref{eq:V-eff-close-to-Wgl-sqrt-gamma} as well.
\end{proof}

The following lemma bounds the first integral term on the right-hand side of~\eqref{eq:hitting-time-explicit}, showing that this term is of lower order than the right-hand side of~\eqref{eq:clusters-masses:Eyring-Kramers} as $\gamma \to\infty$.

\begin{lemma}
    \label{lemma:waiting-time-minor-term}
    There exist constants $c,C_{\rm minor}\in(0,\infty)$ depending on $\ell, w$, $X^{(1)}, \dots, X^{(K)}$ and $m^{(1)}, \dots, m^{(K)}$ such that for all $x\in [L_j, R_j]$ and large enough $\gamma$, we have the bound
    \begin{multline}
        \label{eq:lemma:waiting-time-minor-term}
        \int_a^x \int_a^y e^{V_{\rm eff}(y)- V_{\rm eff}(z)} \dx z \dx y
        \le C_{\rm minor} \exp\bra*{-c\sqrt{\gamma} + \gamma \ell \Delta m^{(j)}}
        \\ +C_{\rm minor}\exp\bra*{\gamma \ell \abs{w(\alpha)} m^{(j-1)} + 2C_{\rm coarse} \sqrt{\gamma}}.
    \end{multline}
\end{lemma}
Note that $\abs{w(\alpha)} m^{(j-1)} < \Delta m^{(j)}$ by the assumption~\eqref{eq:w-assumption:alpha-condition}. In particular, \cref{eq:lemma:waiting-time-minor-term} is of lower order than $\exp\bra*{\gamma \ell \Delta m^{(j)}}$ as $\gamma \to\infty$.
\begin{proof}
    Using that
    \begin{align}
        \set*{(y,z) \in \R^2 \st a \le z  \le y \le x} \subseteq \bra[\Big]{[a, R_j] \times [a, L_j]} \cup [L_j, R_j]^2,
    \end{align}
    we bound the integral as
    \begin{align}
         & \int_a^x \int_a^y e^{V_{\rm eff}(y)- V_{\rm eff}(z)} \dx z \dx y
        \\ \le{}& \int_a^{R_j} e^{V_{\rm eff}(y)}\dx y
        \int_a^{L_j} e^{- V_{\rm eff}(z)} \dx z
        + \int_{L_j}^{R_j} e^{V_{\rm eff}(y)}\dx y  \int_{L_j}^{R_j} e^{- V_{\rm eff}(z)} \dx z.
        \label{eq:waiting-time-minor-term}
    \end{align}
    To bound the first term on the right-hand side of \eqref{eq:waiting-time-minor-term}, we first use~\cref{lemma:W-convolution-monotonicity} to obtain for all $z\in [a, R_{j-1}]=[X^{(j-1)} + \alpha\ell, X^{(j-1)} + s_w\ell]$ that
\begin{align}
        \bra*{\calW_{\gamma,\ell}\ast \rho_{j-1}}(z) & \ge \bra*{\calW_{\gamma,\ell}\ast \rho_{j-1}}(X^{(j-1)} + \alpha\ell).
    \end{align}
    By~\cref{lemma:V-eff-close-to-Wgl-sqrt-gamma}, we have that
    \begin{align}
         \bra*{\calW_{\gamma,\ell}\ast \rho_{j-1}}(X^{(j-1)} + \alpha\ell) &\ge m^{(j-1)} \calW_{\gamma,\ell}(\alpha\ell) - C_{\rm coarse} \sqrt{\gamma}
         \\&= \gamma \ell w(\alpha) m^{(j-1)} - C_{\rm coarse} \sqrt{\gamma}.
    \end{align}
    Combining the last two inequalities, we obtain for all $z\in [a, R_{j-1}]$ that
    \begin{align}
        \int_a^{R_{j-1}} e^{-\bra*{\calW_{\gamma,\ell}\ast \rho_{j-1}}(z)} \dx z & \le s_w\ell \exp\bra*{\gamma \ell \abs{w(\alpha)} m^{(j-1)} + C_{\rm coarse} \sqrt{\gamma}}.
    \end{align}
    In order to show that a similar estimate holds with $\calW_{\gamma,\ell}$ replaced by $V_{\rm eff}$, we use that the sets $[L_k, R_k]$ are disjoint for $k=1,\dots, K$ by assumption. Thus, by~\cref{lemma:W-convolution-outside-valley-bound}, we have for all $z\in [a, R_{j-1}]$ and all $k\neq j-1$ that
    \begin{align}
        \abs*{\bra*{\calW_{\gamma,\ell}\ast \rho_k}(z)}
        & \le \gamma \ell \Delta \exp\bra*{ - \frac{\tilde{\eps}^2 w''(0) m_{\rm min}}{2\ell}} \gamma
    \end{align}
    where $\tilde{\eps} := \inf\set*{ \abs{x-X^{(k)}} - s_w\ell \st x \in [a, R^{j-1}], k\in\{1,\dots, K\}\setminus\{j-1\}}>0$ and $m_{\rm min} := \min_{k=1,\dots, K} m^{(k)}$.
    In particular, $V_{\rm eff} - \calW_{\gamma,\ell}\ast \rho_{j-1}$ converges to zero uniformly on $[a, R_{j-1}]$ as $\gamma \to\infty$ and therefore we obtain for large enough $\gamma$ that
    \begin{align}
        \label{eq:a-R-j-1-integral-bound}
        \int_a^{R_{j-1}} e^{-V_{\rm eff}(z)} \dx z & \le 2s_w\ell \exp\bra*{\gamma \ell \abs{w(\alpha)} m^{(j-1)} + C_{\rm coarse} \sqrt{\gamma}}.
    \end{align}
    Observe that
    \begin{align}
        \sum_{k=1}^K m^{(k)} \calW_{\gamma,\ell}\bra*{x - X^{(k)}} = 0 \quad \text{for all } x \in [R_{j-1}, L_j]
    \end{align}
    since $\supp \calW_{\gamma,\ell} = [-s_w\ell, s_w\ell]$, so \eqref{eq:V-eff-close-to-Wgl-sqrt-gamma} implies that $V_{\rm eff}(z) \ge - C_{\rm coarse} \sqrt{\gamma}$ for all $z\in[R_{j-1}, L_j]$.
    Together with \eqref{eq:a-R-j-1-integral-bound}, this gives for large enough $\gamma$ that
    \begin{multline}
        \int_a^{R_j} e^{V_{\rm eff}(y)}\dx y
        \int_a^{L_j} e^{- V_{\rm eff}(z)} \dx z \le (R_j -a)\bra*{\int_a^{R_{j-1}} e^{-V_{\rm eff}(z)} \dx z + \int_{R_{j-1}}^{L_j} e^{-V_{\rm eff}(z)} \dx z}
        \\\le \bra*{X^{(j)}-X^{(j-1)}+s_w\ell}^2 \exp\bra*{\gamma \ell \abs{w(\alpha)} m^{(j-1)} +C_{\rm coarse} \sqrt{\gamma}},
        \label{eq:a-R-j-integral-bound}
    \end{multline}
    where we used that $V_{\rm eff}\le 0 $ and that $R_j-a \le X^{(j)} - X^{(j-1)} + s_w\ell$ as well as $2s_w \ell + (L_j - R_{j-1}) \le X^{(j)} - X^{(j-1)}$.

    As for the second term in~\eqref{eq:waiting-time-minor-term}, the idea is that even though $\exp(-V_{\rm eff}(z))$ will give an exponential contribution of the order $\exp(\gamma \ell \Delta m^{(j)})$, the term $\exp(V_{\rm eff}(y))$ will super-compensate this exponential contribution as we will show in the following. From \eqref{eq:w-assumption:linear-upper-bound-around-sw} we obtain $\calW_{\gamma,\ell}(z) \le \gamma \tau\cdot (z- s_w \ell)$ for all $z \in [\delta_w\ell,s_w\ell]$.
    Therefore, we have
    \begin{align}
         & \bra*{\calW_{\gamma,\ell}\ast \rho_j}\bra*{X^{(j)}+s_w\ell}
        \le \int_{\delta_w\ell}^{s_w\ell} \calW_{\gamma,\ell}(z)
        \rho_j\bra*{X^{(j)}+s_w\ell -z} \dx z
        \\  \le{}& \gamma \tau m^{(j)}\int_{\delta_w\ell}^{s_w\ell}  (z- s_w \ell) g(s_w\ell -z; 0, \sigma_j^2) \dx z
        \\&= - \frac{\gamma \tau m^{(j)}}{2}\int_{-(s_w-\delta)\ell}^{(s_w-\delta)\ell}  \abs*{z- s_w \ell} g(s_w\ell -z; 0, \sigma_j^2) \dx z
    \end{align}
    For large $\gamma$, replacing the integration domain $[-(s_w-\delta)\ell, (s_w-\delta)\ell]$ by $\R$ will only change the integral by a term of order $\exp(-c\gamma)$ for some constant $c>0$. Therefore, since $\sigma_j^2=\ell/(\gamma w''(0) m^{(j)})$, there exists some constants $c>0$ (depending on $\ell, w, m^{(j)}$) such that for $\gamma$ large enough, we have
    \begin{align}
        \bra*{\calW_{\gamma,\ell}\ast \rho_j}\bra*{X^{(j)}+s_w\ell} \le -c \sqrt{\gamma}.
    \end{align}
    This shows that even though the effective potential will eventually converge to zero on $(R_j, L_{j+1})$ as $\gamma \to\infty$ by~\cref{eq:W-convolution-outside-valley-bound}, the error near $R_j$ will be of order $\sqrt{\gamma}$, see also~\cref{lemma:V-eff-close-to-Wgl-sqrt-gamma}.
    Similarly, we obtain that
    \begin{align}
        \bra*{\calW_{\gamma,\ell}\ast \rho_j}\bra*{X^{(j)}-s_w\ell} \le -c \sqrt{\gamma},
    \end{align}
    soby~\cref{lemma:W-convolution-monotonicity}, we therefore obtain
    \begin{align}
        \label{eq:Veff-slow-convergence-at-boundaries}
        \calW_{\gamma,\ell}\ast \rho_j(y) \le - c \sqrt{\gamma} \quad \text{for all } y \in [L_j, R_j].
    \end{align}
    Using that $\calW_{\gamma,\ell}\ge - \gamma\ell \Delta$, we can use~\eqref{eq:Veff-slow-convergence-at-boundaries} to obtain
    \begin{align}
        \int_{L_j}^{R_j}e^{\bra*{\calW_{\gamma,\ell}\ast \rho_j}(y)}\dx y  \int_{L_j}^{R_j} e^{-\bra*{\calW_{\gamma,\ell}\ast \rho_j}(z)} \dx z
        \le 2s_w\ell e^{-c\sqrt{\gamma} + \gamma \ell \Delta m^{(j)}}.
    \end{align}
    Using~\cref{lemma:W-convolution-outside-valley-bound}, we obtain therefore
    \begin{align}
        \int_{L_j}^{R_j} e^{V_{\rm eff}(y)}\dx y  \int_{L_j}^{R_j} e^{- V_{\rm eff}(z)} \dx z
        \le 4s_w\ell e^{-c\sqrt{\gamma} + \gamma \ell \Delta m^{(j)}} \quad \text{for $\gamma$ large enough}.
    \end{align}
    Together with \eqref{eq:a-R-j-integral-bound}, this completes the proof.
\end{proof}

We now have everything at hand to prove~\cref{prop:prop-laplace-hitting-time}.
\begin{proof}[Proof of~\cref{prop:prop-laplace-hitting-time}]
    Let's first prove~\eqref{eq:clusters-masses:equilibrium-potential}. By~\eqref{eq:clusters-masses:equilibrium-potential-solution}, the equilibrium potential $h_{a,b}$ is given by
    \begin{align}
        h_{a,b}(x) & = \frac{
            \int_x^b e^{V_{\rm eff}(y)} \dx y
        }
        {
            \int_a^b e^{V_{\rm eff}(y)} \dx y
        }.
    \end{align}
    It follows from~\eqref{eq:W-convolution-outside-valley-bound} that
    \begin{align}
        \lim_{\gamma\to\infty} e^{V_{\rm eff}(x)} =1 \quad \text{for all } x\in \R\setminus \bra*{\bigcup_{k=1}^K [L_k, R_k]}.
    \end{align}
    Moreover, we have that
    \begin{align}
        \lim_{\gamma\to\infty} e^{V_{\rm eff}(x)} = 0 \quad \text{for all } x \in \bigcup_{k=1}^K (L_k, R_k)
    \end{align}
    since $\abs*{V_{\rm eff}(x) - \gamma m^{(k)} \ell w\bra*{\ell^{-1}\bra*{x-X^{(k)}}}} \le C_{\rm coarse} \sqrt{\gamma}$ for all $x\in (L_k, R_k)$ by~\cref{lemma:V-eff-close-to-Wgl-sqrt-gamma} and because $w(y)<0$ on $(-s_w, s_w)$.
    Since $V_{\rm eff} \le 0$ on $\R$, the dominated convergence theorem implies that
    \begin{align}
        \lim_{\gamma\to\infty} h_{a,b}(x) = \frac{L_{j+1}- R_{j}}{L_{j+1} - R_{j-1}} = \frac{X^{(j+1)} - X^{(j)} - 2s_w\ell}{X^{(j+1)} - X^{(j-1)} - 4s_w\ell}
    \end{align}
    for all $x\in (L_j, R_j)$, which shows~\eqref{eq:clusters-masses:equilibrium-potential}.

    Next, we want to determine the asymptotic behaviour of $\condexpect*{\tau_{a,b} \given X^i_0 = x}$ which by~\eqref{eq:hitting-time-explicit} can be expressed as
    \begin{multline}
        \label{eq:waiting-time-solution}
        \condexpect*{\tau_{a,b} \given X^i_0 = x_0}  =
        - \int_a^{x_0} \int_a^y e^{V_{\rm eff}(y)- V_{\rm eff}(z)} \dx z \dx y  \\
        \quad+ \frac{
            \int_a^{x_0} e^{V_{\rm eff}(y)} \dx y
        }{
            \int_a^b e^{V_{\rm eff}(y)} \dx y
        }\cdot
        \int_a^b \int_a^y e^{V_{\rm eff}(y)- V_{\rm eff}(z)} \dx z \dx y
    \end{multline}
    Similar to the above, we have that
    \begin{align}
        \lim_{\gamma\to\infty} \frac{\int_a^{x_0} \exp\bra*{V_{\rm eff}(y)}\dx y}{\int_a^b \exp\bra*{V_{\rm eff}(y)} \dx y} = \frac{X^{(j)}-X^{(j-1)}-2s_w \ell}{X^{(j+1)} - X^{(j-1)} - 4s_w\ell} \quad \text{for all } x_0 \in (L_j, R_j),
    \end{align}
    which determines the asymptotic behaviour of the prefactor of the second term in~\eqref{eq:waiting-time-solution}. Moreover, we have that
    \begin{multline}
        \label{eq:waiting-time-second-term-lower-order-term}
        \int_a^b \int_a^y e^{V_{\rm eff}(y)- V_{\rm eff}(z)} \dx z \dx y
        \\ =  \int_a^{L_j} \int_a^y e^{V_{\rm eff}(y)- V_{\rm eff}(z)} \dx z \dx y
        + \int_{L_j}^{b} \int_a^y e^{V_{\rm eff}(y)- V_{\rm eff}(z)} \dx z \dx y.
    \end{multline}
    We claim that the second term has the asymptotic behaviour
    \begin{multline}
        \label{eq:waiting-times-final-claim}
        \int_{L_j}^{b} \int_a^y e^{V_{\rm eff}(y)- V_{\rm eff}(z)} \dx z \dx y\\ \sim
        \bra*{X^{(j+1)} - X^{(j)} - 2s_w\ell}
        \sqrt{\frac{2\pi \ell}{\gamma w''(0) m^{(j)}}} e^{\gamma\ell \Delta m^{(j)} - \frac{1}{2}}.
    \end{multline}
    By~\cref{lemma:waiting-time-minor-term},
    the first term on the right-hand side of~\eqref{eq:waiting-time-solution} as well as the first term on the right-hand side of~\eqref{eq:waiting-time-second-term-lower-order-term} is of lower order
    than the right-hand side of~\eqref{eq:waiting-times-final-claim}
    as $\gamma \to\infty$.
    Therefore, the proof of~\cref{prop:prop-laplace-hitting-time} is complete once we have shown~\eqref{eq:waiting-times-final-claim}.

    To show~\eqref{eq:waiting-times-final-claim}, first note that
    \begin{align}
        [R_j,L_{j+1}] \times [L_j, R_j] \subset
        \set*{(y,z) \in \R^2 \st  L_j \le y \le b, ~ a \le z \le y}
        \subset [L_j,b]\times [a,b].
    \end{align}
    In particular, we have that
    \begin{multline}
        \label{eq:waiting-time-dominant-term-lower-upper-bound}
        \int_{R_j}^{L_{j+1}} e^{V_{\rm eff}(y)} \dx y \int_{L_j}^{R_j} e^{-V_{\rm eff}(z)} \dx z
        \\
        \le \int_{L_j}^{b} \int_a^y e^{V_{\rm eff}(y)- V_{\rm eff}(z)} \dx z \dx y
        \le \int_{L_j}^{b} e^{V_{\rm eff}(y)} \dx y \int_a^b e^{-V_{\rm eff}(z)} \dx z.
    \end{multline}
    We want to show that the upper and lower bounds in~\eqref{eq:waiting-time-dominant-term-lower-upper-bound} are asymptotically equivalent as $\gamma \to\infty$, which will then show~\eqref{eq:waiting-times-final-claim}. We first deal with the lower bound in \eqref{eq:waiting-time-dominant-term-lower-upper-bound}.
    As before, it follows from the dominated convergence theorem that
    \begin{align}
        \lim_{\gamma\to\infty}   \int_{R_j}^{L_{j+1}} e^{V_{\rm eff}(y)} \dx y = \lim_{\gamma\to\infty}\int_{L_j}^{b} e^{V_{\rm eff}(y)} \dx y = L_{j+1} - R_j = X^{(j+1)} - X^{(j)} - 2s_w\ell.
    \end{align}
    In order to estimate $\int_{L_j}^{R_j} e^{-V_{\rm eff}(z)} \dx z$, we use the approximation of $V_{\rm eff}$ around its minimum near $X^{(j)}$ from~\cref{lemma:V-eff-parabolic-approximation-rigorous}. Define
    \begin{align}
        f^{\pm}_j(x):= f_j(x) \pm C_{\rm par} \cdot L''' \cdot \gamma \cdot \bra*{x-X^{(j)}}^4
    \end{align}
    and choose $\beta\in(0,r_w)$ so small that $f^-_j$ has its unique minimum on $[X^{(j)}- \beta\ell, X^{(j)}+\beta\ell]$ at $X^{(j)}$. This works since  $(f^-_j)'(X^{(j)})=0$ and $(f^-_j)''(X^{(j)})>0$.
    Then,
    \cref{lemma:V-eff-parabolic-approximation-rigorous} implies for all $z\in [X^{(j)}- \beta\ell, X^{(j)}+ \beta\ell]$ that
    \begin{align}
        f^-_j(z) - \calE & \le V_{\rm eff}(z) \le f^+_j(z) + \calE \quad \text{where } \calE := C_{\rm par} \gamma \exp\bra*{-\frac{w''(0) m^{(j)} \beta^2}{4} \gamma \ell}.
    \end{align}
    Therefore we obtain
    \begin{align}
        \label{eq:waiting-time-dominant-term-liminf-limsup-bound}
        \limsup_{\gamma \to\infty} \frac{\int_{X^{(j)}- \beta \ell}^{X^{(j)}+ \beta\ell} e^{-f^+_j(z)} \dx z}{\int_{X^{(j)}- \beta \ell}^{X^{(j)}+ \beta\ell} e^{-V_{\rm eff}(z)} \dx z} \le 1 \le \liminf_{\gamma \to\infty} \frac{\int_{X^{(j)}- \beta \ell}^{X^{(j)}+ \beta\ell} e^{-f^-_j(z)} \dx z}{\int_{X^{(j)}- \beta \ell}^{X^{(j)}+ \beta\ell} e^{-V_{\rm eff}(z)} \dx z}.
    \end{align}
    By Laplace's method, we have the asymptotic equivalences
    \begin{align}
        \label{eq:laplace-dominating-term}
        \int_{X^{(j)}- s_w\ell/2}^{X^{(j)}+ s_w\ell/2} e^{-f_j^\pm(z)} \dx z
        \sim \sqrt{2\pi \sigma_j^2} e^{\gamma\ell \Delta m^{(j)} - \frac{1}{2}} = \sqrt{\frac{2\pi \ell}{\gamma w''(0) m^{(j)}}} e^{\gamma\ell \Delta m^{(j)} - \frac{1}{2}}
    \end{align}
    for both $f_j^+$ and $f_j^-$,
    as $\gamma \to\infty$.
    In particular, we obtain from~\eqref{eq:waiting-time-dominant-term-liminf-limsup-bound} that
    \begin{align}
        \lim_{\gamma\to\infty} \frac{\int_{X^{(j)}- s_w\ell/2}^{X^{(j)}+ s_w\ell/2} e^{-V_{\rm eff}(z)} \dx z}{\sqrt{\frac{2\pi \ell}{\gamma w''(0) m^{(j)}}} e^{\gamma\ell \Delta m^{(j)} - \frac{1}{2}}} = 1.
    \end{align}
    Now splitting the integral $\int_{L_j}^{R_j} e^{-V_{\rm eff}(z)} \dx z$ as    
    \begin{align}
        \int_{L_j}^{R_j} e^{-V_{\rm eff}(z)} \dx z = \bra*{\int_{X^{(j)}- s_w\ell}^{X^{(j)}- \beta\ell/2}  + \int_{X^{(j)}- \beta\ell/2}^{X^{(j)}+ \beta\ell/2} + \int_{X^{(j)}+ \beta\ell/2}^{X^{(j)}+ s_w\ell} }e^{-V_{\rm eff}(z)} \dx z,
        \label{eq:Veeff-Lj-Rj-3-split}
    \end{align}
    it follows from~\cref{lemma:V-eff-close-to-Wgl-sqrt-gamma} in the same way as in \cref{lemma:waiting-time-minor-term} that the first and the third integral term on the right-hand side of~\eqref{eq:Veeff-Lj-Rj-3-split} are of lower order than the middle term, which satisfies~\eqref{eq:laplace-dominating-term} as $\gamma \to\infty$.

    This shows that the lower bound in~\eqref{eq:waiting-time-dominant-term-lower-upper-bound} has the asymptotic behaviour as claimed in~\eqref{eq:waiting-times-final-claim}. The upper bound in~\eqref{eq:waiting-time-dominant-term-lower-upper-bound} can be shown to have the same asymptotic behaviour by a similar argument, based on the fact that the additional domain integrated over is of lower order by~\cref{lemma:W-convolution-outside-valley-bound} and~\cref{lemma:waiting-time-minor-term}. This completes the proof.
\end{proof}

\printbibliography
\end{document}